\documentclass{amsart}
\usepackage{amssymb}
\usepackage{amsmath}    
\usepackage{graphicx}   
\usepackage{verbatim}   
\usepackage{color}      
\usepackage{subfigure}  
\usepackage{hyperref}   
\usepackage[all,cmtip]{xy}
\usepackage[mathscr]{eucal}
\usepackage[lite,abbrev,msc-links,numeric]{amsrefs}
\newcommand{\Z}{\mathbb{Z}}
\newcommand{\N}{\mathbb{N}}
\newcommand{\Q}{\mathbb{Q}}
\newcommand{\R}{\mathbb{R}}
\newcommand{\C}{\mathbb{C}}
\newcommand{\K}{\mathbb{K}}
\newcommand{\hpl}{HP^*_{S^1,\mathrm{loc}}}
\newcommand{\hp}{HP^*_{S^1}}
\numberwithin{equation}{section}
\newtheorem{thm}{Theorem}[section]

\newtheorem{prop}[thm]{Proposition}
\newtheorem{lem}[thm]{Lemma}
\newtheorem{cor}[thm]{Corollary}

\theoremstyle{definition}
\newtheorem{definition}[thm]{Definition}

\theoremstyle{remark}
\newtheorem{rmk}[thm]{Remark}

\def\bysame{\leavevmode\hbox to3em{\hrulefill}\thinspace}
\usepackage{hyperref}
\newcommand\scalemath[2]{\scalebox{#1}{\mbox{\ensuremath{\displaystyle #2}}}}
\DeclareMathOperator{\Det}{Det}
\DeclareMathOperator{\coker}{coker}
\DeclareMathOperator{\Crit}{Crit}
\DeclareMathOperator{\Ind}{Ind}
\DeclareMathOperator{\Int}{Int}
\DeclareMathOperator{\Span}{Span}
\DeclareMathOperator{\sign}{sign}
\DeclareMathOperator{\Tr}{Tr}
\DeclareMathOperator{\Diag}{Diag}
\DeclareMathOperator{\Sp}{\mathrm{Sp}}

\DeclareFontFamily{U}{mathx}{\hyphenchar\font45}
\DeclareFontShape{U}{mathx}{m}{n}{
      <5> <6> <7> <8> <9> <10>
      <10.95> <12> <14.4> <17.28> <20.74> <24.88>
      mathx10
      }{}
\DeclareSymbolFont{mathx}{U}{mathx}{m}{n}
\DeclareFontSubstitution{U}{mathx}{m}{n}
\DeclareMathAccent{\widecheck}{0}{mathx}{"71}
\DeclareMathAccent{\wideparen}{0}{mathx}{"75}

\begin{document}

\title{Periodic Symplectic Cohomologies}

\author[Zhao]{Jingyu Zhao}
\address{Center of Mathematical Sciences and Applications\\Harvard University}
\email{jzhao@cmsa.fas.harvard.edu, jzhao0105@gmail.com}

\date{}

\begin{abstract}
Goodwillie \cite{Goodwillie} introduced a periodic cyclic homology group associated to a mixed complex. In this paper, we apply this construction to the symplectic cochain complex of a Liouville domain $M$ and obtain two periodic symplectic cohomology theories, denoted as $\hp(M)$ and $\hpl(M)$. Our main result is that both cohomology theories are invariant under Liouville isomorphisms and there is a natural isomorphism $\hpl(M, \Q) \cong H^*(M, \Q)((u))$, which can be seen as a localization theorem for $\hpl(M, \Q)$.
\end{abstract}

\maketitle

\section{Introduction}
  Symplectic homology of an exact symplectic manifold with boundary was introduced by Cieliebak--Floer--Hofer--Wysocki in a sequence of papers \cite{CFH, CFHW, FH2, FHW}. There are several versions of the theory defined for slightly different notions of symplectic manifolds with boundary. In this paper, we will consider a specific class, called Liouville domains, which are exact symplectic manifolds with contact type boundary. The completion $\widehat{M}$ of a Liouville domain $(M, \omega=d\theta)$ is defined to be $(M \cup_{\partial M}[1, \infty) \times \partial M, d(R\alpha))$ with $\alpha= \theta|_{\partial M}$. Symplectic homology theory can be viewed as an infinite-dimensional Morse theory, or Floer theory, for an action functional $\mathscr{A} \colon \mathscr{L}\widehat{M} \rightarrow \R$. There is a natural $S^1$-action on the free loop space $\mathscr{L}\widehat{M}$ by reparametrizations which makes $\mathscr{L}\widehat{M}$ into a topological $S^1$-space. This suggests that one can define $S^1$-equivariant symplectic homology using the Borel construction of $X:=\mathscr{L}\widehat{M}$. There are two equivalent definitions of $S^1$-equivariant symplectic homology. In \cite[Section 5]{Viterbo1}, Morse theory on $X_{\mathrm{Borel}}:=X \times_{S^1} ES^1$ is defined by Viterbo as Morse theory on $X \times ES^1$ modulo a free $S^1$-action, that is, Morse theory for an $S^1$-invariant action functional on $X \times ES^1$. Whereas in \cite{seidel_biased}, using the fact that $X_{\mathrm{Borel}}$ is a locally trivial fibration over $\C P^{\infty}$, Morse theory on $X_{\mathrm{Borel}}$ can be viewed as the family Morse homology associated to a family of Morse functions parametrized by $\C P^{\infty}$ in the sense of \cite{Hut}. Both of these formulations are made explicit in the work of Bourgeois and Oancea in \cite{BO2}. \\
\indent If one takes the Liouville domain $M$ to be the unit cotangent bundle of some closed smooth manifold $Q$, then the Liouville completion $\widehat{M}$ is $T^*Q$. Under the assumption that $Q$ is Spin, the work of \cite{Viterbo1, AbbS, abu_viterbo} shows that Viterbo's isomorphism gives rise to an isomorphism of BV algebras
  \begin{equation}\label{eqn:Viso}
  SH^{n-*}(T^*Q, \Z) \cong H_*(\mathscr{L}Q, \Z),
  \end{equation}
 where $H_*(\mathscr{L}Q, \Z)$ is the homology of the free loop space equipped with the Chas-Sullivan product \cite{CS} and the BV operator. There is an $S^1$-action on $\mathscr{L}Q$ which is homotopy equivalent to the $S^1$-space $\mathscr{L}T^*Q$. Given an $S^1$-space, one can define three versions of $S^1$-equivariant homology theories, which are called positive, negative and periodic $S^1$-equivariant homology respectively. In the situation that the $S^1$-space is the free loop space $\mathscr{L}Q$, there are isomorphisms between positive, negative and periodic $S^1$-equivariant homology of $\mathscr{L}Q$ and positive, negative and periodic cyclic homology of the singular cochains on the based loop space of $Q$ respectively. This is proved by Goodwillie in \cite{Goodwillie} also see \cite{Jones, Loday}. \\
\indent Motivated by Viterbo's isomorphism \eqref{eqn:Viso} and the definition of the periodic cyclic homology of the cochains on the based loop space, we apply periodic cyclic homology to the symplectic cochain complex on the symplectic side. One feature of the symplectic cochain complex shown in \cite{BO2} is that it is an $S^1$-complex, or an $\infty$-mixed complex which was introduced in \cite{kassel}. This means that it satisfies the definition of a mixed complex up to higher homotopies. We call the periodic cyclic homology of the symplectic cochain complex \emph{the periodic symplectic cohomology} and denote it by $\hp(M)$.\\
\indent There is another version of periodic symplectic cohomology that one can associate to a Liouville domain $M$. This version is inspired by the work of Jones and Petrack in \cite{JP} on the periodic equivariant cohomology of the free loop space. It is proved by Goodwillie in \cite{Goodwillie} that the usual periodic $S^1$-equivariant cohomology does not satisfy localization or the fixed point theorem when the $S^1$-space is infinite-dimensional as in the case of the free loop space. To deal with this problem, Jones and Petrack defined a variant of the usual periodic $S^1$-equivariant cohomology, which satisfies localization but not the weak homotopy invariance property in \cite[Sections 1, 2]{JP}. The application of this equivariant cohomology theory to on the symplectic cochain complex was first proposed by P. Seidel in \cite[Remark 8.1]{seidel_biased}. Our main result can be now viewed as a symplectic cohomology analogue of the localization theorem for the periodic $S^1$-equivariant cohomology of Jones and Petrack. We call the corresponding cohomology theory \textit{the localized periodic symplectic cohomology} of the Liouville domain $M$ and denote it as $\hpl(M)$. 
\begin{thm} \label{thm:1.1}
Given a Liouville domain $(M, \theta)$, the natural inclusion of the constant loops $\iota\colon \widehat{M} \hookrightarrow \mathscr{L}\widehat{M}$ induces a map 
\begin{equation}
\iota_*\colon H^*(M)((u)) \rightarrow \hpl(M), \nonumber
\end{equation}
which is an isomorphism as $\Z/2$-graded $\Q((u))$-modules after tensoring both sides by $\Q$.
\end{thm}

\begin{rmk}
Albers, Cieliebak and Frauenfelder in \cite{ACF} have independently discovered a similar construction under the assumption that $2c_1(M)=0$ for Rabinowitz Floer homology, and named it \textit{symplectic Tate homology}. When $2c_1(M)=0$, we can equip our periodic symplectic cohomology groups $\hp(M)$ and $\hpl(M)$ with extra $\Z$-grading, and the isomorphism $H^*(M,\Q)((u)) \rightarrow \hpl(M,\Q)$ is in fact an isomorphism of $\Z$-graded $\Q((u))$-modules. 
\end{rmk}
\indent We end the introduction with an outline of this paper. Section \ref{sec:cyc} contains a discussion of the algebraic formalism of $S^1$-complexes following \cite{BO2} and the definition of the cyclic homology of an $S^1$-complex. In section \ref{sec:symcoho}, symplectic cohomology $SH^*(M)$ of a Liouville domain $M$ is defined to be the direct limit of Floer cohomology groups of admissible Hamiltonians that are linear on $[R_0, \infty) \times \partial M$ for $R_0 \gg 1$. In section \ref{sec:psh}, we introduce the localized periodic symplectic cohomology $\hpl(M)$. The Floer cochain group $CF^*(M, H_t)$ of an admissible Hamiltonian $H_t$ can be equipped with an $S^1$-complex structure $\{\delta_i\}_{i \geq 0}$, where $\delta_0=d$ is the usual Floer differential and $\delta_1=\Delta$ is the BV operator. This implies that
\begin{equation}
CF^*(M, H_t)((u)), \ \ \delta^{S^1}=\delta_0+u\delta_1+u^2\delta_2+ \hdots \nonumber
\end{equation}
is a cochain complex, where $CF^*(M, H_t)((u))$ denotes  the formal Laurent series with coefficients in $CF^*(M, H_t)$. Let $\hp(M, H_t)$ denote the homology of this cochain complex. Then $\hpl(M)$ is defined to be $\varinjlim \hp(M, H_t)$, where the direct limit is taken over all admissible Hamiltonians $H_t$ with respect to appropriate continuation maps. By definition $\hpl(M)$ is invariant under Liouville isomorphism. In practice, to compute $\hp(M)$ and $\hpl(M)$ one can use a specific class of admissible Hamiltonians, described in section \ref{sec:concx}, which are autonomous away from its $1$-periodic orbits, grow quadratically on $[1, R_0-\epsilon_0) \times \partial M$ and grow linearly of slope $\tau$ on $[R_0 ,\infty) \times \partial M$ for some large enough constant $R_0$ and sufficiently small $\epsilon_0>0$. Such autonomous Hamiltonians $H^{\tau}$ are convenient to work with, since non-constant $1$-periodic orbits of $H^{\tau}$ are in bijection with Reeb orbits of period less than $\tau$ on the contact manifold $\partial M$. Furthermore, if we define the Floer cochain complex using a carefully chosen perturbation $H^{\tau}_t$ of the autonomous Hamiltonian $H^{\tau}$, then there is in fact a filtered $S^1$- structure on $CF^*(M, H^{\tau}_t)$. One can then deduce the main theorem by computing the local $S^1$-equivariant Floer (co)homology contributions to the equivariant differential $\delta^{S^1}=d+u\Delta$ between the $1$-periodic orbits $\widehat{\gamma}$ and $\widecheck{\gamma}$ of $H^{\tau}_t$ associated to a Reeb orbit $\gamma$ of multiplicity $k$ on $\partial M$. A detailed computation in Proposition \ref{prop:1} shows that $d(\widecheck{\gamma})=\pm 2\widehat{\gamma}$ and $\Delta(\widehat{\gamma})=0$ if $\gamma$ corresponds to a bad Reeb orbit, and $d(\widecheck{\gamma})=0$ and $\Delta(\widehat{\gamma})=\pm k\widehat{\gamma}$ if $\gamma$ corresponds to a good Reeb orbit. This implies that the associated spectral sequence of the filtration on $CF^*(M, H^{\tau}_t)((u))$ degenerates at $E_1$-page over $\Q$-coefficients and converges to $H^*(M, \Q)((u))$ for each Hamiltonian $H^{\tau}_t$ of slope $\tau$. After taking direct limit, one obtains the result of Theorem \ref{thm:1.1}. In section \ref{sec:cpsh}, we define the periodic symplectic cohomology $\hp(M)$ using the homotopy colimit construction introduced in \cite[Section 3g]{AS} and prove that $\hp(M)$ is an invariant of $\widehat{M}$ up to Liouville isomorphism. Finally, we include explicit computations of $\hpl(M)$ and $\hp(M)$ when $M$ is the unit disk $D^2$ and the annulus $A$ in $\C$ in section \ref{sec:comp}. The localized periodic cohomology $\hpl(D^2)$ of the unit disk $D^2$ is $\Q((u))$, which agrees with the main result. Whereas the periodic symplectic cohomology $\hp(D^2)$ vanishes. This shows that $\hp(M)$ does not satisfy localization as in the case of the usual periodic $S^1$-equivariant cohomology of the free loop space. 

\subsection*{Standing Notations}
 \ \ \\
$\bullet$ (Grading) We will assume that the cochain complexes are only $\Z/2$-graded. In the case that the symplectic manifold $(M, \omega)$ satisfies the condition $2c_1(M)=0$, then one can promote the $\Z/2$-grading to a $\Z$-grading of the cochain complexes. Nevertheless, we write all of our operations and indices with respect to some $\Z$-grading. When only a $\Z/2$-grading is given, one should interpret an operation of degree $k$ as one of degree $k \text{ mod } 2$, and similarly interpret the indices of $1$-periodic Hamiltonian orbits $|\gamma|=k$ and Morse critical points $\mathrm{Ind}_M(Z)=k$ as $|\gamma|=k \text{ mod } 2$ and $\mathrm{Ind}_M(Z)=k \text{ mod } 2$ respectively.\\
$\bullet$ ($u$-adic completion) Let $u$ be a formal variable of degree $2$, or degree $0$ if only a $\Z/2$-grading is present. Let $\K$ be a fixed commutative ring with unity, for instance, $\K$ can be taken to be $\Z$ or $\Q$. We denote by $\K[[u]]$ the ring of \textit{formal} power series in $u$ and $\K((u))$ the ring of \textit{formal} Laurent series in $u$. Given a \textit{graded} $\K$-module $M$, we will denote by $M[[u]]$ the \textit{completion} of $M[u]:=M\otimes_{\K}\K[u]$ in the category of graded $\K[u]$-modules with respect to the $u$-adic filtration $F^pM[u]:=u^p M[u]$ for $p \in \N$. Elements in $M[[u]]:=\varprojlim M[u]/ u^p M[u]$ consist of formal power series in $u$ with coefficients in $M$. Similarly, the notation $M((u))$ denotes the completion of $M[u, u^{-1}]:=M\otimes_{\K}\K[u, u^{-1}]$ with respect to the $u$-adic filtration defined by $F^pM[u, u^{-1}]:=u^p M[u, u^{-1}]$. For a cochain complex $C^*$, we define $$C^*[[u]]:=\varprojlim C^*\otimes_{\K}\K[u]/ F^p$$ to be the completion of the graded cochain complexes $C^*[u]:=C^*\otimes_{\K}\K[u]$ with respect to the $u$-adic filtration $F^p:=C^*\otimes_{\K}u^p\K[u]$. Similarly, the notation $C^*((u))$ denotes the completion of the graded cochain complex $C^*[u,u^{-1}]:=C^*\otimes_{\K}\K[u,u^{-1}]$ with respect to the $u$-adic filtration $F^p=C^*\otimes_{\K} u^p\K[u, u^{-1}]$.

\subsection*{Acknowledgements} I would like to thank my advisor Mohammed Abouzaid for his patient guidance and numerous comments, and Paul Seidel for explaining the subtle definitions of (negative) cyclic homology using u-adic completions. I am also grateful to Mark McLean who pointed out the work \cite{CFHW}, which is crucial to the proof of the main theorem. \\
\indent The latest revision of this work was completed during my postdoctoral year at the Institute for Advanced Study, where the author was supported by the National Science Foundation under agreement No. DMS-1128155. Any opinions, findings, conclusions or recommendations in this work are those of the author and do not necessarily reflect the views of the National Science Foundation. 

\section{Cyclic homology of $S^1$-Complexes}\label{sec:cyc}

Given a cochain complex $(C^*, \delta_0)$, there are notions of an $S^1$-structure on $C^*$ and three cyclic homology theories of such an $S^1$-complex $C^*$. The exposition in this section is based on \cite[Section 2]{Jones} and \cite[Section 2.4]{BO2} while using a cohomological convention.

\begin{definition}
Let $(C^*, \delta_0)$ be a $\Z/2$-graded (resp. $\Z$-graded) cochain complex over $\K$ such that $C^i$ is a free $\K$-module for all $i \in \Z/2$ (resp. $i \in \Z$). An \textit{$S^1$-structure }on $(C^*, \delta_0 )$ is given by a sequence of maps 
\begin{equation}
\delta:= (\delta_0, \delta_1, \delta_2, \cdots ) \text{ with }\delta_i\colon C^* \rightarrow C^{*+1-2i}
\nonumber \end{equation}
such that the relation
\begin{equation} \label{eqn:s1-d}
\sum_{i+j=k} \delta_i \delta_j =0
\end{equation}
is satisfied for every $k \geq 0.$ Such pair $(C^*, \delta )$ is called an $S^1$-complex.
\end{definition}

A mixed complex $(C^*, b, B)$ is a cochain complex $C^*$ equipped with maps
\begin{align}
b\colon C^* \rightarrow C^{*+1}, \nonumber \\
B\colon C^* \rightarrow C^{*-1}\nonumber
\end{align}
that satisfy the relations
\begin{equation}\label{eqn:mix}
b^2=0,\ \  B^2=0 \text{ and } Bb=-bB.
\end{equation}
Cyclic homology of a mixed complex is studied in \cite{Loday1}.
It is clear from above definitions that a mixed complex is an $S^1$-complex with
\begin{equation}
\delta_0=b, \ \ \delta_1=B \text{ and }\delta_i=0 \text{ for all } i \geq 2.
\nonumber \end{equation}
One should view an $S^1$-complex as an $\infty$-mixed complex, which means that an $S^1$-complex $(C^*, \delta)$ satisfies \eqref{eqn:mix} up to higher order homotopies.

There are also notions of cochain maps and homotopy operators defined for $S^1$-complexes.

\begin{definition}
Let $(C^*, \delta)$ and $(D^*, \partial)$ be  two $S^1$-complexes. An $S^1$-equivariant cochain map is a sequence of maps 
\begin{equation}
\kappa:=(\kappa_0, \kappa_1, \kappa_2,\cdots) \text{ with }\kappa_i\colon C^* \rightarrow D^{*-2i}
\nonumber \end{equation} 
such that the relation
\begin{equation} \label{eqn:s1-map}
\sum_{i+j=k} \kappa_i  \delta_j - \partial_j \kappa_i=0
\end{equation}
is satisfied for every $k \geq 0$.
Similarly, given two $S^1$-equivariant maps $\kappa, \kappa'$, an $S^1$-equivariant homotopy operator is a sequence of maps 
\begin{equation}
h:=(h_0, h_1, h_2,\cdots) \text{ with } h_i \colon C^* \rightarrow D^{*-2i-1}
\nonumber \end{equation}
such that the relation
\begin{equation}
\kappa_k - \kappa'_k=\sum_{i+j=k} h_i \delta_j + \partial_j h_i
\nonumber \end{equation}
is satisfied for all $k \geq 0$.
\end{definition}

To an $S^1$-complex $(C^*, \delta)$, one can associate three cyclic homology theories as follows.

Let $u$ be a formal variable of degree $2$. We define the following cochain complexes
\begin{align}
& C^{-}:=C^*[[u]]; \nonumber \\
& C^{\infty}:=u^{-1}(C^-)\cong C^*((u)); \nonumber\\
& C^{+}:=C^{\infty}/uC^{-},\nonumber
\end{align}
where $C^*[[u]]$ denotes the completion of $C^*[u]:=C^* \otimes_{\K} \K[u]$ in the category of graded cochain complexes with respect with the exhaustive $u$-adic filtration $F^p=C^*\otimes_{\K} u^p\K[u]$, that is, $C^*[[u]]=\varprojlim C^*[u]/u^pC^*[u]$. This convention has been used, for instance, in \cite[Section 8b]{seidel_biased} and \cite[Section 3.6]{sheridan}. The $S^1$-structure on $C^*$ gives rise to a differential on $C^{-}$ given by 
\begin{equation}
\delta ^{S^1}=\delta_0+ u \delta_1 + u^2 \delta_2 + \cdots,
\nonumber \end{equation}
and the relation \eqref{eqn:s1-d} ensures that $(\delta ^{S^1})^2=0$.
We also denote the differentials induced in the localization $C^{\infty}$ and the quotient $C^{-}$ by $\delta ^{S^1}$. Then the cyclic homology of the $S^1$-complex $(C^*, \delta)$, denoted as $HC_*(C^*)$ (or $HC_*^+(C^*)$), is defined to be $H_*(C^{+}, \delta ^{S^1})$. The periodic cyclic homology $HC_*^{\infty}(C^*)$ of $C^*$ is given by $H_*(C^{\infty}, \delta ^{S^1})$, and the negative cyclic homology $HC_*^-(C^*)$ of $C^*$ is defined to be $H_*(C^{-}, \delta ^{S^1})$. All three cyclic homology theories satisfy the following invariance property.

\begin{prop} \label{prop:cycinv}
Let $\kappa=(\kappa_0, \kappa_1, \kappa_2, \cdots)$ be a cochain map between $S^1$-complexes 
\begin{equation}
(C^*, \delta=(\delta_0, \delta_1,\delta_2, \cdots)) \text{ and } (D^*, \partial=(\partial_0, \partial_1, \partial_2, \cdots)). \nonumber
\end{equation}
If $\kappa_0\colon (C^*, \delta_0) \rightarrow (D^*, \partial_0)$ is a quasi-isomorphism, then $\kappa$ induces isomorphisms
\begin{align}
& HC_*(C^*) \cong HC_*(D^*), \nonumber \\
& HC_*^{\infty}(C^*) \cong HC_*^{\infty}(D^*),\nonumber\\
& HC_*^-(C^*) \cong HC_*^-(D^*).\nonumber
\end{align}
\end{prop}
The proof of Proposition \ref{prop:cycinv} is similar to \cite[Lemma 2.1]{Jones} in the case of mixed complexes.

\begin{proof}
The isomorphism $HC_*(C^*) \cong HC_*(D^*)$ follows from applying the Comparison Theorem \cite[5.2.12]{weibel} to the $E_1$-pages of the spectral sequences associated to the bounded below and exhaustive $u$-adic filtrations on $C^+$ and $D^+$. As taking homology commutes with localization with respect to the multiplicative set $S= \{1, u, u^2,\cdots \}$, we obtain that $$HC_*^{\infty}(C^*) \cong u^{-1}HC_*^-(C^*).$$ So it suffices to prove that $HC_*^-(C^*) \cong HC_*^-(D^*)$. For fixed $m,n \in \Z$ with $m \geq  n$, we consider the quotient complex
\begin{equation}
C \langle m,n \rangle :=u^n C^-/u^m C^- .
\nonumber \end{equation} 
There is a finite filtration on $C\langle m,n \rangle$
\begin{equation}
F^p C\langle m,n \rangle :=\{ \sum_{i \geq p}^{m-1} a_i u^i \mathbin{|} a_i \in C^*, p \geq n\} \text{ for }  n \leq p \leq m-1.
\nonumber \end{equation}
The filtration $F^p D \langle m,n \rangle$ on $D \langle m,n \rangle$ is defined similarly for $D^-$. The $E_1$-pages of the associated spectral sequences are given by
\begin{equation}
E^{p,q}_1=H_{q-p}(C^*, \delta_0) \text{ and } E'^{p,q}_1=H_{q-p}(D^*, \partial_0) \text{ if } n \leq p \leq m-1,
\nonumber \end{equation} 
and $E^{p,q}_1 = E'^{p,q}_1=0$ otherwise.
Since $\kappa_0$ induces a quasi-isomorphism between $E_1^{p,q}$ and $E'^{p,q}_1$ for all $p,q$, by the Comparison Theorem we obtain an isomorphism
\begin{equation}
f_*\colon H_*(C \langle m,n \rangle, \delta ^{S^1}) \xrightarrow{\cong} H_*(D\langle m,n \rangle, \partial^{S^1}) \text{ for any }m,n \in \Z \text{ with } m \geq  n.
\nonumber \end{equation}
By construction we have that 
\begin{equation} \label{eqn:u_complete}
u^n C^- =  \displaystyle\lim_{\substack{\longleftarrow \\ m}} C \langle m,n \rangle,
\end{equation}
where the inverse limit is taken as $m \rightarrow\infty$. Again, the same statement holds for $D^-$. By the $\displaystyle\lim{}^1$ exact sequence, there is a commutative diagram of short exact sequences
\[
\xymatrixcolsep{.65pc}\xymatrix{
 0 \ar[r] & 
 \displaystyle\lim_{\substack{\longleftarrow \\ m}} {}^1H_{q+1}( C \langle m,n \rangle, \delta ^{S^1}) \ar[r] \ar[d]^{f_*} & 
H_q(u^n C^-, \delta ^{S^1})  \ar[r]  \ar[d] & 
 \displaystyle\lim_{\substack{\longleftarrow \\ m}} H_q( C \langle m,n \rangle, \delta ^{S^1}) \ar[r] \ar[d]^{f_*} & 0  \\
 0 \ar[r] & 
 \displaystyle\lim_{\substack{\longleftarrow \\ m}} {}^1H_{q+1}( D \langle m,n \rangle, \partial^{S^1}) \ar[r]  & 
H_q(u^n D^-, \partial^{S^1})  \ar[r]  & 
 \displaystyle\lim_{\substack{\longleftarrow \\ m}} H_q( D \langle m,n \rangle, \partial^{S^1}) \ar[r]  & 0 
}
\]
This implies that 
$$H_*(u^n C^-, \delta ^{S^1}) \cong H_*(u^n D^-, \partial^{S^1}),\ \  \forall n \in \Z.$$ 
In particular, when $n=0$ we obtain that $H_*(C^-, \delta ^{S^1}) \cong H_*(D^-, \partial^{S^1})$, that is, $HC_*^-(C^*) \cong HC_*^-(D^*)$ as desired. 
\end{proof}

If $C^*$ is acyclic, then the proof of Proposition \ref{prop:cycinv} implies the following immediate consequence.

\begin{cor} \label{cor:van}
If $H_*(C^*,\delta_0)=0$, then 
$$HC_*(C^*)=HC_*^{\infty}(C^*)=HC_*^-(C^*)=0.$$
\end{cor}

\begin{rmk}

The fact that the complexes $C^-$ and $C^{\infty}$ are \textit{u-adically complete} is used in equation \eqref{eqn:u_complete} in Proposition \ref{prop:cycinv} to prove that the negative and periodic cyclic homologies of $S^1$-complexes are quasi-isomorphism invariants. Another reason to take u-adic completed complexes is that the structure morphisms $\delta_i$ may be nonzero for infinitely many $i \in \N$ for general $S^1$-complexes unlike the case of mixed complexes, so the differential $\delta^{S^1}$ is only well-defined on the $u$-adic completion $C^*[[u]]$.
\end{rmk}

\begin{rmk}\label{rmk:refined}
Given a $\Z$-graded $S^1$-complex $(C^*, \delta)$,  there is another refined negative cyclic cochain complex that one can consider $$C^*[[u]]^{gr}:=\bigoplus_{k\in \Z} C^k[[u]]^{gr},$$ where $C^k[[u]]^{gr}$ is defined to be the subspace of $C^*[[u]]$ which are spanned by homogeneous series of degree $k$. Such completion has been considered in \cite[Section 1]{JP} and \cite[Section 3.5]{Shklyarov}. If our cochain complex $C^*$ is furthermore bounded below and finitely generated over $\K$ in each degree, then we have the following isomorphisms
\begin{equation}
C^*[[u]]^{gr} \cong C^*[u] \text{ and } C^*((u))^{gr} \cong C^*[u,u^{-1}]
\nonumber \end{equation}
as $\K[u]$ and $\K[u, u^{-1}]$ modules.
\end{rmk}
\section{Symplectic Cohomology}\label{sec:symcoho}

In this section, we define symplectic cohomology of the class of exact symplectic manifolds with boundary. This is first defined by Cieliebak--Floer--Hofer--Wysocki in a sequence of papers \cite{CFH, CFHW, FH2, FHW} and reformulated by Viterbo in \cite{Viterbo1} and Seidel \cite{seidel_biased}. In this section, we follow the exposition in \cite{seidel_biased}.

\begin{definition} 
A Liouville domain is a compact manifold with boundary $(M, \partial M)$, together with a 1-form $\theta$ such that:\\
(1) $\omega=d \theta$ is a symplectic form on $M$; \\
(2) The Liouville vector field $Z$, defined by $i_Z \omega =\theta$, points strictly outwards along the boundary $\partial M$.
\end{definition}
This definition implies that $\alpha:=\theta|_{\partial M}$ is a contact form on $\partial M$. Therefore a Liouville domain is also called an exact symplectic manifold with contact type boundary. \\
\indent Given such a Liouville domain $(M, \theta)$, the flow of the Liouville vector field $\psi^r_Z$ for $r \in (-\infty,0]$ gives rise to a canonical collar neighborhood of $\partial M$
\begin{equation}
\Psi\colon (-\infty, 0]\times \partial M \rightarrow M, (r, y)\mapsto \psi^r_Z(y).
\nonumber \end{equation}
It follows that $\Psi^* \theta = e^r \alpha$ and $\Psi^* Z=\partial_r$. One can then define the completion of the Liouville domain $(M, \theta)$ to be
\begin{equation}
\widehat{M}=M \cup_{\partial M}([0, \infty) \times \partial M), \widehat{\theta}|_{[0, \infty) \times \partial M}=e^r\alpha,\ \widehat{Z} |_{[0, \infty) \times \partial M}=\partial_r,\ \ \widehat{\omega}=d\widehat{\theta}.
\nonumber \end{equation}
We set $R:=e^r$ from now on. Then the completion of $(M, \theta)$ is given by 
\begin{equation} 
\widehat{M}=M \cup_{\partial M}([1, \infty) \times \partial M),  \widehat{\theta}|_{[1, \infty) \times \partial M}=R\alpha, \  \widehat{Z} |_{[1, \infty) \times \partial M}=R\partial_R,\ \  \widehat{\omega}=d\widehat{\theta}. \nonumber
\end{equation}
\begin{definition}
Two Liouville domains $ M_0$ and $M_1$ are Liouville isomorphic if there is a diffeomorphism $\psi\colon \widehat{M}_0 \rightarrow \widehat{M}_1$ satisfying $\psi^* \widehat{\theta}_1=\widehat{\theta}_0+ df$, where $f$ is a compactly supported function on $\widehat{M}_0$.
\end{definition}

We will be interested in studying the Liouville isomorphism type of the completion $\widehat{M}$ of a Liouville domain $(M, \theta)$.

\subsection{Floer cohomology of admissible Hamiltonians}
Let $(\widehat{M}, \widehat{\omega}=d(R\alpha))$ be the completion of $(M, \theta)$. We know that $(\partial M, \alpha)$ is a contact manifold with the contact distribution $\xi:=\ker(\alpha)$. The Reeb vector field $\mathcal{R}_{\alpha}$ of $\alpha$ is defined by
\begin{equation}
i_{\mathcal{R}_{\alpha}}d\alpha \equiv 0 \text{ and } \alpha(\mathcal{R}_{\alpha})=1.
\nonumber \end{equation} 
We denote by $\psi_{\mathcal{R}_{\alpha}}^t$ the flow of the Reeb vector field $\mathcal{R}_{\alpha}$ at time $t$. A Reeb orbit is a smooth map $x\colon \R/l\Z \rightarrow \partial M$ satisfying $\dot{x}(t)=\mathcal{R}_{\alpha}(x(t))$. Such Reeb orbits are called nondegenerate if the linearized return map $d\psi^l_{\mathcal{R}_{\alpha}}\colon \xi_{x(0)} \rightarrow  \xi_{x(0)}$ has no eigenvalues equal to 1. For a generic choice of the contact form $\alpha$ on $\partial M$, all Reeb orbits of $\mathcal{R}_{\alpha}$ are nondegenerate and their periods form a discrete subset of $\R^+$, called the action spectrum, defined by
\begin{equation} \label{eqn:spectrum}
\mathscr{S}= \{ \mathscr{A}(x)=\int_{x}\alpha \  \mathbin{|}  \ x  \text{ is a Reeb orbit of } \mathcal{R}_{\alpha} \}.
\end{equation}
A time-dependent Hamiltonian $H_t\colon \widehat{M}\rightarrow \R$ is admissible if it satisfies
\begin{equation} \label{eqn:adham}
H_t (r,y ) =\tau R+ C \text{ for } R \gg 1,
\end{equation}
for some $\tau \notin \mathscr{S}$, and $\tau$ is called the slope of $H_t$. We denoted by $\mathscr{H}(M)$ the space of admissible Hamiltonians on $\widehat{M}$ such that all 1-periodic orbits are nondegenerate. There is a pre-order on $\mathscr{H}(M)$ defined by 
\begin{equation}
H_t \preceq K_t \text{  if the slope of } H_t \text{ is less than or equal to the slope of } K_t.
\nonumber \end{equation} 
Given a Hamiltonian $H_t \in \mathscr{H}(M)$, the action functional $\mathscr{A}_{H_t} \colon \mathscr{L}\widehat{M} \rightarrow \R$ is given by
\begin{equation}
\mathscr{A}_{H_t} (x) = - \int_{S^1} x^*\widehat{\theta} + \int_{S^1} H_t(x(t))dt.
\nonumber \end{equation}
The Hamiltonian vector field $X_{H_t}$ of $H_t$ is defined by $i_{X_{H_t}}\omega= - dH_t$.
Critical points of $\mathscr{A}_{H_t}$ are precisely the nondegenerate $1$-periodic orbits of $X_{H_t}$, denoted by $\mathscr{P}(H_t)$.
We consider the space of time-dependent $\widehat{\omega}$-compatible almost complex structures $\mathscr{J}(M)$ such that elements $J_t \in \mathscr{J}(M)$ satisfy
\begin{equation}
d(R) \circ J_t=-\widehat{\theta} \text{ for } R \gg 1.
\nonumber \end{equation}
This implies that on 
$$T([R_0,\infty) \times \partial M) \cong (\R Z \oplus \R \mathcal{R}_{\alpha}) \oplus \xi \text{ with }R_0 \gg 1,$$ 
the almost complex structure $J_t$ is standard on $\R Z \oplus \R \mathcal{R}_{\alpha}$ and allowed to be any $d\alpha$-compatible time-dependent almost complex structure on $\xi=\ker(\alpha)$. If $H_t \in \mathscr{H}(M)$ and $J_t \in  \mathscr{J}(M)$, then we call $(H_t, J_t)$ an admissible pair.
Given an admissible pair $(H_t, J_t)$, we let $\widetilde{\mathscr{M}}(x_0,x_1, H_t, J_t)$ be the moduli space of solutions to the Floer equation
\begin{equation} \label{eqn:Floer}
\partial_s u + J_t(u) (\partial_t u-X_{H_t}(u))=0,
\end{equation}
with $E(u)= \int_{\R \times S^1} ||\partial_s u||^2 dsdt < \infty$. The finite energy condition is equivalent to having asymptotic conditions
\begin{equation}
\displaystyle\lim_{s \rightarrow - \infty} u(s,t)= x_0(t), \displaystyle\lim_{s \rightarrow + \infty} u(s,t)= x_1(t),
\nonumber \end{equation}
where $x_0, x_1$ are $1$-periodic orbits of $X_{H_t}$ such that 
$$E(u)=\mathscr{A}_{H_t}(x_0)-\mathscr{A}_{H_t}(x_1).$$
It can be shown that for fixed $H_t \in \mathscr{H}(M)$ and a generic choice of $J_t \in \mathscr{J}(M)$, the moduli space 
$$\mathscr{M}(x_0,x_1):=\widetilde{\mathscr{M}}(x_0,x_1, H_t, J_t)/\R$$ 
is a smooth manifold of dimension $|x_0|-|x_1|-1$, where $|x|=n-\mu_{CZ}(x)$ for $x \in \mathscr{P}(H_t)$, and $\mu_{CZ}(x)$ denotes the Conley--Zehnder index of a nondegenerate $1$-periodic orbit $x$.\\
\indent Due to the openness of $(\widehat{M}, \widehat{\omega})$, the usual Gromov-compactness argument fails for arbitrary Hamiltonians and $\widehat{\omega}$-compatible almost complex structures on $\widehat{M}$. However for an admissible pair $(H_t,J_t)$, Floer trajectory $u \in  \widetilde{\mathscr{M}}(x_0, x_1, H_t, J_t)$ obeys a maximal principle, shown in \cite[Lemma 1.4]{Oan}, which implies that images of any sequence of $J_t$-holomorphic maps lie in a compact region $M \cup_{\partial M} [1, R_0]$ for some $R_0 \gg 1$.
One can define the Floer cochain complex of the admissible Hamiltonian $H_t$ by
\begin{equation}
CF^i(M, H_t):=\bigoplus_{x \in \mathscr{P}(H_t), |x|=i} | o_x |,
\nonumber \end{equation}
where $|o_x|$ and the differential $d\colon CF^*(M,H_t) \rightarrow CF^{*+1}(M,H_t)$  are defined in the following section. 

\begin{rmk}
We have suppressed the dependence of $CF^*(M,H_t)$ on the almost complex structure $J_t$ as the homology of $(CF^*(M, H_t),d)$ is independent of the choice of $J_t \in \mathscr{J}(M)$.
\end{rmk}

\subsection{Orientations} \label{subsec:orient}
To define the Floer cochain complex over $\Z$-coefficients, we will relatively orientate the moduli space $\mathscr{M}(x_0,x_1)$ for all $x_0$ and $x_1$ in $\mathscr{P}(H_t)$ as in \cite[Section 1.4]{abu_viterbo}.\\
\indent Given a path $\Psi(t)$ in $\Sp(2n)$ such that $\Psi(0)=\mathbb{I}$ and $\det(\mathbb{I}-\Psi(1)) \neq 0$, one can reparametrize $\Psi(t)$ and associate to it a loop of symmetric matrices $S(t)$ that satisfies
\begin{equation}
\dot{\Psi}(t)=J_0 S(t)\cdot \Psi(t).
\nonumber \end{equation}
Such a loop $S(t)$ will be called nondegenerate if $\det(\mathbb{I}-\Psi(1)) \neq 0$.
For nondegenerate $S(t)$, we denote by $\mathscr{O}_{+}(S)$ and $\mathscr{O}_{-}(S)$ the spaces of all operators of the form
\begin{eqnarray}
&& D_{\Psi}\colon W^{1,p}(\C, \R^{2n}) \rightarrow L^p(\C, \R^{2n}),\ \ p>2 \label{eqn:opd}  \\
&& D_{\Psi}(X)=\partial_s X + J_0\partial_t X + S \cdot X, \nonumber
\end{eqnarray}
where $S \in C^{0}(\C, \mathfrak{gl}(2n))$ is required to satisfy
\begin{eqnarray}
&& S(e^{s+2\pi it})=S(t) \text{ for } s \gg 0, \text{ if } D_{\Psi} \in \mathscr{O}_+(S), \nonumber \\
&& S(e^{-s-2\pi it})=S(t) \text{ for } s \ll 0, \text{ if }D_{\Psi} \in \mathscr{O}_-(S). \nonumber
\end{eqnarray}
Similarly, for loops of nondegenerate symmetric matrices $S_-(t)$ and $S_+(t)$ corresponding to paths $\Psi_-(t)$ and $\Psi_+(t)$ in $\Sp(2n)$, we define $\mathscr{O}(S_-, S_+)$ to be the space of all operators 
\begin{eqnarray}
&& D\colon W^{1,p}(\R \times S^1, \R^{2n}) \rightarrow L^p(\R \times S^1, \R^{2n}), \ \ p>2  \nonumber  \\
&& D(X) =\partial_s X + J_0\partial_t X + S \cdot X,\nonumber
\end{eqnarray}
with $S \in C^{0}(\R \times S^1, \mathfrak{gl}(2n))$ satisfying
\begin{equation}
S(s,t)=S_-(t) \text{ for } s \ll 0 \text{ and } S(s,t)=S_+(t) \text{ for } s\gg0. \nonumber
\end{equation}

It can be shown that $\mathscr{O}_{+}(S)$, $\mathscr{O}_{-}(S)$ and $\mathscr{O}(S_-,S_+)$ consist of Fredholm operators. One can define line bundles $\Det(\mathscr{O}_{\pm}(S))$, $\Det(\mathscr{O}(S_-,S_+))$ over $\mathscr{O}_{\pm}(S)$ and $\mathscr{O}(S_-,S_+)$ by declaring the fiber over an element $D$ in $\mathscr{O}_{\pm}(S)$ or $\mathscr{O}(S_-,S_+)$ to be the determinant line bundle $\det(D)$ of the Fredholm operator $D$. If we fix nondegenerate asymptotic data $S(t)$ and $S_{\pm}(t)$, the spaces $\mathscr{O}_{\pm}(S)$ and $\mathscr{O}(S_-,S_+)$ are contractible. This implies that line bundles $\Det(\mathscr{O}_{\pm}(S))$ and $\Det(\mathscr{O}(S_-,S_+))$ are trivial for fixed loops of symmetric matrices $S(t)$ and $S_{\pm}(t)$.\\
\indent For given nondegenerate asymptotic data $S_+(t)$, $S_-(t)$ and $S(t)$, we consider Fredholm operators 
\begin{align*}
\text{ \ \ \  } K \in  \mathscr{O}(S_-, S) \text{ and } L \in \mathscr{O}(S, S_+),\nonumber \\
 \text{ or }K \in  \mathscr{O}_+(S) \text{ and } L \in \mathscr{O}(S, S_+), \nonumber  \\
 \text{ or }K \in \mathscr{O}(S_-, S)  \text{ and } L \in \mathscr{O}_-(S). \nonumber
\end{align*}

There is a linear gluing operation, denoted as $K \#_{\rho} L$, for Fredholm operators $K$ and $L$ defined as follows. Let $Z=\R \times S^1$ be the infinite cylinder. One first forms the glued Riemann surfaces $\ Z \#_{\rho} Z$, $\C \#_{\rho} Z$ and $Z \#_{\rho} \C$ under the identifications 
\begin{eqnarray}
&& [-\rho, \rho] \times S^1 \subset (-\infty, \rho] \times S^1 \rightarrow [-\rho,\rho] \times S^1 \subset [-\rho, \infty) \times S^1 \colon \nonumber \\
&& \ \ \ \ \ \ \ \ \ \ \ \ \ \ \ \ \ \ \ \ \ \ \ \ \ \ \ \ \ \ \   (s,t) \mapsto  (s, t), \nonumber   
\end{eqnarray}
\begin{eqnarray}
&& [\rho,2\rho] \times S^1  \subset (-\infty, 2\rho]\times S^1 \rightarrow  \{ z \mathbin{|} e^{\rho} \leq |z| \leq e^{2\rho} \} \subset \{ z \mathbin{|} |z| \leq e^{2\rho} \}\colon \nonumber \\
&&  \ \ \ \ \ \ \ \ \ \ \ \ \ \ \ \ \ \ \ \ \ \ \ \ \ \ \ \ \ \ \  (s,t) \mapsto  e^{3\rho-s-2\pi it}, \nonumber
\end{eqnarray}
\begin{eqnarray}
&& [\rho,2\rho] \times S^1 \subset [\rho, \infty) \times S^1  \rightarrow  \{ z \mathbin{|} e^{\rho} \leq |z| \leq e^{2\rho} \} \subset \{z \mathbin{|} |z| \leq e^{2\rho} \}\colon \nonumber \\
&&  \ \ \ \ \ \ \ \ \ \ \ \ \ \ \ \ \ \ \ \ \ \ \ \ \ \ \ (s,t) \mapsto  e^{s+2\pi it}.\nonumber
\end{eqnarray}
For $\rho\gg0$, the restrictions of the inhomogeneous terms of $K$ and $L$ to the glued regions coincide with the value of $S$, we then obtain the glued operators $K \#_{\rho} L$ in $\mathscr{O}(S_-,S_+)$, $\mathscr{O}_+(S_+)$ and $\mathscr{O}_-(S_-)$ in each case.
With respect to this gluing operation, it is shown in \cite[Proposition 9]{FH} that there is a canonical isomorphism
\begin{equation} \label{eqn:gluing}
\det(K \#_{\rho} L)\cong \det(K) \otimes \det(L)
\end{equation}
up to multiplication by a positive real number. 

We now define the Floer differential on $CF^*(M, H_t)$ over $\Z$. Let $\psi^t_H$ be the Hamiltonian flow generated by the Hamiltonian vector field $X_{H_t}$. Each $1$-periodic orbit yields a map $x\colon S^1 \rightarrow \widehat{M}$ such that $\psi^t_H(x(0))=x(t)$. For any $J_t \in \mathscr{J}(M)$, the complex vector bundle $(x^*T\widehat{M}, J_t)$ over $S^1$ can be trivialized. We choose a trivialization $\xi(t)\colon \R^{2n} \rightarrow T_{x(t)}\widehat{M}$ and obtain a path $\Psi_x(t)$ in $\mathrm{Sp}(2n)$ as the composition of
\begin{equation}
\R^{2n} \xrightarrow{\xi(0)} T_{x(0)}\widehat{M} \xrightarrow{d\psi^t_H} T_{x(t)}\widehat{M} \xrightarrow{\xi^{-1}(t)} \R^{2n}.
\nonumber \end{equation}
 Given $x_0$ and $x_1$ in $\mathscr{P}(H_t)$, we denote by $S_0(t)$ and $S_1(t)$ the loops of symmetric matrices which generate $\Psi_{x_0}(t)$ and $\Psi_{x_1}(t)$, respectively. The construction in \eqref{eqn:opd} yields operators $D_{\Psi_{x_0}}$ and $D_{\Psi_{x_1}}$ in $\mathscr{O}_-(S_0)$ and $\mathscr{O}_-(S_1)$ respectively. We introduce the notation 
 $$o_{x_i}:=|\det(D_{\Psi_{x_i}})| \text{ for } i=0,1,$$ 
 where $|\cdot|$ denotes the graded abelian group generated by the two orientations of $\det(D_{\Psi_{x_i}})$ and modulo the relation that the sum vanishes. Similarly for $u \in \widetilde{\mathscr{M}}(x_0,x_1)$ we define $o_u:=|\det(D_u)|$, where  $D_u \in \mathscr{O}(S_0, S_1)$ is the linearization of the Floer equation at $u$ with respect to a trivialization of $u^*(T\widehat{M}) \rightarrow \R \times S^1$ that agree with the trivializations of $x_i^*(T\widehat{M}) \rightarrow S^1$ as $s \rightarrow \pm \infty$.
For different choices of trivializations, Lemma 13 in \cite{FH} and Proposition 1.4.10 in \cite{abu_viterbo} show that the corresponding determinant line bundles $\det(D_{\Psi_{x_i}})$ and $\det(D_u)$ are isomorphic. This implies that $o_{x_i}$ and $o_u$ are well-defined for $i=0,1$.
By the gluing property \eqref{eqn:gluing}, we have a canonical isomorphism
\begin{equation} 
o_u \otimes o_{x_0} \cong o_{x_1}.
\nonumber \end{equation}

Together with the fact that $o_u \cong |\R \partial_s| \otimes |\mathscr{M}(x_0, x_1)|$, we obtain an isomorphism
\begin{equation} \label{eqn:oriet2}
o_{x_1} \cong |\R \partial_s| \otimes |\mathscr{M}(x_0, x_1)| \otimes o_{x_0},
\end{equation}
where $\R \partial_s$ is the 1-dimensional subspace of $\ker(D_u)$ spanned by translation in positive $s$-direction. For $|x_0|=|x_1|+1$, we have that $T\mathscr{M}(x_0,x_1)$ is canonically trivial as $\mathscr{M}(x_0,x_1)$ is a 0-dimensional manifold. By comparing the fixed orientations on both sides of \eqref{eqn:oriet2}, we obtain an isomorphism
\begin{equation}\label{eqn:d_u}
d_u\colon o_{x_1} \rightarrow o_{x_0}.
\end{equation}

The differential $d\colon CF^*(M, H_t) \rightarrow CF^{*+1}(M, H_t)$ is then defined to be
\begin{equation}
d|_{o_{x_1}} := \bigoplus_{|x_0|=|x_1|+1} \sum_{u \in \mathscr{M}(x_0,x_1)} d_u.
\nonumber \end{equation}

\subsection{Symplectic cohomology as a direct limit}
Given admissible pairs $(H_t, J_t^+)$ and $(K_t, J_t^-)$ with $H_t \preceq K_t$, one can choose a monotone homotopy $(H_s, J_s)$ between $(H_t, J_t^+)$ and $(K_t, J_t^-)$. Precisely, this means that the family of Hamiltonians satisfy
\begin{equation}
 H_s(r,y)=h_s(R) \text{ and } \frac{\partial h_s'}{\partial s} \leq 0 \text{ for }  R \gg 1,
\nonumber \end{equation}
and the family of $\widehat{\omega}$-compatible almost complex structures $J_s$ belongs to $\mathscr{J}(M)$ for each $s \in \R$.
It is shown in \cite[Section 3c]{seidel_biased} that there is a well-defined continuation map
\begin{equation} \label{eqn:kappa}
\kappa\colon HF^*(M, H_t) \rightarrow HF^*(M, K_t)
\end{equation}
associated to such a monotone homotopy. Furthermore, the continuation map $\kappa$ is independent of the choice of monotone homotopy. The symplectic cohomology of $M$ is then defined by
\begin{equation} 
SH^*(M)= \varinjlim\limits_{H_t \text{ admissible} } HF^*(M, H_t),
\nonumber \end{equation}
which is the direct limit of all Floer cohomology groups of admissible Hamiltonians $H_t$ with respect to the continuation map \eqref{eqn:kappa}.

\subsection{The BV operator on the symplectic cohomology} \label{sec:bv}
Given an admissible Hamiltonian $H_t$ of slope $\tau$ as in \eqref{eqn:adham}, we define a family of admissible Hamiltonians $H^{\theta}_{s,t}\colon \widehat{M} \rightarrow \R$ parametrized by $\R \times S^1$ by 
\begin{eqnarray}
&& H^{\theta}_{s,t}=H_t \text{ if } s \gg 0,\ \  H^{\theta}_{s,t}=H_{t+ \theta} \text{ if } s \ll 0, \nonumber \\
&& H_{s,t}^{\theta}(r,y)=\tau R+C \text{ for } R \gg 1.\nonumber
\end{eqnarray}
We also choose a family of almost complex structures $J^{\theta}_{s,t} \in \mathscr{J}(M)$ that agree with $J_t$ if $s \gg 0$ and $J_{t+\theta}$ if $s \ll 0$. The moduli space $\mathscr{M}_{\Delta}(x_0, x_1)$ consists of pairs $(\theta, u)$ such that $\theta \in S^1$ and $u\colon \R \times S^1 \rightarrow \widehat{M}$ satisfies
\begin{eqnarray}
&& \partial_s u+ J^{\theta}_{s,t}(\partial_tu-X_{H^{\theta}_{s,t}}(u))=0  \label{eqn:bv} \nonumber \\
&& \displaystyle\lim_{s \rightarrow -\infty} u(s, t+\theta)=x_0(t+ \theta), \ \ \displaystyle\lim_{s \rightarrow +\infty} u(s, t)=x_1(t).\nonumber
\end{eqnarray}
For each element $(\theta, u) \in \mathscr{M}_{\Delta}(x_0, x_1),$ the linearization of \eqref{eqn:bv} at $u$ is a Fredholm operator 
\begin{equation}
T_{\theta}S^1 \oplus W^{1,p}(Z, u^*(T\widehat{M})) \rightarrow L^p(Z, u^*(T\widehat{M})),
\nonumber \end{equation}
where the second component of this Fredholm map is given by the linearized operator $D_u$ with respect to the Floer data $(H^{\theta}_{s,t}, J_{s,t}^{\theta})$ for the fixed element $\theta \in S^1,$ and the first component of the Fredholm map is given by
\begin{equation}
\partial_{\theta} \mapsto \frac{\partial J^{\theta}_{s,t}}{\partial \theta}\bigg(\partial_t u-X_{H^{\theta}_{s,t}}(u) \bigg)-J_{s.t}^{\theta} \frac{\partial X_{H^{\theta}_{s,t}}}{\partial \theta}.
\nonumber \end{equation} 
The pair $(\theta, u)$ is said to be regular if for generic choices of Floer data $H^{\theta}_{s,t}$ and $J^{\theta}_{s,t}$, the map $T_{\theta}S^1 \rightarrow \coker(D_u)$ is surjective. For such generic Floer data, the dimension of $\mathscr{M}_{\Delta}(x_0,x_1)$ is $|x_0|-|x_1|+1$. Given $(\theta, u) \in \mathscr{M}_{\Delta}(x_0,x_1)$, there is a short exact sequence
\begin{equation}
0 \rightarrow T_u\mathscr{M}_{\Delta}(x_0, x_1) \rightarrow T_{\theta}S^1 \oplus \ker(D_u) \rightarrow \coker(D_u) \rightarrow 0,
\nonumber \end{equation}
which induces an isomorphism of determinant lines
\begin{equation}\label{eqn:bvorient}
\det(T_u\mathscr{M}_{\Delta}(x_0, x_1))\cong \det(D_u)\otimes \det(T_{\theta}S^1).
\end{equation}
By the gluing described in \eqref{eqn:gluing}, an orientation of $\det(D_u)$ is determined by the isomorphism
\begin{equation}
|\det(D_u)|\otimes o_{x_1} \cong o_{x_0}.
\nonumber \end{equation}
If one fixes the orientation of $T_{\theta}S^1$ by the trivialization $T_{\theta}S^1 \cong \R \langle \partial_{\theta} \rangle$, then \eqref{eqn:bvorient} yields an isomorphism
\begin{equation} \label{eqn:bvou}
\Delta_u\colon o_{x_1} \rightarrow o_{x_0}.
\end{equation}
Then we define the BV operator $\Delta\colon CF^*(M, H_t) \rightarrow CF^{*-1}(M, H_t)$ as
\begin{equation}
\Delta|_{o_{x_1}}=\bigoplus_{|x_0|=|x_1|-1} \sum_{(\theta, u) \in \mathscr{M}_{\Delta}(x_0,x_1)} \Delta_u .
\nonumber \end{equation}
It is shown in \cite[Section 2.2]{abu_viterbo} that this in fact defines an operation $\Delta$ on the symplectic cohomology $SH^*(M)= \varinjlim HF^*(M, H_t)$.

\section{Localized Periodic Symplectic Cohomology $\hpl(M)$}\label{sec:psh}
\subsection{The equivariant differential}
The equivariant differential was first introduced by Seidel in \cite[Section 8b]{seidel_biased} and defined explicitly by Bourgeois and Oancea \cite{BO2} and by Seidel \cite[Section (5b)]{seidel_conn} using Floer (co)homology associated to the locally trivial fibration $\mathscr{L}\widehat{M} \times_{S^1}ES^1 \rightarrow \C P^{\infty}$. The following exposition follows the definition of the equivariant differential given by Bourgeois and Oancea \cite[Sections 2.2, 2.3]{BO2} while using a cohomological convention.\\
\indent We consider a sequence of Morse-Smale pairs $(f_N, g_N)$ on $\C P^N$, where $f_N$ is the standard Morse function
\begin{equation} \label{eqn:morse}
f_N([z_0: \cdots :z_N])= \frac{C \sum_{j=0}^N(j+1)|z_j|^2}{\sum_{j=0}^N|z_j|^2} \ \text{ for some } C>0,
\end{equation}
and $g_N$ is the metric induced by the standard round metric on $S^{2N+1}$. We denote the $S^1$-invariant lift of $f_N$ to $S^{2N+1}$ by $\widetilde{f}_N$. Each critical point $z_0$ of $f_N$ on $\C P^N$ gives rise to a critical orbit $S^1\cdot z_0$ of the Morse-Bott function $\widetilde{f}_N$ on $S^{2N+1}$. On the principal $S^1$-bundle $\pi_N \colon S^{2N+1} \rightarrow \C P^N,$ we choose a flat connection $\nabla^{N}$ so that the horizontal section $s$ of $\pi_N$ over some neighborhood $N(z_0)$ of $z_0$ defines a horizontal local slice $U_{z_0}:=s(N(z_0))$ that is transverse to orbits of the $S^1$-action for each $z_0 \in \mathrm{Crit(f_N)}$. For any point $z$ in $V_{z_0}:=S^1 \cdot U_{z_0}$, there is a unique element $\theta_z \in S^1$ such that $\theta_z^{-1} \cdot z$ belongs to the horizontal local slice $U_{z_0}$. A Hamiltonian $H_N \colon S^1 \times \widehat{M} \times S^{2N+1} \rightarrow \R$ is said to be admissible if it satisfies the following conditions:
\begin{enumerate}
\item[(1)] ($S^1$-invariance) $H_N(t+\theta,x, \theta \cdot z)=H_N(t,x,z)$ for $\theta \in S^1.$
\item[(2)] (Nondegeneracy) The Hamiltonian $H^{z_0}_N(t,x):=H_N(t,x,z_0)$ has only nondegenerate $1$-periodic orbits for all $z_0 \in \mathrm{Crit}(f_N).$
\item[(3)] (Locally constant family near critical orbits of $\widetilde{f}_N$) On $S^1 \times \widehat{M} \times V_{z_0}$, the Hamiltonian $H_N$ satisfies $H_N(t,x,z)=H_{t-\theta_z}(x)$ for some $H_t$ in $\mathscr{H}(M)$ and for all $z_0 \in \mathrm{Crit}(f_N)$.
\end{enumerate}
We denote by $\mathscr{H}^{S^1}_N$ the space of admissible time-dependent Hamiltonians on $S^1 \times\widehat{M} \times S^{2N+1}.$ The condition $(3)$ ensures that the Hamiltonian extension $H_N$ is the constant family along each local horizontal slice $U_{z_0}$. Together with the $S^1$-invariance property in condition $(1)$, we conclude that the admissible extension $H_N$ are locally constant families defined over each neighborhood $N(z_0)$ of $z_0 \in \mathrm{Crit}(f_N).$ For instance, given a Hamiltonian $H_t$ in $\mathscr{H}(M)$, we can extend it to an admissible Hamiltonian $H_N \in \mathscr{H}_N^{S^1}$ by
\begin{equation} \label{eqn:H_N}
H_N(t,x,z)=\beta(z)H_{t-\theta_z}(x)+(1-\beta(z))\rho(x)H_t(x),
\end{equation}
where $\beta(z)$ is an $S^1$-invariant cut-off function on $S^{2N+1}$ which is equals to $1$ on some open subset $V'_{z_0} \subset V_{z_0}$ and $0$ outside $V_{z_0}$ for each $z_0 \in \Crit(f_N),$ and $\rho(x)$ is a cut-off function which is $1$ on $[R_0, \infty) \times \partial M$ for $R_0 \gg 0$ and 0 in the region where $H_t$ depends on time. Similarly, we define the space of $S^1$-invariant $\widehat{\omega}$-compatible almost complex structures on $\widehat{M}$ parametrized by $S^{2N+1}$ as follows
\begin{equation}
\mathscr{J}_N^{S^1}=\{ J_t^z, t \in S^1, z \in S^{2N+1} \mathbin{|}  J_{t+\theta}^{\theta \cdot z}=J_t^z \text{ and } J^z_t \in \mathscr{J}(M),\forall z \in S^{2N+1} \}.
\nonumber 
\end{equation}
For each $N$, we choose $J_N \in \mathscr{J}_N^{S^1}$ such that $J_N^z$ is independent of $z$ along the horizontal local slice $U_{z_0}$, and we choose an $S^1$-invariant metric $\widetilde{g}_N$ on $S^{2N+1}$ such that $\nabla \widetilde{f}_N$ is tangent to $U_{z_0}$ for each $z_0 \in \Crit(f_N)$. \\
\indent For Floer data $(H_N, J_N, \widetilde{f}_N, \widetilde{g}_N)$ chosen as above, it is shown in \cite{BO2} that for large enough constant $C >0$ in \eqref{eqn:morse} the critical points of
 $$\mathscr{A}_{H_N+\widetilde{f}_N}\colon \mathscr{L}\widehat{M} \times S^{2N+1} \rightarrow \R$$ are of the form $$\coprod_{j=1}^N S^1 \cdot (\mathscr{P}(H_t) \times \{ Z_j\}),$$ where $Z_j$ is a critical point of index $2j$. Let $CF^*(M, H_t)$ be the Floer cochain complex of $H_t$. The cochain complex associated to $\mathscr{A}_{H_N+\widetilde{f}_N}$ has a particularly simple form
\begin{equation}
CF^*(H_N+\widetilde{f}_N, J_N, \widetilde{g}_N) \cong CF^*(M, H_t) \otimes_{\Z} \Z[u]/(u^{N+1})
\nonumber \end{equation}
under the map $S^1 \cdot (\gamma, Z_j) \mapsto \gamma \otimes u^j$ for $\gamma \in \mathscr{P}(H_t)$ and $Z_j \in \Crit(f_N)$. The index of $x=(\gamma,Z)$ in $CF^*(H_N+\widetilde{f}_N, J_N, \widetilde{g}_N)$ is defined by
\begin{equation}
|x|=|\gamma|+\Ind_M(Z),
\nonumber 
\end{equation}
where $\Ind_M(Z)$ is the standard Morse index of the critical point $Z$ with respect to $f_N.$\\
\indent Let $H_N^z:=H_N(\cdot,\cdot,z)$ be the Hamiltonian defined on $\widehat{M}$. For critical points given by $x_0=(\gamma_0, Z_i)$ and $ x_1=(\gamma_1, Z_0 )$ in $\mathscr{P}(H_t) \times \Crit(f_N)$ with $\Ind_M(Z_0)=0$ and $\Ind_M(Z_i)=2i$, we denote by 
\begin{equation}
\widetilde{{\mathscr{M}}}_i(\gamma_0, \gamma_1):=\widetilde{\mathscr{M}}_i(x_0, x_1, H_N, J_N, \widetilde{f}_N, \widetilde{g}_N)
\nonumber \end{equation}
the moduli space of solutions $u\colon \R \times S^1 \rightarrow \widehat{M}$ and $z\colon \R \rightarrow S^{2N+1}$ to the system of equations
\begin{equation}  \label{eqn:delta}
\begin{cases}
\partial _s u+J_N^{z(s)}(u)(\partial _t u-X_{H^{z(s)}_N}(u))=0, \\
\dot{z}+\nabla \widetilde{f}_N(z)=0,
\end{cases}
\end{equation}
with asymptotic behaviors
\begin{equation}
\displaystyle\lim_{s \rightarrow - \infty}(u(s, \cdot), z(s)) \in S^1 \cdot x_0, \ \displaystyle\lim_{s \rightarrow \infty}(u(s, \cdot), z(s)) \in S^1 \cdot x_1,
\nonumber \end{equation}
where $S^1 \cdot x_j$ is the $S^1$-orbit of $x_j$ for $j=0,1$.
There is a free $\R \times S^1$-action on $\widetilde{\mathscr{M}}_i(\gamma_0, \gamma_1)$ given by reparametrizations in the domain $\R \times S^1$. We denote the quotient space by 
$$\mathscr{M}_i(\gamma_0, \gamma_1):=\widetilde{\mathscr{M}}_i(\gamma_0, \gamma_1)/\R \times S^1.$$ 
For a generic choice of data $(H_N, J_N, \widetilde{f}_N, \widetilde{g}_N)$, it is proved in \cite[Section 2.1]{BO2} that the dimension of the moduli space $\mathscr{M}_i(\gamma_0, \gamma_1)$ is
\begin{equation}
|x_0|-|x_1|=|\gamma_0|-|\gamma_1|-1+2i.
\nonumber \end{equation} 
In particular, if $|\gamma_0|=|\gamma_1|+1-2i$, then the moduli space $\mathscr{M}_i(\gamma_0, \gamma_1)$ is zero dimensional. For each element $u$ of $\mathscr{M}_i(\gamma_0, \gamma_1)$ one obtains an isomorphism 
\begin{equation}
\delta_{i,u}\colon o_{\gamma_1} \rightarrow o_{\gamma_0}
\nonumber \end{equation}
by relatively orienting the parametrized moduli space $\mathscr{M}_i(\gamma_0, \gamma_1)$ as in the case of the BV operator. This yields an operation $\delta_i\colon CF^*(M, H_t) \rightarrow CF^{*+1-2i}(M, H_t)$ defined by
\begin{equation} \label{eqn:delta_i}
\delta_i|_{o_{\gamma_1}}= \bigoplus_{|\gamma_0|=|\gamma_1|+1-2i} \sum_{u \in \mathscr{M}_i(\gamma_0,\gamma_1)} \delta_{i,u}.
\end{equation}
The formula $\delta^{S^1}_N(\gamma)=\sum_{i=0}^N u^i \delta_i(\gamma)$ defines a differential on the cochain complex $CF^*(M, H_t)\otimes_{\Z} \Z[u]/(u^{N+1}\Z[u])$. \\
\indent For different $N$, the Floer data $(H_N, J_N, \widetilde{f}_N,\widetilde{g}_N)$ chosen for $\widehat{M} \times S^{2N+1}$ and the Floer data $(H_{N+1}, J_{N+1},\widetilde{f}_{N+1}, \widetilde{g}_{N+1})$ chosen for $\widehat{M} \times S^{2N+3}$ are required to satisfy compatibility conditions as in \cite[Section 2.3]{BO2}
\begin{eqnarray}
&& H_{N+1}(t,x,i_1(z))=H_{N+1}(t, x,i_0(z)) = H_N(t, x, z),\nonumber \\
&& J_{N+1}^{i_1(z)} = J_{N+1}^{i_0(z)} = J_N^z\text{ and } i_1^*\widetilde{g}_{N+1} = i_0^*\widetilde{g}_{N+1} = \widetilde{g}_N,\nonumber
\end{eqnarray}
where $i_0$ and $i_1$ are the natural inclusions  $i_0, i_1 \colon S^{2N+1} \rightarrow S^{2N+3}$ defined by
\begin{eqnarray}
&& i_0 (z_0, z_1,\cdots, z_{N})= (z_0, z_1,\cdots, z_{N},0), \nonumber\\
&& i_1(z_0, z_1,\cdots, z_{N})= (0,z_0, z_1,\cdots, z_{N}), \text{ where } \Sigma_{i}|z_i|^2=1.\nonumber
\end{eqnarray}
Similarly, the flat connections $\nabla^N, \nabla^{N+1}$ on the principal $S^1$-bundles $\pi_N \colon S^{2N+1} \rightarrow \C P^N$  and $\pi_{N+1} \colon S^{2N+3} \rightarrow \C P^{N+1}$ can also be chosen so that they are compatible with the inclusions $i_0$ and $i_1.$ These conditions ensure that the $S^1$-equivariant differentials defined on the cochain complexes 
\begin{equation}
CF^*(M, H_t)\otimes_{\Z}\Z[u]/(u^{N+1}\Z[u]) \text{ and } CF^*(M, H_t)\otimes_{\Z}\Z[u]/(u^{N+2}\Z[u]) 
\nonumber \end{equation}
are compatible. After taking inverse limit as $N \rightarrow \infty$, one obtains a well-defined $S^1$-equivariant differential on 
\begin{equation}
CF^*(M, H_t)[[u]] \cong \varprojlim_N CF^*(M, H_t)\otimes_{\Z}\Z[u]/(u^{N+1}\Z[u])
\nonumber \end{equation}
given by
\begin{equation}
\delta ^{S^1}=\delta _0+u\delta _1+u^2\delta_2 + \cdots.
\nonumber \end{equation}

It is shown in \cite[Section 2.2]{BO2} that $(\delta^{S^1})^2=0 $, which is equivalent to the fact that $\delta=(\delta_0, \delta_1, \delta_2, \cdots)$ gives rise to an $S^1$-structure on $CF^*(M, H_t)$. The operations $\delta_0$ and $\delta_1$ agree with the usual differential $d$ and the BV operator $\Delta$ defined on symplectic cohomology $SH^*(M)$, respectively. \\
\indent Given admissible pairs $(H_t, J_t^+)$ and $(K_t, J_t^-)$ with $H_t \preceq K_t$, there is a map of $S^1$-complexes between $CF^*(M,H_t)$ and $CF^*(M,K_t)$. We choose a monotone homotopy $(H_{s,t}, J_{s,t})$ between $(H_t, J_t^+)$ and $(K_t, J_t^-)$. Using the same extension formula as in \eqref{eqn:H_N}, we can extend $H_{s,t}$ to a family of monotone homotopies parametrized by $S^{2N+1}$ 
\begin{equation}\label{eqn:ext}
 H_{N,s}\colon \R \times S^1 \times \widehat{M} \times S^{2N+1} \rightarrow \R  
\end{equation} 
such that $H_{N,s,t}^z:=H_{N,s}(t,\cdot, z)$ belongs to $\mathscr{H}(M)$ for all $z \in S^{2N+1}$. One can also choose $J_{N,s} \in \mathscr{J}_N^{S^1}$ for $s \in \R$ such that along each horizontal local slice $U_{z_0}$ one has that $J_{N,s,t}^z$ is independent of $z$ and agrees with the chosen homotopy of almost complex structures $J_{s,t} \in \mathscr{J}(M)$. Given $x_-=(\gamma_-, Z_j)$ and $x_+=(\gamma_+, Z_0)$ with $\Ind_M(Z_0)=0$ and $\Ind_M(Z_j)=2j$, we denote by 
\begin{equation}
\widetilde{\mathscr{M}}^{\kappa}_j(\gamma_-, \gamma_+):=\widetilde{\mathscr{M}}^{\kappa}_j(x_-, x_+,H_{N,s}, J_{N,s}, \widetilde{f}_N, \widetilde{g}_N) 
\nonumber \end{equation}
the moduli space of solutions $u\colon \R \times S^1 \rightarrow \widehat{M}$ and $z\colon \R \rightarrow S^{2N+1}$ to the system of equations
\begin{equation} \label{eqn:k_i}
\begin{cases}
\partial _s u+J^{z(s)}_{N,s,t}(u)(\partial _t u-X_{H_{N,s,t}^{z(s)}}(u))=0, \\
\dot{z}+\nabla \widetilde{f}_N(z)=0,
\end{cases}
\end{equation}
with asymptotic behaviors
\begin{equation}
\displaystyle\lim_{s \rightarrow - \infty}(u(s, \cdot), z(s)) \in  S^1 \cdot x_- ,\ \ \displaystyle\lim_{s \rightarrow \infty}(u(s, \cdot), z(s)) \in  S^1 \cdot x_+.
\nonumber \end{equation}
There is a free $S^1$-action on $\widetilde{\mathscr{M}}^{\kappa}_j(\gamma_-, \gamma_+)$ and we define $$\mathscr{M}^{\kappa}_j(\gamma_-, \gamma_+):=\widetilde{\mathscr{M}}^{\kappa}_j(\gamma_-, \gamma_+)/S^1.$$ For generic choice of data $(H_{N,s}, J_{N,s})$, the dimension of the moduli space $\mathscr{M}^{\kappa}_j(\gamma_-, \gamma_+)$ is given by
\begin{equation}
|x_0|-|x_1|=|\gamma_-|-|\gamma_+|+2j.
\nonumber \end{equation} 
For each $u \in \mathscr{M}^{\kappa}_j(\gamma_-, \gamma_+)$ with $|\gamma_-|=|\gamma_+|-2j$, there is an isomorphism
\begin{equation}
\kappa_{j,u}\colon o_{\gamma_+} \rightarrow o_{\gamma_-},
\nonumber \end{equation} 
by relatively orienting the parametrized moduli space $\mathscr{M}^{\kappa}_j(\gamma_-, \gamma_+).$
As $N \rightarrow \infty$, there are operations $\kappa_j\colon CF^*(M, H_t) \rightarrow CF^{*-2j}(M, K_t)$ defined for all $j \geq 0$ by
\begin{equation}  \label{defn:k_i}
\kappa_j|_{o_{\gamma_+}}=\bigoplus_{|\gamma_-|=|\gamma_+|-2j}\sum_{u \in \mathscr{M}^{\kappa}_j(\gamma_-, \gamma_+)} \kappa_{j,u}.
\end{equation}
By the definitions of $\delta_i$ and $\kappa_j$, one can verify that
\begin{equation}
\sum_{i+j=k} (\kappa_j \delta_i- \delta_i \kappa_j)=0 \text{ for all } k \geq 0.
\nonumber \end{equation}
For admissible Hamiltonians $H_t \preceq K_t$, we obtain a cochain map defined by
\begin{eqnarray} \label{defn:kappa}
&& \kappa\colon CF^*(M, H_t)[[u]] \rightarrow CF^*(M, K_t)[[u]],  \nonumber \\
&& \kappa=\kappa_0+u\kappa_1+u^2 \kappa_2 + \cdots.\nonumber
\end{eqnarray}

\subsection{Localized periodic symplectic cohomology}For an admissible Hamiltonian $H_t$ of slope $\tau \notin \mathscr{S}$, the periodic symplectic cochain complex of $H_t$ is defined to be
\begin{equation} 
CP^*(M,H_t):=CF^*(M, H_t)((u)).
\nonumber \end{equation}
We denote the cohomology of $(CP^*(M, H_t), \delta^{S^1})$ by $\hp(M, H_t)$.\\
\indent For any Hamiltonians $H_t \preceq K_t$, there is a cochain map 
\begin{equation}
\kappa\colon CF^*(M, H_t)((u)) \rightarrow CF^*(M, K_t)((u))
\nonumber \end{equation}
induced by \eqref{defn:kappa} in the localization. The induced map in the homology 
\begin{equation} \label{eqn:pshcon}
 \kappa_*\colon \hp(M, H_t)\rightarrow \hp(M, K_t)
\end{equation} 
is well-defined and independent of the choice of the monotone homotopy between $H_t$ and $K_t$. The \textit{localized periodic symplectic cohomology} is defined to be
\begin{equation}
\hpl(M):= \varinjlim \hp(M, H_t),
\nonumber \end{equation} 
where the direct limit is taken over the pre-ordered set of all admissible Hamiltonians $H_t$ in $\mathscr{H}(M)$ with respect to \eqref{eqn:pshcon}. By definition, the localized periodic symplectic cohomology $\hpl(M)$ is an invariant of the completion $\widehat{M}$ up to Liouville isomorphisms. We denote $\hpl(M) \otimes_{\Z} \Q$ by $\hpl(M, \Q)$. It will be shown that $\hpl(M, \Q) \cong H^*(M, \Q) ((u))$ in section \ref{sec:local}.  

\begin{rmk}
Suppose that $2c_1(M)=0$, we can equip the Floer cochain complex $CF^*(M, H_t)$ with a $\Z$-grading. Since there are only finitely many generators of $CF^*(M, H_t)$, the degrees of the generators are bounded from above and below, and the $S^1$-equivariant differential $\delta^{S^1}$ on each $CF^*(M, H_t)$ has only finitely many non-zero terms by degree reasons. This implies that the (refined) periodic cyclic homology of the $S^1$-complex $(CF^*(M,H_t), \delta)$ defined in Remark \ref{rmk:refined} satisfies
\begin{equation}
CF^*(M,H_t)((u))^{gr} \cong CF^*(M, H_t)[u,u^{-1}]
\nonumber \end{equation}
for any admissible Hamiltonian $H_t \in \mathscr{H}(M).$ As a direct limit of $\Z[u,u^{-1}]$-modules, one obtains that the (refined) localized periodic symplectic cohomology
\begin{equation}
HP^*_{S^1,loc}(M):=\varinjlim H^*(CF^*(M,H_t)((u))^{gr}, \delta^{S^1})
\nonumber \end{equation}
is a $\Z[u,u^{-1}]$-module when $2c_1(M)=0.$ The corresponding localization Theorem states that $HP^*_{S^1,loc}(M) \cong H^*(M, \Q)[u,u^{-1}]$. One can compare this with the main theorem in \cite{ACF}. However, for a general Liouville domain $M$, the localized periodic symplectic cohomology $HP^*_{S^1,loc}(M)$ is only $\Z/2$-graded. So we keep the notation of the $u$-adic completion in the statement of Theorem \ref{thm:1.1} as $\hpl(M, \Q) \cong H^*(M, \Q) ((u))$ .
\end{rmk} 

\section{A Convenient Complex for computations}\label{sec:concx}
In this section, we restrict ourselves to a special class of admissible Hamiltonians which are autonomous away from neighborhoods of its $1$-periodic orbits. Explicitly, for $\tau \notin \mathscr{S}$ we consider an autonomous Hamiltonian shown in Figure \ref{fig:1}, which is of the form 
\begin{equation} \label{eqn:autham}
\widehat{H}^{\tau}(x) =
\begin{cases}
f(x), & \text{if } x \in \Int(M);\\
\frac{(R-1)^2}{2}, & \text{if } x=(r,y) \subset [1, \tau +1] \times \partial M;\\
\tau (R-1)-\frac{\tau^2}{2}, & \text{if }  x=(r,y) \subset [\tau +1, \infty) \times \partial M,
\end{cases}
\end{equation} 
where $f(x)$ is a negative $C^2$-small Morse function in the interior of $M$. To obtain a smooth Hamiltonian, we define the value of the Hamiltonian $\widehat{H}^{\tau}$ on the collar neighborhood $[1-\epsilon_0,1] \times \partial M$ to be $\rho(R)f(x)$, where $\rho(R)$ is a smooth cut-off function which equal to $1$ at $\{ 1-\epsilon_0\} \times \partial M$ and zero at $\{1\} \times \partial M$ for $\epsilon_0$ sufficiently small. One notices that $\widehat{H}^{\tau}$ is only $C^0$ at $R=\tau+1$, so we need to modify the value of $\widehat{H}^{\tau}$ on ${[\tau+1, \tau+1+\epsilon_0]\times \partial M}$ such that 
\begin{equation}
\frac{d \widehat{H}^{\tau}(R)}{d R} =
\begin{cases}
R-1, & \text{if } R \in [1, \tau +1]; \\
g(R), & \text{if } R \in [\tau+1, \tau+1+\epsilon_0];\\
\tau, & \text{if } R \in [\tau +1 +\epsilon_0, \infty),
\end{cases}
\nonumber \end{equation}
where $g(R)$ is a smooth function whose values agree with $R-1$ and $\tau$ to all orders at $R=\tau +1$ and $R=\tau +1 +\epsilon_0$ respectively. On $[\tau+1, \tau+1+\epsilon_0]$, the function $g(R)$ satisfies 
\begin{equation}
\int_{\tau+1}^{\tau+1+\epsilon_0}g(R) dR=\tau+\epsilon_0.
\nonumber \end{equation}
We denote the resulting smooth Hamiltonian by $H^{\tau}$.\\
\indent For the autonomous Hamiltonian $H^{\tau}$ chosen above, all $1$-periodic orbits $\gamma$ of $H^{\tau}$ are transversely nondegenerate, that is, the linearized return map $d\psi_{X_{H^{\tau}}}^1$ when restricting to the contact distribution $\xi,$ $d\psi_{X_{H^{\tau}}}^1|_{\xi} \colon \xi_{\gamma(0)} \rightarrow \xi_{\gamma(0)},$ does not have $1$ as its eigenvalue. There is a non-trivial $S^1$-action on transverse nondegenerate $1$-periodic orbits by 
$$S^1 \times \mathscr{P}(H^{\tau}) \rightarrow \mathscr{P}(H^{\tau}), (s, \gamma(t)) \mapsto \gamma(t+s).$$ 
We can break the $S^1$-symmetry by choosing small perturbations of the Hamiltonian $H^{\tau}$ in some isolated neighborhood $N(\gamma)$ of each orbit $\gamma$ as follows. Let $U= \bigcup_{\gamma \in \mathscr{P}(H^{\tau})} N(\gamma)$ and define the space of time-dependent perturbations of $H^{\tau}$ by
\begin{equation} \label{eqn:H_per}
\mathscr{H}^{per}:=\{ h \in C^{\infty}(S^1 \times U, \R) \mathbin{|} |\nabla h(t,x)| \leq 1, \ \forall (t,x)\in S^1 \times U \},
\end{equation}
where $|\cdot|$ is the norm  with respect to the metric $\langle \cdot , \cdot \rangle=\widehat{\omega}(\cdot,J_t\cdot)$ for some almost complex structure $J_t$ on $U.$ 
For a fixed 1-periodic orbit $\gamma$ of $H^{\tau}$ that corresponds to a Reeb orbit of multiplicity $k$, one can choose a Morse function $h_0\colon S^1 \rightarrow [0, \frac{1}{2}]$ that has one minimum value $0$ at $0 \in S^1$ and one maximum value $1/2$ at $c_0$ for some small enough $c_0 \in S^1$. We define 
\begin{equation}
h_{\gamma}(t)\colon N(\gamma) \rightarrow [0,1], \ \ h_{\gamma}(t)(\gamma(s))=h_0(ks-kt) \text{ for } s \in S^1,
\nonumber \end{equation}
and smoothly extend it to $\overline{N(\gamma)}$ such that $h_{\gamma}(t)|_{\partial \overline{N(\gamma)}} \equiv 0$. We define an explicit time-dependent perturbation $H^{\tau}_{\epsilon}(t)$ of $H^{\tau}$ by
\begin{equation}\label{eqn:pb}
H^{\tau}_{\epsilon}(t) =H^{\tau}+\epsilon h_t^{\tau}:=H^{\tau}+\epsilon \sum_{\gamma \in \mathscr{P}(H^{\tau})}  h_{\gamma}(t)
\end{equation}
for some $\epsilon >0$ and $\mathrm{supp}(h_t^{\tau}) \subset \overline{U}.$ One observes that the perturbed Hamiltonian $H^{\tau}_{\epsilon}(t)$ has two $1$-periodic orbits in $N(\gamma)$ for each $\gamma \in \mathscr{P}(H^{\tau})$. It is proved in \cite[Proposition 2.2]{CFHW} that the new $1$-periodic orbits are translations of $\gamma(t)$ defined by
\begin{equation} \label{defn:gamma}
\widecheck{\gamma}(t):=\gamma(t),\ \ \  \widehat{\gamma}(t):=\gamma(t+\frac{c_0}{k}).
\end{equation} 
In fact, these orbits $\widehat{\gamma}(t)$ and $\widecheck{\gamma}(t)$ are the only $1$-periodic orbits of $H^{\tau}_{\epsilon}(t)$ in the neighborhood $N(\gamma)$ such that
\begin{equation}
|\widehat{\gamma}|=|\widecheck{\gamma}|+1.
\nonumber \end{equation}
In the subsequent discussions, we will fix an appropriate choice of $\epsilon_{i} >0$ depending on $i$ and $h_t^{\tau} \in \mathscr{H}^{per}$ in \eqref{eqn:pb} satisfying Lemma \ref{lem:preact} and Lemma \ref{lem:no_escape}. For simplicity of notations, we denote by $H_t^{\tau}$ the time-dependent perturbation $H^{\tau}_{\epsilon_{i}}(t)$ of the autonomous Hamiltonian $H^{\tau}$ if there is no further confusions.

\subsection{Relations with Reeb dynamics on $\partial M$}
Given a Liouville domain $(M, \theta)$, we recall that $(\partial M, \xi=\ker(\theta|_{\partial M}))$ is a contact manifold. Let $\psi^t_{\mathcal{R}_{\alpha}}$ be the flow of the Reeb vector field. Reeb orbits of period $l$ are smooth maps $x\colon \R/l\Z \rightarrow \partial M$ such that $\psi^t_{\mathcal{R}_{\alpha}}(x(0))=x(t)$. Such an orbit is simple if the map $x$ is injective. We further assume that all Liouville domains in consideration satisfy:
\begin{equation}\label{eqn:cond}
\parbox{28em}{distinct Reeb orbits  $x, x'$ have different periods $\int_{x} \alpha \neq \int_{x'} \alpha$.}
\end{equation}
This is a generic condition which can be achieved by perturbing the contact form in the neighborhood of each Reeb orbit as in Theorem 13 of \cite{ABW}.

\begin{definition}
Let $x$ be a simple Reeb orbit. Then $x$ is called hyperbolic if $d\psi^l_{\mathcal{R}_{\alpha}}\colon \xi_{x(0)} \rightarrow \xi_{x(0)}$ has no eigenvalues in the unit circle, otherwise it is called non-hyperbolic. If $d\psi^l_{\mathcal{R}_{\alpha}}$ has an odd number of eigenvalues belonging to $(-1,0)$, then $x$ is called a negative hyperbolic Reeb orbit.
\end{definition}

\begin{definition} \label{defn:bad}
Let $x$ be a simple Reeb orbit that is negative hyperbolic. Its $2k$-fold covers $x^{2k}$ are called bad Reeb orbits. A Reeb orbit is called good if it is not a bad Reeb orbit.
\end{definition}

\begin{rmk} \label{rmk:badcz}
There is another characterization of bad Reeb orbits appearing in \cite{EGH} on symplectic field theory, that is, a Reeb orbit $x^{2k}$ is called bad if it is a $2k$-fold cover of some simple orbit $x$ such that $\mu_{CZ}(x^{2k})-\mu_{CZ}(x)$ is odd. 
\end{rmk}

\indent For any autonomous Hamiltonian $H^{\tau}|_{[1,\infty)\times \partial M}=h^{\tau}(R)$ as in \eqref{eqn:autham}, we have 
\begin{equation}
X_{H^{\tau}}=h^{\tau}{}'(R) \cdot\mathcal{R}_{\alpha}=(R-1)\mathcal{R}_{\alpha} \text{ on } [1, \infty) \times \partial M.
\nonumber \end{equation} 
This implies that Reeb orbits on $\partial M$ of period $l$ with $l < \tau$ are in one-to-one correspondence to non-constant transversely nondegenerate $1$-periodic Hamiltonian orbits of $H^{\tau}$ under the map 
$$x(t) \mapsto \gamma(t)=(l+1, x(lt)).$$ 
It should be observed that for a $1$-periodic orbit $\gamma(t)=(l+1, x(lt))$ in $\{l+1\} \times \partial M \subset [1, \infty) \times \partial M$, the action of the Hamiltonian orbit $\gamma(t)$ has a simple form
\begin{equation} \label{eqn:autact}
\mathscr{A}_{H^{\tau}}(\gamma(t))= -\int_{S^1} \gamma^*(\widehat{\theta})+H^{\tau}(\gamma(t))dt=-R h^{\tau}{}'(R)+h^{\tau}(R)=-\frac{l^2}{2}-l.
\end{equation}
In particular, the action of $\gamma(t)=(l+1, x(lt))$ is given by the $y$-intercept of the tangent line to $h^{\tau}$ at $R=l+1$, which is exactly $-l^2/2-l$ in this case. We set 
\begin{equation}
a_l:=-\frac{l^2}{2}-l
\nonumber \end{equation}
in the following discussions.

\begin{figure}
\begin{center}
\includegraphics[scale=0.65]{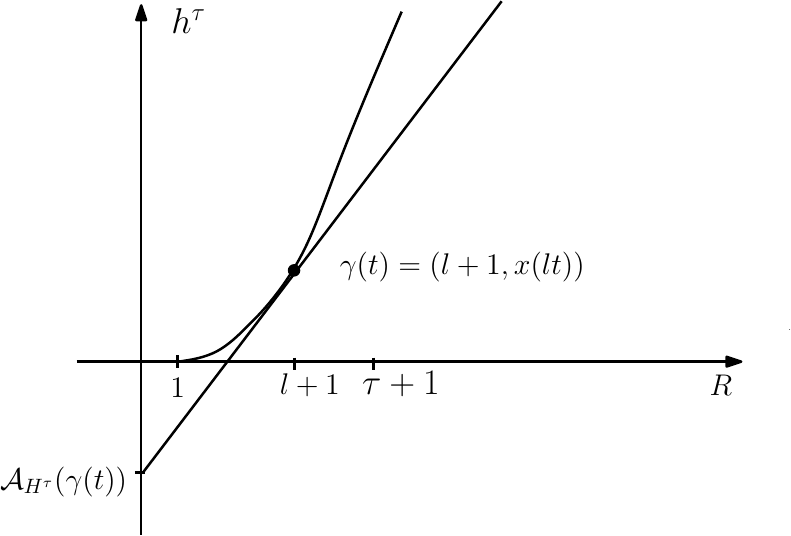}
\end{center}
\caption{Action of 1-periodic orbits of an autonomous function.}
\label{fig:1}
\end{figure}

For the perturbed Hamiltonian $H^{\tau}_t$, the actions of $1$-periodic orbits $\widehat{\gamma}$ and $\widecheck{\gamma}$ of period $l$ are given by 
\begin{eqnarray} \label{eqn:actclose}
&& \mathscr{A}_{H^{\tau}_t}(\widehat{\gamma})=\mathscr{A}_{H^{\tau}}(\gamma)+\epsilon \int_{S^1}h_{\gamma}\circ \psi_{H^{\tau}}^{-t}(\widehat{\gamma})dt =a_l+\epsilon h_0(c_0)=a_l+\frac{\epsilon}{2}, \\
&& \mathscr{A}_{H^{\tau}_t}(\widecheck{\gamma})=\mathscr{A}_{H^{\tau}}(\gamma)+\epsilon \int_{S^1}h_{\gamma}\circ \psi_{H^{\tau}}^{-t}(\widehat{\gamma})dt=a_l+\epsilon h_0(0) =a_l.\nonumber
\end{eqnarray}
This implies that for sufficiently small choices of $\epsilon>0$, the values of $\mathscr{A}_{H^{\tau}_t}(\widehat{\gamma})$ and $\mathscr{A}_{H^{\tau}_t}(\widecheck{\gamma})$ differ from that of $\mathscr{A}_{H^{\tau}}(\gamma)$ slightly. 

\subsection{The action filtration on the $S^1$-complex $CF^*(M,H^{\tau}_t)$}
Let $\mathscr{S}$ be the action spectrum of the Reeb orbits associated to the contact form $\alpha$ defined in \eqref{eqn:spectrum}. One can list all the elements of $\mathscr{S}$ in an increasing order $0=l_0 < l_1 < l_2 < \cdots$, where $l_1$ corresponds to the minimum period of a Reeb orbit. We choose $\tau_j \notin \mathscr{S}$ such that $\tau_j =\frac{l_j+l_{j+1}}{2}$. There is a natural action filtration on $CF^*(M, H^{\tau}_t)$ defined by
\begin{equation} \label{defn:af}
F^j CF^*(M,H^{\tau}_t)=\bigoplus_{\substack {\mathscr{A}_{H^{\tau}_t}(\gamma) \geq a_{\tau_j} \\ \gamma \in \mathscr{P}(H^{\tau}_t)}} \Z\langle o_{\gamma} \rangle.
\end{equation}
The following lemma implies that for specific perturbations of $H^{\tau}$ defined using \eqref{eqn:pb}, the action filtration is preserved by the $S^1$-structure. 
\begin{lem} \label{lem:preact}
For each $i \geq 0$, there is a constant $\epsilon_{i} >0$ which depends on $i$ and the time-dependent perturbation $h_t^{\tau} \in \mathscr{H}^{per}$ such that
\begin{equation}
\delta_i F^j CF^*(M,H^{\tau}_{\epsilon_{i}}) \subset  F^j CF^*(M,H^{\tau}_{\epsilon_{i}}) \text{ for all } j  \geq 0 \text{ and }\tau \notin \mathscr{S}. \nonumber
\nonumber \end{equation}
\end{lem}
\begin{proof}
Given an element $(\bar{u}, \bar{z}_i) \in \mathscr{M}_i(x_0, x_1)$ for $i\geq 0$, one can choose a fixed representative $(u,z_i) \in \widetilde{\mathscr{M}}_i(x_0, x_1)$ such that $u\colon \R \times S^1 \rightarrow \widehat{M}$ is a solution to \eqref{eqn:delta} that is asymptotic to $\gamma_0$ and $\gamma_1$ with $\mathscr{A}_{H^{\tau}_{t}} (\gamma_1)\geq a_{\tau_j}$, and $z_i \colon \R \rightarrow S^{2N+1}$ is a negative gradient flow line from $\theta_{Z_i} \cdot Z_i$ to $\theta_{Z_0} \cdot Z_0$ on $S^{2N+1}$ for some $\theta_{Z_0}$ and $\theta_{Z_i} \in S^1$. The energy of $u$ is given by
\begin{eqnarray}
 E(u)&=&\int_{\R \times S^1} ||\frac{\partial u}{\partial s}||^2 ds dt \label{eqn:energy}  \\
 & =&\mathscr{A}_{H^{\tau}_{t}}(\gamma_0)-\mathscr{A}_{H^{\tau}_{t}}(\gamma_1)+\int_{\R \times S^1} \big(\partial_s H_{N}^{z_i(s)}\big)(u(s,t)) ds dt.\nonumber
\end{eqnarray} 
We will prove that the last term in \eqref{eqn:energy} is bounded by $\epsilon \cdot i \cdot C$, where $C$ is some constant which depends on the choice of time-dependent perturbation $h_t^{\tau}$ in $\mathscr{H}^{per}$. For Hamiltonians of the forms $H^{\tau}_{t}=H^{\tau}+\epsilon h_t^{\tau}$ for the fixed time-dependent perturbation $h_t^{\tau} \in \mathscr{H}^{per}$ defined in \eqref{eqn:pb}, we extend it to a Hamiltonian $H_{\epsilon, N}^z \colon S^1 \times \widehat{M} \times S^{2N+1} \rightarrow \R$ by
\begin{equation} \label{eqn:extension}
H_{\epsilon, N}^z(x):= \beta(z)(H^{\tau}(x)+\epsilon h_{t-\theta(z)}(x))+(1-\beta(z))H^{\tau}(x),
\end{equation}
where $\beta(z)$ and $\theta(z):=\theta_z$ are defined in \eqref{eqn:H_N}. We remark that this extension is different from the extension defined in \eqref{eqn:H_N}. It gives a convenient admissible extension in $\mathscr{H}^{S^1}_N$ for the time-dependent perturbation of the autonomous Hamiltonian $H^{\tau}_{t}$ considered in \eqref{eqn:pb}. Given this explicit extension \eqref{eqn:extension}, we can compute that 
\begin{figure}
\begin{center}
\includegraphics[scale=0.75]{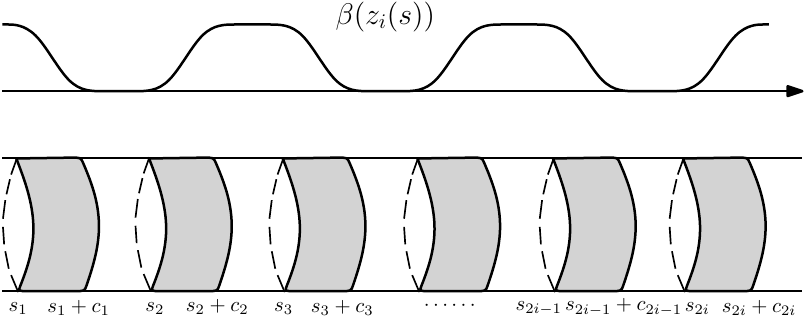}
\end{center}
\caption{The value of $\beta(z_i(s))$ and the support of $\partial_s \beta(z_i(s))$ shown as the shaded regions.}
\label{fig:3}

\end{figure}
\begin{equation}
 \int_{\R \times S^1} \big(\partial_s H_{\epsilon, N}^{z_i(s)}\big)(u(s,t)) ds dt =\int_{\R \times S^1} \big(\partial_s \beta(z_i(s)) \epsilon h_{t-\theta(z_i(s))}(u(s,t))ds dt, \nonumber
\end{equation}
where we have used the fact that the term $\beta(z_i(s)) \epsilon \partial_sh_{t-\theta(z_i(s))}(u)$ vanishes. This is due to the local triviality of the extension \eqref{eqn:extension} on $\bigcup_{z_0 \in \mathrm{Crit}(f_N)}U_{z_0}$, so $\partial_s h_{t-\theta(z_i(s))}=0$ for all $s$ satisfying $\mathrm{supp} \beta(z_i(s))$ in $\bigcup_{z_0 \in \mathrm{Crit}(f_N)}N(z_0)$. As $z_i$ is a Morse flow line from between the index zero and $2i$ critical points of the $S^1$-invariant function $\tilde{f}_N$ on $S^{2N+1}$, the support of the integrand $\partial_s \beta(z_i(s))\epsilon h_{t-\theta(z_i(s))}(u)$ is contained in the subset $\bigcup_{j=1}^{2i} I_j\times S^1$ of $\R \times S^1$ for some intervals $I_j=[s_j, s_j+c_j]$ shown in Figure \ref{fig:3}. This implies that the last term in \eqref{eqn:energy} satisfies
\begin{eqnarray} 
&&\ \ \ \ \  \int_{\R \times S^1} \big(\partial_s H_{\epsilon, N}^{z_i(s)}\big)(u(s,t)) ds dt \label{eqn:last_term} \\
&&\leq \sum_{j=1}^{2i}\int_{I_j} \partial_s \beta(z_i(s)) \int_{S^1} \epsilon h_{t-\theta(z_i(s))}(u(s,t))ds dt \nonumber \\
&& \leq \sum_{j=1}^{i} \epsilon \Big( \int_{I_{2j}} \partial_s \beta(z_i(s)) \int_{S^1} \max_{(x,t)}  h_t^{\tau}(x)dt ds\nonumber \\
&& +\int_{I_{2j-1}} \partial_s \beta(z_i(s)) \int_{S^1} \min_{(x,t)} h_t^{\tau}(x)dt ds \Big) \nonumber \\
&& \leq i \epsilon \cdot  \Big| \int_{S^1} \max_{(x,t)} h_t^{\tau}(x)dt - \int_{S^1} \min_{(x,t)} h_t^{\tau}(x)dt \Big| \nonumber,
\end{eqnarray}
where the last inequality in \eqref{eqn:last_term} is derived from the following facts
\begin{equation}
\int_{I_{2j-1}} \partial_s \beta(z_i(s)) ds \in [ -1, 0] \text{ and } \int_{I_{2j}} \partial_s \beta(z_i(s)) ds \in [0, 1], \ \forall j=1,\cdots,i.
\nonumber \end{equation}
Equation \eqref{eqn:last_term} shows that the last term can be bounded by a constant which is independent of the solution $(u, z_i) \in \widetilde{\mathscr{M}}_i(x_0,x_1)$. So we have an a priori energy estimate for the energy of solutions to equation \eqref{eqn:delta_i} for a fixed $i \geq 0$.\\
\indent The estimate \eqref{eqn:last_term} shows that for a fixed choice of time-dependent perturbation $h_t^{\tau} \in \mathscr{H}^{per}$, one can choose $\epsilon_i$ sufficiently small for each $i \geq 0$ depending on $i$ and the constant
\begin{equation}
C:=\int_{S^1} \max_{(x,t)} h_t^{\tau}(x)dt - \int_{S^1} \min_{(x,t)} h_t^{\tau}(x)dt
\nonumber \end{equation}
such that $\mathscr{A}_{H^{\tau}_{\epsilon_{i}}}(\gamma_0) \geq \mathscr{A}_{H^{\tau}_{\epsilon_{i}}}(\gamma_1)$. It remains to be verified that for each $i \geq 0$ there exist consistent choices of $\epsilon_1, \epsilon_2, \cdots, \epsilon_i$ so that the Floer data in defining the operations $\delta_0, \delta_1,\cdots, \delta_i$ are compatible with compactness of $\mathscr{M}_i(x_0, x_1)$. For a fixed $i \geq 0,$ one can first choose $\epsilon_i$ small enough such that
\begin{eqnarray}
&& \int_{\R \times S^1} \big(\partial_s H_{\epsilon_i, N}^{z_i(s)}\big)(u) ds dt < i\cdot\epsilon_iC \text{ and } \nonumber \\
&& \mathscr{A}_{H^{\tau}_{\epsilon_{i}}}(\gamma_0) > \mathscr{A}_{H^{\tau}_{\epsilon_{i}}}(\gamma_1)-i\cdot\epsilon_i\cdot C \geq a_{\tau_j},
\nonumber
\end{eqnarray}
for all $u \in \mathscr{M}_{i}(x_0, x_1)$ and for all $x_0, x_1 \in \mathrm{Crit}(\mathscr{A}_{H^{\tau}_{\epsilon_{i}}})$. Then we inductively choose $\epsilon_1, \cdots, \epsilon_{i-1}$ such that for any sequence $(u_j)_{j\in \N} \in \mathscr{M}_{i}(x_0,x_1)$ converging to some broken configuration $((v_1,\tilde{z}_{i_1}),\cdots, (v_n,\tilde{z}_{i_n}))$ in the boundary of the Gromov compactification $\partial  \overline{\mathscr{M}_i}(x_0, x_1)$
\begin{equation}
\bigcup_{n}\bigcup_{\Sigma_{k=1}^n i_k=i}\bigcup_{y_{i_1}, \cdots, y_{i_n}} \mathscr{M}_{i_1}(x_0, y_{i_1}) \times \mathscr{M}_{i_2}(y_{i_1}, y_{i_2})\times \cdots \times
\mathscr{M}_{i_n}(y_{i_n}, x_1),
\nonumber \end{equation}
the following condition is satisfied
\begin{equation}
\sum_{k=1}^n \int_{\R \times S^1} \big(\partial_s H_{\epsilon_k, N}^{\tilde{z}_k(s)}\big)(v_k) ds dt= \int_{\R \times S^1} \big(\partial_s H_{\epsilon_i, N}^{z_i(s)}\big)(u_j) ds dt <i \epsilon_i C
\nonumber \end{equation}
for all $j \gg 0$ and $n \geq 1.$
Such choices can be made inductively, since for each fixed $i \geq 0$ there are only finitely many $\epsilon_k >0$ that one needs  to choose to define the operation $\delta_k$ for $k=0,\cdots, i$ so that the condition $\sum_{k+l=i}\delta_k\delta_l=0$ is satisfied. This completes the proof.
\end{proof}

We have shown that for each $i \geq 0$ one can choose $\epsilon_i$ so that the operation $\delta_i$ defined by the Floer data $(H^z_{\epsilon_i, N}, J_N)$ preserves the action filtration \eqref{defn:af} defined on $CF^*(M, H^{\tau}_{\epsilon_i})$ for all $\tau \notin \mathscr{S}$. For $\tau_+ < \tau_-$, we denote by $(H_{s,t}, J_{s,t})$ the monotone homotopy between the Floer data $(H^{\tau_+}_{\epsilon_i}, J^+)$ and $(H^{\tau_-}_{\epsilon_i}, J^-).$ By definition, the Hamiltonian $H_{s,t}$ satisfies
\begin{equation} \label{eqn:monotone_homo}
H_{s,t}=H_s+\epsilon_i h_{s,t},
\end{equation}
where $H_s$ is the monotone homotopy between the autonomous functions $H^{\tau_+}$ and $H^{\tau_-}$ and $h_{s,t} \in \mathscr{H}^{per}$ is the monotone homotopy between the time-dependent perturbations $h_t^{\tau_+}$ and $h_t^{\tau_-}$ defined in \eqref{eqn:pb}. Lemma \ref{lem:prekappa} below will show that for each fixed $i \geq 0,$ one can choose the same time-dependent perturbations $H^{\tau^+}_{\epsilon_i}$ and  $H^{\tau^-}_{\epsilon_i}$ for $\tau_+, \tau_-$ so that the continuation maps associated to the monotone homotopy $H_{s,t}$
\begin{equation} \label{eqn:comp}
\kappa_i\colon CF^*(M, H^{\tau_+}_{\epsilon_{i}}) \rightarrow CF^{*-2i}(M, H^{\tau_-}_{\epsilon_{i}})
\end{equation}
preserves the action filtration. Furthermore, we have that 
\begin{equation}
\sum_{k+l=i}\kappa_k \delta_l-\delta_l \kappa_k=0, \ \forall i\geq 0.
\nonumber \end{equation}

\begin{lem}\label{lem:prekappa}
For each $i \geq 0$, there is a constant $\epsilon_i$ such that the operation $\kappa_i\colon CF^*(M, H^{\tau_+}_{\epsilon_{i}}) \rightarrow CF^{*-2i}(M, H^{\tau_-}_{\epsilon_{i}})$ satisfying 
\begin{equation}
\kappa_i\big( F^j CF^*(M, H^{\tau_+}_{\epsilon_{i}}) \big)\subset F^j CF^*(M, H^{\tau_-}_{\epsilon_{i}}) \text{ for all } j \geq 0 \text{ and } \tau_+ < \tau_-. \nonumber
\nonumber \end{equation}
\end{lem}

\begin{proof}
Let $\widetilde{\mathscr{M}}^{\kappa}_i(\gamma_-,\gamma_+)$ be the moduli space of solutions to \eqref{eqn:k_i} with Floer data $(H_{N,s,t}^z, J_{N,s,t}^z)$, where $H_{N,s,t}^z$ is the extension of the monotone homotopy $H_{s,t}$ in equation \eqref{eqn:monotone_homo} to $\R \times S^1 \times \widehat{M} \times S^{2N+1}$ given explicitly by
\begin{equation}
H_{N,s,t}^{z,\epsilon_i}(x)=\beta(z)(H_s(x)+\epsilon_i h_{s, t-\theta(z)}(x))+(1-\beta(z))H_{s}(x).
\nonumber \end{equation} 
For any generator $o_{\gamma_+}$ in $F^j CF^*(M,H^{\tau_+}_{\epsilon_{i}})$, we have that $\mathscr{A}_{H^{\tau_+}_{\epsilon_{i}}}(\gamma_+) \geq a_{\tau_j}$. The energy of $u$ in $\widetilde{\mathscr{M}}^{\kappa}_i(\gamma_-, \gamma_+)$ is given by
\begin{eqnarray}
E(u) &=& \mathscr{A}_{H^{\tau_-}_{\epsilon_i}}(\gamma_-)-\mathscr{A}_{H^{\tau_+}_{\epsilon_i}}(\gamma_+)\label{eqn:cenergy} \\
&+& \int_{\R \times S^1} \nabla_{\dot{z}_i(s)} H_{N,s,t}^{z_i,\epsilon_i}(u) ds dt +\int_{\R \times S^1} \partial_s H^{z_i,\epsilon_i}_{N,s,t}(u) dsdt .   \nonumber
\end{eqnarray}
Since $\partial_s H_{N,s,t}^{z,\epsilon_i} \leq 0$, equations \eqref{eqn:last_term} and \eqref{eqn:cenergy} give that 
\begin{eqnarray}
&& \mathscr{A}_{H^{\tau_-}_{\epsilon_i}}(\gamma_-) \geq \mathscr{A}_{H^{\tau_+}_{\epsilon_i}}(\gamma_+)-\int_{\R \times S^1} \nabla_{\dot{z}_i(s)} H_{N,s,t}^{z_i,\epsilon_i}(u) ds dt\nonumber \\
&&\mathscr{A}_{H^{\tau_-}_{\epsilon_i}}(\gamma_-) \geq a_{\tau_j} -i\cdot \epsilon_{i}C. \nonumber
\end{eqnarray} 
Since $i\epsilon_iC$ was taken sufficiently small as in Lemma \ref{lem:preact}, by discreteness of the action spectrum, we conclude that 
\begin{equation}
\kappa_i F^j CF^*(M, H^{\tau_+}_{\epsilon_{i}}) \subset F^j CF^*(M, H^{\tau_-}_{\epsilon_{i}}) \text{ for all } i, j
\nonumber \end{equation}
as desired.
\end{proof}
For simplicity of notations, for each $i \geq 0$ we fix a consistent choices of constants $\epsilon_1, \epsilon_2, \cdots, \epsilon_i,$ whose values depend on $i$, and define the action-preserving operations $\delta_i$ and $\kappa_i$ as in Lemmata \ref{lem:preact} and \ref{lem:prekappa}. By abuse of notations, we will suppress the choices of $\epsilon_i$ and denote the operations simply by
\begin{eqnarray}
&& \delta_i \colon CF^*(M, H^{\tau}_t) \rightarrow CF^{*+1-2i}(M, H^{\tau}_t)\nonumber \\
&& \kappa_i \colon CF^*(M, H^{\tau_+}_t) \rightarrow CF^{*-2i}(M, H^{\tau_-}_t)\nonumber
\end{eqnarray}
in the subsequent discussions.

\subsection{A convenient complex to compute $\hpl(M)$}
For $\tau_i \notin \mathscr{S}$ and $\tau_i \rightarrow \infty$ as $i\rightarrow \infty$, we choose a cofinal system of Hamiltonians $H^{\tau_i}_t$ with the specific choice of $\epsilon_{\tau_i}$ provided by Lemma \ref{lem:preact}, and define the localized periodic symplectic cohomology to be
\begin{equation} \label{eqn:psh}
\varinjlim_{i} \hp(M, H^{\tau_i}_t),
\end{equation}
where the direct limit is taken with respect to the continuation maps 
\begin{equation}
(\kappa_{i, i+1})_* \colon \hp(M, H^{\tau_i}_t) \rightarrow  \hp(M, H^{\tau_{i+1}}_t) \text{ for all }i \geq 0.
\nonumber \end{equation}
The following Lemma then implies that one can instead use the cofinal system of Hamiltonians $\{H^{\tau_i}_t\}_{i \geq 0}$ to compute $\hpl(M)$.  

\begin{lem}\label{lem:lim}
Given any cofinal system of Hamiltonians $H^{\tau_i}_t$ with $\tau_i \rightarrow  \infty$, there is a natural map $$\varinjlim_{i}  \hp(M, H^{\tau_i}_t) \rightarrow \hpl(M)$$ which induces an isomorphism.
\end{lem}

\begin{proof}
The natural inclusion $\bigoplus\limits_i \hp(M, H^{\tau_i}_t) \hookrightarrow \bigoplus\limits_{H_t \text{ admissible}} \hp(M, H_t)$ induces a map of the quotients
$i\colon \varinjlim \hp(M, H^{\tau_i}_t) \rightarrow \hpl(M)$. This map $i$ is surjective because for any admissible Hamiltonian $H_t$ there is a $\tau_i$ large enough such that $H_t \preceq H^{\tau_i}_t$, and we obtain a map via continuation
\begin{equation}
\hp(M, H_t) \rightarrow \hp(M, H^{\tau_i}_t).
\nonumber \end{equation}
Similarly, the map $i$ is an inclusion since for any admissible Hamiltonians $H_t$ and $K_t$ with $H_t \preceq K_t$, there is an autonomous Hamiltonian $H^{\tau_j}_t$ such that $H_t \preceq K_t \preceq H^{\tau_j}_t$ for some $\tau_j$ large enough. The appropriate continuation maps then give rise to a commutative diagram
\[\xymatrix{
\hp(M, H_t) \ar[r] \ar[rd] &
\hp(M, K_t) \ar[d] \\
& \hp(M, H^{\tau_j}_t)}
\]
This implies that the natural map 
\begin{equation}
i\colon \varinjlim \hp(M, H^{\tau_i}_t) \rightarrow \hpl(M)
\nonumber \end{equation}
is an isomorphism.
\end{proof}
This formulation of $\hpl(M)$ in Lemma \ref{lem:lim} will be used in the next section.

\section{The Localization Theorem }\label{sec:local}
We prove the following main result in this section, which can be seen as the localization theorem for $\hpl(M)$.
\begin{thm} \label{thm:local}
Given a Liouville domain $(M, \theta)$, the natural inclusion of the constant loops $\iota\colon \widehat{M} \hookrightarrow \mathscr{L}\widehat{M}$ induces a natural map $$\iota_*\colon H^*(M)((u)) \rightarrow \hpl(M),$$ which is an isomorphism as $\Z/2$-graded $\Q((u))$-modules after tensoring both sides by $\Q$
\end{thm}
In particular, Theorem \ref{thm:local} holds for the large class of Weinstein manifolds, which includes affine varieties. \\
\indent To prove Theorem \ref{thm:local}, we first deal with autonomous Hamiltonian $H^{\tau}$ with time-dependent perturbations $H^{\tau}_t$ as in \eqref{eqn:pb}. For each $\gamma \in \mathscr{P}(H^{\tau})$ there is an isolated neighborhood $N(\gamma)$, which only contains two $1$-periodic orbits $\widehat{\gamma}, \widecheck{\gamma}$ of $H^{\tau}_t$ defined as in \eqref{defn:gamma}. One needs to compute the local contributions to the equivariant Floer differential $\delta^{S^1}$ between $o_{\widehat{\gamma}}$ and $o_{\widecheck{\gamma}}$, that is,  $d_0:=\delta^{S^1}|_{\Z\langle o_{\widehat{\gamma}},o_{\widecheck{\gamma}}\rangle \otimes_{\Z} \Z((u))}$. Following the original approach of Floer and Hofer in \cite{FH}, we choose a coherent orientation in the proof of Proposition \ref{prop:1} below. This means that the trivializations of $o_{\widehat{\gamma}}, o_{\widecheck{\gamma}}$  are fixed for all $\gamma \in \mathscr{P}(H^{\tau})$ and they are compatible with the gluing operation in the sense of \cite[Definition 11]{FH}. We denote the preferred generators of $o_{\widehat{\gamma}}$ and $o_{\widecheck{\gamma}}$ by $\widehat{\gamma}$ and $\widecheck{\gamma}$, respectively. The advantage of this approach is that the equivariant differential $\delta^{S^1}$ can be expressed explicitly as follow.

\begin{prop} \label{prop:1}
Let $\gamma$ be a $1$-periodic orbit of $H^{\tau}$ and 
$$d_0:=\delta^{S^1}|_{\Z \langle o_{\widehat{\gamma}},o_{\widecheck{\gamma}}\rangle \otimes_{\Z} \Z((u))}.$$ If $\gamma$ corresponds to a $k$-fold Reeb orbit, then we have
$$d_0 (\widecheck{\gamma} )= 0, \ \ d_0 (\widehat{\gamma})=\pm ku \widecheck{\gamma}$$ if the Reeb orbit is good, and $$d_0 (\widecheck{\gamma} )=\pm 2\widehat{\gamma},\ \ \ d_0 (\widehat{\gamma})=0$$ if the Reeb orbit is bad.
\end{prop}

It is shown in \cite[Proposition 2.2]{CFHW} that there are only two solutions to the Floer equation that are asymptotic to $1$-periodic orbits $\widehat{\gamma}$ and $\widecheck{\gamma}$ of $H^{\tau}_t$ at $\pm \infty$. We will first prove that all solutions to equation \eqref{eqn:delta} with asymptotic conditions on $\widehat{\gamma}$ or $\widecheck{\gamma}$ do not leave some tubular neighborhood of $\gamma$, which implies that the local $S^1$-equivariant Floer cohomologies are well-defined.

\begin{lem} \label{lem:no_escape}
Assume that $\gamma$ is a transversely nondegenerate $1$-periodic orbit of $X_{H^{\tau}}$ for some $\tau \notin \mathscr{S}$ and $U$ is an open neighborhood of $\gamma(S^1)$ which does not contain any other $1$-periodic orbits of $X_{H^{\tau}}.$ Then for any open neighborhood $V$ satisfying $\gamma(S^1) \subset V \subset U$, there exists a real number $\epsilon_0 >0$ such that for any $\epsilon \in (0, \epsilon_0)$ we have 
\begin{enumerate}
\item[(1)] All $1$-periodic orbits of $X_{H^{\tau}_{\epsilon, N}}$ in $U$ are contained in $V$. 
\item[(2)] All solutions $u \in \widetilde{\mathscr{M}}_i(x_0,x_1, H^{\tau}_{\epsilon, N},  J_{N}, U)$ are contained in $V$, $\forall i, N\in \N,$
where $H^{\tau}_{\epsilon, N}$ is defined in equation \eqref{eqn:extension}.
\end{enumerate}
\end{lem}

\begin{proof}
Part (1) and the case that $i=0$ of part (2) are proved in Lemma 2.1 of \cite{CFHW}. Let $u \in \widetilde{\mathscr{M}}_i(x_0,x_1,  H^{\tau}_{\epsilon, N},  J_{N}, U)$ for $i \geq 1.$ Suppose by contradiction that we can find a neighborhood $V$ of $\gamma(S^1)$ and sequences $\epsilon_n \rightarrow 0$ and parametrized Floer trajectories $u_n \in \widetilde{\mathscr{M}}_i(x_0,x_1,  H^{\tau}_{\epsilon_n, N},  J_{N}, U)$ such that $u_n$ is not contained in $V$ for each $n.$ By the energy estimate in \eqref{eqn:energy}, we have that for $\gamma_0,\gamma_1 \in\{\widehat{\gamma}, \widecheck{\gamma} \}$ the energy of $u_n$ satisfies
\begin{equation}
E(u_n)=\mathscr{A}_{H^{\tau}_{\epsilon_n}}(\gamma_0)-\mathscr{A}_{H^{\tau}_{\epsilon_n}}(\gamma_1)+\int_{\R \times S^1}(\partial_s H^{z_i(s)}_N)(u_n)dsdt.
\nonumber \end{equation}
By Gromov compactness, the sequence $u_n$ converges in $C^{\infty}_{\mathrm{loc}}$ to some $u$ as $n \rightarrow \infty$ and $\epsilon_n \rightarrow 0.$ It is shown in part (1) of Lemma 2.1 in \cite{CFHW} that as $\epsilon_n \rightarrow 0$, all $1$-periodic solution $\gamma_0(t)$ and $\gamma_1(t)$ of $H^{\tau}_{\epsilon_n}$ converges to $S^1$-translates of $\gamma$, 
$$\gamma(t+a_n) \text{ and }\gamma(t+b_n) \text{ for some }a_n, b_n \in S^1.$$ 
This implies that 
$$\mathscr{A}_{H^{\tau}_{\epsilon_n}}(\gamma_0)-\mathscr{A}_{H^{\tau}_{\epsilon_n}}(\gamma_1) \rightarrow 0,$$ 
since $\mathscr{A}_{H^{\tau}_{\epsilon_n}}(\gamma(t+a_n))-\mathscr{A}_{H^{\tau}_{\epsilon_n}}(\gamma(t+b_n)) \rightarrow 0$ as $n \rightarrow \infty$. Similarly, one also has that 
$$\int_{\R \times S^1}(\partial_s H^{z_i(s)}_{\epsilon_n, N})(u)dsdt \rightarrow 0,$$ 
since $H^{\tau}_{\epsilon_n}$ approaches to the autonomous function $H^{\tau}$ as $\epsilon_n \rightarrow 0.$ Hence the energy $E(u_n) \rightarrow 0$ as $n \rightarrow \infty.$ As the energy of the limit $u$ of $\{u_n\}_{n \geq 0}$ satisfies 
$$E(u)=\int_{\R \times S^1}||\partial_s u||dsdt=0,$$
we conclude that the parametrized Floer solution $u$ is independent of the variable $s,$ and we can write $u(s,t)=\gamma(t+a)$ for some $a \in S^1.$ However, by assumption that $u_n$ is not contained in $V$ for each $n,$ which implies that there exists $s_n \in \R$ such that $u_n(s_n,t)$ is not contained in $V.$ As $u_n$ is asymptotic to $\widehat{\gamma}$ and $\widecheck{\gamma}$ which satisfy $|\widehat{\gamma}|=|\widecheck{\gamma}|+1$ and there are no other $1$-periodic orbits in $V,$ the moduli space $\mathscr{M}_i(x_0,x_1, H^{\tau}_{\epsilon, N},  J_{N}, U)$ is compact by index reasons and exactness of $\widehat{\omega}$. Thus the sequence $u_n$ of Floer trajectories does not break in the limit and the sequence $s_n$ of real numbers converges to some $s_{\infty} \in \R.$ This gives a contradiction since $u(s_{\infty},t)$ is not contained in $V$ whereas we have shown that $u(s,t)=\gamma(t+a) \subset V.$ This completes the proof for part (2). 
\end{proof}
\indent One consequence of Lemma \ref{lem:no_escape} is that the local $S^1$-equivariant Floer cohomology can be defined since the moduli space $\mathscr{M}_i(x_0,x_1, H^{\tau}_{N}, J_{N}, U)$ is compact. Moreover, by the usual continuation arguments, the local $S^1$-equivariant Floer cohomology is independent of the choices of almost complex structure $J_N$ and the time-dependent perturbations $h_{\gamma}(t) \in \mathscr{H}^{per}$ that we choose in \eqref{eqn:pb}. One can then compute the local $S^1$-equivariant Floer cohomology between $\widehat{\gamma}$ and $\widecheck{\gamma}$ as follows. Locally in a neighborhood $N(\gamma)$, a $\Z$-grading always exists due to that fact $2c_1=0$ in $N(\gamma)$. One can conclude by degree reasons that the only operations that can be non-trivial is the Floer differential $d=\delta_0$ and the BV operator $\Delta=\delta_1$. It is then sufficient to prove the following statements
\begin{equation}\label{eqn:check_hat}
d(\widecheck{\gamma})=
\begin{cases}
 0, &\text{if } \gamma \text{ is good}\\
 \pm 2 \widehat{\gamma}, &\text{if } \gamma \text{ is bad}
\end{cases},
\Delta(\widecheck{\gamma})=
\begin{cases}
 k \widehat{\gamma}, &\text{if } \gamma \text{ is good}\\
 0, &\text{if } \gamma \text{ is bad}.
\end{cases}
\end{equation}
In order to obtain this computation, we need the following Lemmata \ref{lem:ksol}, \ref{lem:d} and \ref{lem:delta}.
\begin{lem} \label{lem:ksol}
For each $1$-periodic orbit $\gamma \in \mathscr{P}(H^{\tau})$ that corresponds to a $k$-fold Reeb orbit such that $|\widehat{\gamma}|=|\widecheck{\gamma}|+1$, the zero dimensional manifold $\mathscr{M}_{\Delta}(\widecheck{\gamma}, \widehat{\gamma})$ consists of $k$ points.
\end{lem} 

\begin{proof}
One chooses a trivialization $\xi\colon N(\gamma) \rightarrow S^1 \times \R^{2n-1}$ such that 
\begin{equation}
\xi^*(\omega_0)=\widehat{\omega} \text{ and } \xi(\widehat{\gamma})=(kt+c_0,0), \xi(\widecheck{\gamma})=(kt,0), \nonumber
\end{equation}
where $\omega_0=\sum_{i=1}^n dp_i\wedge dq_i$ is the standard symplectic form on $S^1 \times
\R^{2n-1}$ with coordinates given by $(q_1, p_1, \cdots q_n, p_n).$ In this trivialization, we consider the Hamiltonian defined by 
\begin{equation}
K\colon S^1 \times \R^{2n-1} \rightarrow \R \text{ such that }K(q_1, p_1, \cdots)=-kp_1.
\nonumber \end{equation} 
The associated Hamiltonian flow is given by $\psi_{K}^t(q_1, y)=(q_1-kt, y)$. The Hamiltonian that generates $\psi^t_K \circ \psi^t_{H^{\tau}_{\epsilon}}$ is 
$$\widetilde{H}_{\epsilon}^{\tau}(t,x):=H^{\tau}_{\epsilon}(t, \psi^{-t}_K(x))+K(x).$$
After composing with the flow $\psi_{K}^t$, the time-dependent 1-periodic orbits of $H^{\tau}_t$ can be transferred into time-independent forms
\begin{equation}
\psi^t_K \circ \xi(\widehat{\gamma}(t))=(c_0,0),\ \ \ \psi^t_K\circ \xi(\widecheck{\gamma}(t))=(0,0),
\nonumber \end{equation}
where we have used the fact that $\gamma(t+a)$ corresponds to $(ka,0)$ in terms of the coordinates on $\psi_K^t \circ \xi(N(\gamma)).$ 
It is shown in \cite[Proposition 2.2]{CFHW} that Floer trajectories $\psi_{K}^t \circ \xi(u_i)$ between $\psi^t_K \circ \xi (\widehat{\gamma})$ and $\psi^t_K \circ \xi (\widecheck{\gamma})$ corresponds to negative gradient flow lines $(a_i(s), 0)$ of $h_0$ from the Morse maximum $(c_0,0)$ to the minimum $(0,0)$. In the subsequent discussions, we work in the coordinates given by the trivialization $\psi_K^t \circ \xi \colon N(\gamma) \rightarrow S^1 \times \R^{2n-1}$ and exhibit $k$ solutions of the BV equations.\\
\indent  Let $\widetilde{H}_s$ be the trivial interpolation of the unperturbed Hamiltonian such that $\widetilde{H}_s=\widetilde{H}_{\epsilon}^{\tau}(t,x)$ for all $s \in \R$. We consider the operator
\begin{eqnarray}
&& F\colon W^{1,2}(S^1, S^1\times \R^{2n-1}) \rightarrow L^2(S^1, S^1\times \R^{2n-1}), \nonumber \\
&& F(x)=-J_0(\dot{x}(t)-X_{\widetilde{H}_s}(x(t)).\nonumber
\end{eqnarray}
The kernel of $F$ consists of constant solutions of the form $x_a(t)=(ka,0)$ in $\R \times \R^{2n-1}$. The linearized operator $DF_{x_a}\colon W^{1,2}(S^1, \R^{2n}) \rightarrow L^2(S^1, \R^{2n})$ has the same kernel as $F$. We denote the kernel of $F$ by
$$N:=\{x_a=(ka,0) \mathbin{|} a \in \R \}$$
and its $L^2$-orthogonal complement by $N^{\perp}$.

\indent  The time-dependent perturbation $h_t \in \mathscr{H}^{per}$ that we choose in \eqref{eqn:pb} is of the form
 \begin{equation}
 h_t(q_1, y):=h(\psi_K^t(q_1,y))=h_0(q_1-kt) \text{ for } (q_1, y) \in \xi(N(\gamma)).
\nonumber \end{equation}  
After precomposing with $\psi_K^t,$ it can be transferred into a time-independent form
\begin{equation}
h(q_1,y):=h_0(q_1) \text{ with }(q_1, y) \in \psi_K^t\circ \xi(N(\gamma))).
\nonumber \end{equation}
In this coordinates, the $S^1$-family of time-dependent perturbations is given by
\begin{equation}
h_{\theta, s}(q_1, y)=\begin{cases}
h_0(q_1)\ \ \ \ \ \  \text{ if } s \gg 0;\\
h_0(q_1-k\theta) \text{ if } s \ll 0.
\end{cases}
\nonumber \end{equation} 
We define an $S^1$-family of operators
\begin{eqnarray}
&& f_{\theta, s} \colon W^{1,2}(S^1,S^1\times \R^{2n-1}) \rightarrow L^2(S^1, S^1\times \R^{2n-1}), \label{eqn:f_theta}\\
&& f_{\theta, s}(x)=\nabla h_{\theta, s}(x(t)). \nonumber
\end{eqnarray}

Each solution to the BV-equation yields a map $u: \R \rightarrow W^{1,2}(S^1,  S^1\times\R^{2n-1})$ that satisfies the equation
\begin{equation}\label{eqn:BVs}
u'(s) - F(u(s))+\epsilon f_{\theta, s}(u(s))=0.
\end{equation}
One can write $u(s)=x_a(s)+y(s)$ with $x_a(s) \in N$ and $y(s) \in N^{\perp}$. Equation \ref{eqn:BVs} now becomes
\begin{eqnarray}
x_a'(s)+y'(s)&=&-\epsilon f_{\theta,s}(x_a(s))+DF_{x_a}(y(s))-\epsilon \partial_s f_{\theta, s}(x_a)y(s)\nonumber \\
&\ \ &-\epsilon Df_{\theta,s}(x_a)\cdot y(s)+O(||y(s)||^2_{W^{1,2}(S^1)}), \nonumber
\end{eqnarray} 
where we have that $F(x_a(s))=0$, $f_{\theta,s}(x_a(s)) \in N$ and $DF_{x_{a}}( y(s)) \in N^{\perp}$. Collecting terms in $N$ and $N^{\perp}$, we obtain a similar estimate as in \cite[Proposition 2.2]{CFHW} for some constant $C_1 >0$ satisfying $C_1 \rightarrow 0$ as $\epsilon \rightarrow 0$ as follows,
\begin{eqnarray}
&& \ \ ||x_a'(s)+\epsilon f_{\theta, s}(x_a(s)) ||^2_{L^2}+|| y'(s)-DF_{x_a}(y(s))||^2_{L^2} \label{eqn:estimate1} \\
&& \leq C_1 ||y(s)||^2_{W^{1,2}}+\epsilon \sup|\partial_s f_{\theta, s}(x_a)|\cdot||y(s)||_{W^{1,2}}, \nonumber
\end{eqnarray}
where $L^2$ and $W^{1,2}$ indicates that the norms of the corresponding elements are taken in $L^2(S^1, S^1 \times \R^{2n-1})$ and $W^{1,2}(S^1,S^1\times \R^{2n-1})$ respectively. Since this estimate holds for any $\epsilon >0$, as $\epsilon \rightarrow 0$ and we have that $C_1 \rightarrow 0$, the right hand side of \eqref{eqn:estimate1} approaches zero. It must be the case that $y'(s)-DF_{x_a}y(s)=0$ in \eqref{eqn:estimate1}. It is shown in \cite[Proposition 3.14]{RS} that the Fredholm operator $\frac{\partial}{\partial s}-DF_{x_a}$ is injective on $N^{\perp}$, where $DF_{x_a}$ is invertible. This implies that there is a constant $C_2 >0$ so that
\begin{equation}
||y'-DF_{x_a}y||_{L^2} \geq C_2||y||_{W^{1,2}} \text{ for } y \in N^{\perp}.
\nonumber \end{equation}
This implies that all solutions $u(s)=x_a (s)+y(s)$ to the BV equation satisfy $y(s) \equiv 0$ for $y(s) \in N^{\perp}.$ Setting $u(s)=x_a(s)$ in equation \eqref{eqn:BVs}, we obtain that solutions to the BV equation \eqref{eqn:bv} are given by $x_a \colon \R \rightarrow S^1 \times \R^{2n-1}$ satisfying
\begin{eqnarray}
&&  \ \ \ \ \ \ \ x_a '(s)+\epsilon f_{\theta,s}(x_a (s))=0, \label{eqn:morseflow}\\
&& \ \  \ \ \ \ \displaystyle\lim_{s \rightarrow -\infty} (x_a (s)+k\theta, 0)=( k\theta,0), \displaystyle\lim_{s \rightarrow + \infty} (x_a (s), 0)=(c_0,0), \nonumber \\
&& \text{ or, }\displaystyle\lim_{s \rightarrow -\infty} x_a (s)=0, \displaystyle\lim_{s \rightarrow + \infty} x_a (s)=c_0.\nonumber
\end{eqnarray}
We will show that there are exactly $k$ solutions to equation \eqref{eqn:morseflow} when $\theta=\theta_i$ and $\theta_i=\frac{i}{k}-\frac{\epsilon_0}{k}$ for $i=1, \cdots, k$. Since a Hamiltonian orbit $\gamma$ that corresponds to a $k$-fold Reeb orbit is of the form $\gamma(t)=(kt,0)$ under the trivialization $\xi$, the pullback Morse function $\gamma^*h_0$ on the domain of $\gamma$ has $k$ maxima at $M_i:=\frac{c_0}{k}+\frac{i}{k}$ and $k$ minima at $m_i:=\frac{i}{k}$, which are shown in Figure \ref{fig:2}. The system of equations \eqref{eqn:morseflow} becomes
\begin{eqnarray}
&&  \tilde{x}_a '(s)+\epsilon \tilde{f}_{\theta,s}(\tilde{x}_a (s))=0, \text{ where } \tilde{f}_{\theta,s}=\nabla (\gamma^*h_{\theta_, s}),\label{eqn:morseflow1}\\
&& \displaystyle\lim_{s \rightarrow -\infty} \tilde{x}_a (s)=m_i=\frac{i}{k}, \displaystyle\lim_{s \rightarrow + \infty} \tilde{x}_a (s)=M_0=\frac{c_0}{k},\ \ i=1,\cdots ,k.\nonumber
\end{eqnarray}
One observes that equation \eqref{eqn:morseflow1} is equivalent to the negative gradient flow line equation for the family of Morse functions defined by $\tilde{h}_{\theta, s}(\tilde{q}_1)=\epsilon h_0(\tilde{q}_1-\rho_{\theta}(s))$ for some smooth bump function $\rho_{\theta}\colon \R \rightarrow [0, \theta] \subset [0, k)$ such that 
$$\rho_{\theta}(s) \rightarrow 0 \text{ as }s \rightarrow \infty, \ \ \rho_{\theta}(s) \rightarrow \theta \text{ as }s \rightarrow -\infty.$$
For each fixed $\theta \in [0,k)$, we denote by $\phi^{\theta}_s$ the gradient flow of $\tilde{h}_{\theta, s}$ at the time $s \in \R.$ Given a solution $\tilde{x}_a$ to equation \eqref{eqn:morseflow1}, the map $x:=\phi_{-s}^{\theta} \circ \tilde{x}_a \colon \R \rightarrow S^1$ satisfies the continuation equation for the trivial family defined by $\tilde{f}_s=\nabla (\gamma^*h_0)$ for all $s \in \R$ given by
\begin{equation}\label{eqn:morseflow2}
 x'(s)+\epsilon \tilde{f}_s(x(s))=0.
\end{equation}
By index reasons, there is a unique constant solution at the maximum $x(s)=\frac{c_0}{k}$ to the equation \eqref{eqn:morseflow2} and satisfies asymptotic conditions $x (s) \rightarrow c_0/k$ as $s \rightarrow \pm\infty$. It gives rise to a unique solution $\tilde{x}_{a_i}(s):=\phi^{\theta_i}_s\circ x(s)=\phi^{\theta_i}_s(\frac{c_0}{k})$ to equation \eqref{eqn:morseflow1} satisfying 
$$\tilde{x}_{a_i}(s) \rightarrow \frac{c_0}{k} \text{ as } s \rightarrow \infty, \ \tilde{x}_{a_i}(s) \rightarrow \frac{i}{k} \text{ as }s \rightarrow -\infty$$
for some $\theta_i$ and each $i=1,\cdots,k$. We denote the solution to equation \eqref{eqn:morseflow1} when $\theta=\frac{i}{k}-\frac{\epsilon_0}{k}$ by $a_i(s):=\tilde{x}_{a_i}(s)$ for simplicity. This shows that there are $k$ solutions $x_{a_i}(s)=(k\tilde{x}_{a_i}(s),0)=(ka_i(s),0)$ to the original equation \eqref{eqn:morseflow}. Transferring back to the original coordinates, we have that $k$ solutions to the BV-equation are of the form
\begin{eqnarray}
&& v_i\colon \R \times S^1 \rightarrow S^1 \times \R^{2n-1}, \ \ i=1, \hdots, k,\nonumber\\
&& v_i(s,t)=\psi^{-t}_K(x_{a_i(s)})=(ka_i(s)+kt, 0)\nonumber
\end{eqnarray}
as claimed.
\begin{figure}
\begin{center}
\includegraphics[scale=0.7]{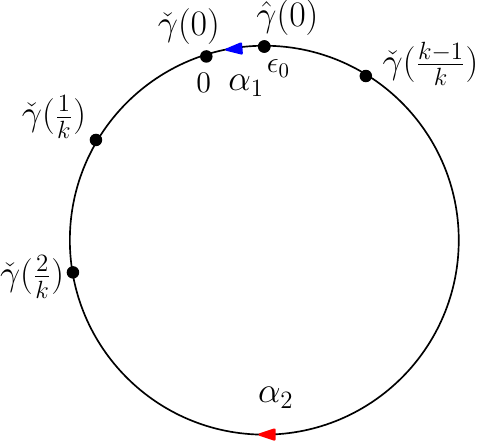}
\end{center}
\caption{Flow lines $\alpha_1$, $\alpha_2$ and the BV solutions from $\widecheck{\gamma}(t+\frac{i}{k})$ and $\widehat{\gamma}(t)$. Here $\epsilon_0=c_0/k$. (\textit{The labels reflect the values on the image $\gamma(S^1)$}).}
\label{fig:2}
\end{figure}
\end{proof}
Next, for $u_i \in \mathscr{M}(\widehat{\gamma}, \widecheck{\gamma}), i=1,2$ we compute $d_{u_i}$  defined in \eqref{eqn:d_u}. Since we have chosen trivializations of $o_{\widehat{\gamma}}$ and $o_{\widecheck{\gamma}}$, the isomorphism $d_{u_i}$ is given by multiplication by $\pm 1$. One notices that the computation result for the Floer differential $d$ in Lemma \ref{lem:d} below coincides with that of Bourgeios and Oancea \cite{BO2} in the Morse-Bott homology setting.
\\
\indent It is shown in \cite[Proposition 2.2]{CFHW} that the two Floer trajectories between $\widehat{\gamma}$ and $\widecheck{\gamma}$ explicitly are given by
\begin{eqnarray}
&& u_i\colon \R \times S^1 \rightarrow S^1 \times \R^{2n-1}\nonumber \\
&& u_i(s,t)=(k\alpha_i(s)+kt,0),\nonumber
\end{eqnarray}
where $\alpha_1$ and $\alpha_2$ are Morse flow lines of $h_0\colon S^1 \rightarrow \R$ (see Figure \ref{fig:2}) such that 
\begin{equation}
\displaystyle\lim_{s \rightarrow -\infty} \alpha_i(s)=c_0, \ \ \displaystyle\lim_{s \rightarrow \infty} \alpha_i(s)=0 \text{ for } i=1,2.
\nonumber \end{equation}

We denote by $\Psi(s,t)$ the linearizations of the flows of the Hamiltonian vector fields $X_{H^{\tau}_{\epsilon}(t)}$ on the image of a Floer trajectory $u \in \widetilde{\mathscr{M}}(\widehat{\gamma}, \widecheck{\gamma})$. 
For $\Psi(s,t)$ in $\Sp(2n)$, one considers symmetric matrices $S(s,t)$ defined by
\begin{equation}
\frac{\partial}{\partial s} \Psi(s,t)=J_0 S(s,t)\Psi(s,t).
\nonumber \end{equation} 
The parallel transport map $d_u \colon o_{\widecheck{\gamma}} \rightarrow o_{\widehat{\gamma}}$ is determined by the spectral flow of this family of self-adjoint operators
\begin{equation}
 s \mapsto A(s):=J_0\frac{\partial}{\partial t}+ S(s,\cdot) \text{ on } L^2(S^1, \R^{2n}).
\nonumber \end{equation}
A crossing $s$ of $A(s)$ is a real number such that $\ker(A(s))\neq 0$ and the crossing form is defined as
\begin{equation}
q(A, s) \colon \ker(A(s)) \rightarrow \R, \ \ q(A(s), s)\xi=\langle \xi, \dot{A}(s)\xi \rangle.
\nonumber \end{equation} 
In particular, Lemma 7.43 in \cite{RS} implies that the path of unbounded self-adjoint operators $A(s)$ has the same crossings as the symplectic path $s \mapsto \Psi(s,1)$ and the crossing forms are isomorphic. A crossing of $\Psi(s,1)$, a real number $s_0 \in [0,1]$ such that $\det(\Psi(s_0,1)-\mathbb{I})=0$, is called a regular crossing if the crossing form $q(\Psi(s,1))\colon \ker(\Psi(s_0,1)-\mathbb{I}) \rightarrow \R$ given by
\begin{equation}
q(\Psi(s,1))v=\omega_0(v, \frac{\partial}{\partial s} \Psi(s,1)\cdot v)=\langle v, S(s, 1) v \rangle
\nonumber \end{equation}
is nondegenerate at $s=s_0$. A crossing is called simple if $\ker(\Psi(s_0,1)-\mathbb{I})$ is one dimensional. For a path with simple crossings, one has $d_u$ is the product of the signs of the crossing form at various crossings $s_0$ in $[0,1]$.\\
\indent We now denote by $\Psi_{\epsilon}^i(s,t)$, $\Psi^i(s,t)$ and $\Phi_{\epsilon}^i(s,t)$ the linearizations of the flows of the Hamiltonian vector fields $X_{H^{\tau}_{\epsilon}}, X_{H^{\tau}}$ and $X_{\epsilon h_{\gamma}}$ on the image of $u_i(s,t)$ respectively. It is proved in Proposition 2.2 \cite{CFHW} that for $\epsilon$ sufficiently small $\Psi_{\epsilon}^i(s,t)$ and $\Psi^i(s,t)\cdot\Phi_{\epsilon}^i(s,t)$ are homotopic with end points via
\begin{equation}
L(r,s,t)=\Psi_{r\epsilon}^i(s,t)\Phi^i_{(1-r)\epsilon}(s,t), \ \ (r,s,t) \in [0,1]\times \R \times [0,1]
\nonumber \end{equation}
for $i=0,1.$ By the homotopy invariance property, it suffices to compute the contributions $d_{u_i}$ for the path of symplectic matrices $\Psi^i(s,1)\cdot\Phi_{\epsilon}^i(s,1)$. A detailed computation of the signs of the crossing form is given in the following Lemma.  

\begin{lem} \label{lem:d}
For $u_i \in \mathscr{M}(\widehat{\gamma}, \widecheck{\gamma}), i=1,2$, we have $d_{u_1}=-\epsilon(\gamma)d_{u_2}$, where $\epsilon(\gamma)$ is equal $1$ if $\gamma$ is good and $-1$ if $\gamma$ is bad.
\end{lem}

\begin{proof}[Proof of Lemma \ref{lem:d}]
Let $\mathscr{SP}(2n)$ be the space of path $\alpha \colon [0,1] \rightarrow \mathrm{Sp}(2n)$ such that 
\begin{equation}
\alpha(0)=\mathbb{I} \text{ and } \alpha(1) \in \mathrm{Sp}^*(2n)=\{ A \in \mathrm{Sp}(2n) \mathbin{|} \det(A-\mathbb{I})\neq 0\}.
\nonumber \end{equation}
A symplectic matrix in $\mathrm{Sp}^*(2n)$ is called \textit{semi-simple} if it can be written as the direct sum of block matrices (with distinct eigenvalues) of the following types:
\begin{enumerate}
\item[(I)]
\text{Eigenvalues in the unit circle: }
$E_{\theta }=\begin{pmatrix}
    \cos \theta & -\sin\theta \\
    \sin \theta &  \cos\theta
  \end{pmatrix}, \ \ \theta \in S^1;$
\item[(II)]
\text{Eigenvalues in the real line: }
 $H_{\lambda}=\begin{pmatrix}
    \lambda &0\\
    0& \frac{1}{\lambda}
  \end{pmatrix}, \ \ \lambda \in \R;$
\item[(III)]
\text{Eigenvalues are quadruples: }
$Q_{a,b}=\begin{pmatrix}
    a & 0 & -b & 0\\
    0 & \frac{a}{a^2+b^2} & 0 & -\frac{b}{a^2+b^2} \\
    b & 0 & a & 0\\
    0 & \frac{b}{a^2+b^2} & 0 & \frac{a}{a^2+b^2}
  \end{pmatrix},\\ (a\pm ib)^{\pm 1} \text{ are eigenvalues}.$
\end{enumerate}
By density of semi-simple matrices \cite[Proposition 26]{Gutt}, one can assume $\Psi(1)$ is semi-simple. 
The contribution $d_{u_i}$ to the differential is a locally constant function of the path components of $\mathscr{SP}(2n)$. In particular, one can deform the path $\Psi(t)$ to another path $\Psi'(t)$ via within the same path components such that the end point $\Psi'(1)$ of $\Psi'(t)$ has no blocks of type (III). This is achieved by degenerating the quadruple of eigenvalues $\lambda$, $\bar{\lambda}$, $\frac{1}{\lambda}$ and $\frac{1}{\bar{\lambda}}$ to either a pair of eigenvalues $\theta, \bar{\theta}$ on the unit circle with multiplicities two, or to a pair of positive eigenvalues $\lambda^+, \frac{1}{\lambda^+}$ if $\lambda >0$ and $\lambda^-, \frac{1}{\lambda^-}$ if $\lambda <0$ with multiplicities two respectively. Therefore without lost of generality, we can assume that the linearized return map $\Psi(1)$ is a direct sum of blocks of type (I) and (II). One can further choose the trivialization $\gamma^*T\widehat{M}$ along $\gamma$ such that $\Psi(t)$ has the following form for convenience
\[ \left( \scalemath{0.85}{  \begin{array}{cccccccccc}
A(t) &       &        &           &        &      &            &        &       & \\
    & E_1(t) &        &           &        &      &            &        &       & \\
    &       & \ddots  &           &        &      &            &        &       & \\ 
    &       &        & E_{n_1}(t) &        &      &            &        &       & \\ 
    &       &        &           &H^+_1(t) &      &            &        &       & \\ 
    &       &        &           &        &\ddots &            &        &       & \\ 
    &       &        &           &        &      &H^+_{n_2}(t) &        &       & \\
    &       &        &           &        &      &            &H^-_1(t) &       & \\ 
    &       &        &           &        &      &            &        &\ddots  & \\ 
    &       &        &           &        &      &            &        &     &H^-_{n_3}
\end{array} }\right) \]

\begin{eqnarray}
&& \text{where } A(t) =\begin{pmatrix}
    1&Ct\\
    0&1
  \end{pmatrix}
\text{for some } C>0, \nonumber\\
&& 
E_j(t)=R_{\theta_j}(t)=\begin{pmatrix}
    \cos (\theta_j t) & -\sin(\theta_j t)\\
    \sin (\theta_j t) &  \cos(\theta_j t)
  \end{pmatrix}
\text{, } \nonumber \\
&& H_i^+(t)=P_{a_i}(t)=\begin{pmatrix}
    e^{a_i t} & 0\\
    0  &  e^{- a_i t}
  \end{pmatrix},\nonumber
\end{eqnarray}
 where $\theta_j \in S^1$ and $a_i=\ln(\lambda^+_i)$. Here $H^-_i(t)$ is homotopic with fixed end points to the concatenation of $2k$ paths of symplectic matrices given by
\begin{equation}
R_{\pi}(t),\ \ P_{\frac{b_i}{k}}(t)R_{\pi}(1),\ \ R_{\pi}(t)P_{\frac{b_i}{k}}(1)R_{\pi}(1), \ \hdots, \ \ \underbrace{P_{\frac{b_i}{k}}(t)R_{\pi}(1)P_{\frac{b_i}{k}}(1)\hdots R_{\pi}(1)}_\text{product of 2k matrices},
\nonumber \end{equation}
where $b_i=\ln(-\lambda^-_i)$ for some $\lambda^-_i <0$. The resulting concatenated path is piece-wise smooth and continuous when $t=j/2k$ for all $j=1,2 \hdots, 2k-1$.\\
\indent Due to simple form of the solutions $u_i(s,t)=(k\alpha_i(s)+kt, 0)$ to the Floer equation in $N(\gamma)$, the linearizations can be computed as follow
\begin{equation}
\Psi^i(s,t)=\Psi(t+\alpha_i(s)) \text{ and } \Phi_{\epsilon}^i(s,t)=e^{\epsilon h''_0(\alpha_i(s))tJ_0B},
\nonumber \end{equation}
where $B=\Diag(1,0, \cdots, 0)$ and $\Psi(t)$ is the linearization of the flow of the unperturbed Hamiltonian $\psi^t_{H^{\tau}}$. The matrix $$\Psi^i(s,1)\cdot\Phi_{\epsilon}^i(s,1)=\Psi(\alpha_i(s))e^{\epsilon h''_0(\alpha_i(s))J_0B}$$ can be evaluated and equals to
\begin{equation} 
\resizebox{0.95\hsize}{!}{$\Diag(\widetilde{A}(\alpha_i), E_1(\alpha_i), \hdots, E_{n_1}(\alpha_i), H^+_1(\alpha_i), \hdots, H^+_{n_2}(\alpha_i), H^-_1(\alpha_i), \hdots, H^-_{n_3}(\alpha_i))$}\nonumber,
\nonumber \end{equation}
where
\begin{eqnarray}
\widetilde{A}(\alpha_i(s)) &=&
\begin{pmatrix}
  1 & C\alpha_i(s) \\ 0  & 1
\end{pmatrix}
\begin{pmatrix}
  1 &  0 \\ \epsilon h''(\alpha_i(s)) & 1 
\end{pmatrix}
\nonumber \\
&=&\begin{pmatrix}
  1 + C\alpha_i(s) \epsilon h''(\alpha_i(s)) & C\alpha_i(s) \\ \epsilon h''(\alpha_i(s))  & 1
\end{pmatrix}. 
\nonumber
\end{eqnarray}
After reparametrizing $\Psi^i(s,1)\cdot\Phi_{\epsilon}^i(s,1)$ as in Remark \ref{rmk:repara} below, one obtains loops of symmetric matrices $S^i(s)$ given by

\[ \left( \scalemath{0.7}{ \begin{array}{cccccccccc}
S^A(\alpha_i(s)) &       &        &           &        &      &            &        &       & \\
    & S_1^E(\alpha_i(s)) &        &           &        &      &            &        &       & \\
    &       & \ddots  &           &        &      &            &        &       & \\ 
    &       &        & S_{n_1}^E(\alpha_i(s)) &        &      &            &        &       & \\ 
    &       &        &           &S_1^+(\alpha_i(s)) &      &            &        &       & \\ 
    &       &        &           &        &\ddots &            &        &       & \\ 
    &       &        &           &        &      &S_{n_2}^+(\alpha_i(s)) &        &       & \\
    &       &        &           &        &      &            &S_1^-(\alpha_i(s)) &       & \\ 
    &       &        &           &        &      &            &        &\ddots  & \\ 
    &       &        &           &        &      &            &        &     &S_{n_3}^-(\alpha_i(s))
\end{array}} \right) \]

\begin{equation} \label{eqn:det1}
\text{where  }
S^A(\alpha_i(s)))=-J_0 \cdot \frac{d}{ds} \widetilde{A}(\alpha_i(s)) \cdot \widetilde{A}(\alpha_i(s))^{-1}.
\end{equation} 
As $\Psi^i(s,1)\cdot                                                                                                                                   \Phi_{\epsilon}^i(s,1)$ is in block diagonal form, and generically one can choose the Morse function $h_{\gamma}(t)\colon N(\gamma) \rightarrow \R$ such that all crossings of $\Psi^i(s,1)\cdot\Phi_{\epsilon}^i(s,1)$ occur in different time. Thus it suffices to analyze the crossings for each block. For the first block $\widetilde{A}(\alpha_i(s))$ with $i=1,2$, there exists a crossing $s_0^i \in \R$ such that $h''_0(\alpha_i(s_0^i))=0$ so that $\det(\widetilde{A}(\alpha_i(s_0^i))-\mathbb{I})=0$. This crossing is transverse and $\ker(\Psi^i(s_0^i,1)\cdot\Phi_{\epsilon}^i(s_0^i,1)-\mathbb{I})=\Span\{e_1\}$ with $e_1=(1,0, \hdots, 0)$. By \eqref{eqn:det1}. One computes the $(1,1)$-entry of $S_A(\alpha_i(s))$ and observes that
\begin{equation}
q(\Psi^i(s_0^i,1)\cdot\Phi_{\epsilon}^i(s_0^i,1))e_1= -\epsilon h'''_0(\alpha_i(s)).
\nonumber \end{equation}
Since $\sign(h_0'''(\alpha_1(s_0^1)))=-\sign(h_0'''(\alpha_2(s_0^2)))$, we have that
\begin{equation}
\sign(q(\Psi^1(s_0^1,1)\cdot\Phi_{\epsilon}^1(s_0^1,1)))=-\sign(q(\Psi^2(s_0^2,1)\cdot\Phi_{\epsilon}^2(s_0^2,1)))
\nonumber \end{equation}
for crossings $s_0^1$ and $s_0^2$. One observes that $s_0^1$ is the only crossing of $\Psi^1(s,1)\cdot\Phi_{\epsilon}^1(s,1)$ as $c_0 \in S^1$ is taken to be sufficiently closed to zero in the definition of the Morse function $h_0$. For $\Psi^2(s,1)\cdot\Phi_{\epsilon}^2(s,1)$, each block matrix $H^-_i(\alpha_2(s))$ contributes $(k-1)$ more crossings precisely when 
\begin{eqnarray}
&&\ \ \  \Tr \big( \underbrace{ R_{\pi}(\alpha_2(s))P_{\frac{b_i}{k}}(1)\hdots R_{\pi}(1)}_\text{product of 2j+1 matrices, $j \geq 1$} \big)  \label{eqn:trace} \\
&&=\Tr
\begin{pmatrix} 
 e^{\frac{jb_i}{k}}\cos(\alpha_2(s)) &  -e^{\frac{jb_i}{k}}\sin(\alpha_2(s)) \\ e^{-\frac{jb_i}{k}}\sin(\alpha_2(s)) & e^{-\frac{jb_i}{k}}\cos(\alpha_2(s))
\end{pmatrix}=2. \nonumber
\end{eqnarray}
For each $ 1 \leq j \leq k-1$, there is only one $s_j \in \R$ such that \eqref{eqn:trace} is satisfied, that is to say that $\alpha_2(s_j)$ is the only crossing in $(\frac{j}{k}, \frac{j+1}{k})$ for each $j$. The kernel of $H^-_i(\alpha_2(s_j))-\mathbb{I}$ is 1-dimensional. For each crossing $s_j$, we compute the sign of the crossing form by observing that 
\begin{equation} \label{eqn:det2}
S_i^-(\alpha_2(s))=\begin{pmatrix}
 e^{-\frac{2jb_i}{k}}\pi \alpha'_2(s) &  0 \\ 0 & e^{\frac{2jb_i}{k}}\pi \alpha'_2(s) 
\end{pmatrix},
\end{equation} 
is negative definite when $\alpha_2(s) \in (\frac{j}{k}, \frac{2j+1}{k})$ since $\alpha'_2(s) < 0$ for all $s \in \R$. This implies that 
$$\sign(q(\Psi^2(s_j,1)\cdot\Phi_{\epsilon}^2(s_j,1))<0 \text{ for all }j=1,\hdots,k-1$$ 
on the $1$-dimensional kernel. Since there are $n_3$ block matrices whose eigenvalues are negative in the linearized return map $\Psi(1)$, the total contributions of all the crossings provide us the relation
\begin{equation}
d_{u_1}=(-1)\cdot(-1)^{(k-1)n_3} d_{u_2}.
\nonumber \end{equation}
One observes that $(-1)^{(k-1)n_3}=-1$ if and only if both $k-1$ and $n_3$ are odd, which precisely corresponds to the definition of bad Reeb orbits in \eqref{defn:bad}. Therefore, we conclude that $d_{u_1}=-\epsilon(\gamma)d_{u_2}$.
\end{proof}

\begin{rmk}  \label{rmk:repara}
To obtain a loop of symmetric matrices in \eqref{eqn:det1}, one needs to reparametrize the path of symplectic matrices as $$\Psi^i(\chi(\alpha_i(s)),1)\cdot\Phi_{\epsilon}^i(\chi(\alpha_i(s)),1),$$ where $\chi \colon [0,1] \rightarrow [0,1]$ is a non-decreasing such that $\chi'(0)=\chi'(1)=0$. Also, the block matrix $H^-_i(t)$ in $\Psi(t)$ is only piece-wise smooth. In order to obtain continuous loops $S_i^-$, we reparametrize $H^-_i(t)$ to be $H^-_i(\chi(t))$ with $\chi'(t)=0$ when $t=\frac{i}{2k}$, $i =1,\hdots,2k-1$. However the sign of the crossing form is independent of the parametrizations, so we have suppressed the reparametrization $\chi$ in the above proof.
\end{rmk}

It remains to compute $\Delta_{v_i}$ associated to each solution $v_i$ of the BV equation in the definition of $\Delta$. We recall that the solutions to the BV equation at $\theta_i$ for $i=1 \hdots k$ are given by 
\begin{eqnarray}
&& v_i\colon \R \times S^1 \rightarrow S^1 \times \R^{2n-1} \nonumber  \\
&& v_i(s,t)=\psi^{-t}_K(x_{a_i(s)})=(ka_i(s)+kt, 0). \nonumber \\
&& \displaystyle\lim_{s \rightarrow -\infty}v_i(s,t+\frac{i}{k})= \widecheck{\gamma}(t+\frac{i}{k}), \displaystyle\lim_{s \rightarrow +\infty}v_i(s,t)= \widehat{\gamma}(t).\nonumber
\end{eqnarray}
For fixed $\theta_i$ the linearization of the equation \eqref{eqn:bv} at $v_i$ can be written as 
\begin{eqnarray}
&& \ \ D_{v_i}\colon W^{1,p}(\R \times S^1, S^1 \times \R^{2n-1}) \rightarrow L^p(\R \times S^1, S^1 \times \R^{2n-1})   \nonumber \\
&& \ \ D_{v_i}(X)=\partial_s(X)+J_0\partial_t(X)+\widecheck{S}(t+a_i(s))\cdot X \nonumber \\
&& \ \ \displaystyle\lim_{s \rightarrow -\infty}\widecheck{S}(t+a_i(s))=\widecheck{S}(t+\frac{i}{k}), \  \displaystyle\lim_{s \rightarrow +\infty}\widecheck{S}(t+a_i(s))=\widehat{S}(t),\nonumber
\end{eqnarray}
where $J_0\widecheck{S}$ and $J_0\widehat{S}$ in $\mathfrak{sp}(2n)$ generates $\Psi_{\widecheck{\gamma}}(t)$ and $\Psi_{\widehat{\gamma}}(t)$ in $\mathrm{Sp}(2n)$ respectively.
As defined in section \ref{sec:bv}, each $v_i \in \mathscr{M}_{\Delta}(\widecheck{\gamma},\widehat{\gamma})$ induces a map $\Delta_{v_i}\colon o_{\widehat{\gamma}} \rightarrow o_{\widecheck{\gamma}}$ under
\begin{equation} \label{eqn:orbv}
o_{\widecheck{\gamma}} \otimes |T_{v_i} \mathscr{M}_{\Delta}(\widecheck{\gamma},\widehat{\gamma})| \cong \det(D_{v_i}) \otimes \det(T_{\theta_i}S^1)\otimes o_{\widehat{\gamma}}.
\end{equation}
If we fix the orientation of $T_{\theta_i}S^1$ to be $\R \langle \frac{\partial}{\partial \theta} \rangle$, then $\ker(D_{v_i})$ is spanned by
\begin{equation}
\epsilon(v_i)(f_i(s), 0) \text{ with } f_i(s):=\frac{\partial}{\partial s}a_i(s) <0,
\nonumber \end{equation} 
where $\epsilon(v_i)$ is equal to $\pm 1$ if the isomorphism $\Delta_{v_i}\colon o_{\widehat{\gamma}} \rightarrow o_{\widecheck{\gamma}}$ is given by multiplication by $\pm 1$. The value of $\epsilon(v_i)$ is determined by comparing the chosen orientations of two sides of \eqref{eqn:orbv}. The following Lemma \cite[Lemma 4.27]{BO1} is required to compute $\epsilon(v_i)$.
\begin{lem}\label{lem:T}
Given a loop of nondegenerate symmetric matrices $S_1(t)$, we define $S_k(t)=S_1(kt)$ and assume that $S_k(t)$ is also nondegenerate. Consider operators
\begin{equation}\label{eqn:twist}
T:=\partial_s + J_0 \partial_t +S_k(t+\frac{\beta(s)}{k}),
\end{equation}
with $\beta\colon \R \rightarrow [0,1]$ a smooth function such that 
$$\displaystyle\lim_{s\rightarrow -\infty}\beta(s)=1, \displaystyle\lim_{s\rightarrow +\infty}\beta(s)=0$$
and $|\beta'(s)|<C$ for some small enough constant $C$, there are actions on $\mathscr{O}:=\mathscr{O}_-(S_k)$ or $\mathscr{O}(S_k, S)$ defined by
\begin{equation}
\Psi_{\rho}^T(D):=T \#_{\rho} D, \ \ \rho \gg0.
\nonumber \end{equation}
The induced action of $\Psi_{\rho}^T$ on $\Det(\mathscr{O})$ is orientation reversing if and only if $k$ is even and the differences of Conley--Zehnder indices $\mu_{CZ}(\Psi_{S_k})-\mu_{CZ}(\Psi_{S_1})$ is odd.
\end{lem}

Let $T$ be the operator defined in \eqref{eqn:twist}. The $i$-th iterations of $T$ is defined to be
\begin{equation}
T^i=\underbrace{T \#_{\rho} \cdots \#_{\rho} T}_\text{i}.
\nonumber \end{equation}
Taking $S_k(t)=\widecheck{S}(t)$ in Lemma \ref{lem:T}, we have
\begin{equation}
T^{i}= \partial_s+J_0 \partial_t +\widecheck{S}(t + \frac{i}{k} \beta_{\rho}^i(s)).
\nonumber \end{equation}
where $\beta_{\rho}^i(s)\colon \R \rightarrow [0,1]$ is another smooth function depending on $\rho$ and $i$ such that 
\begin{equation}
\displaystyle\lim_{s\rightarrow -\infty}\beta_{\rho}^i(s)=1, \displaystyle\lim_{s\rightarrow +\infty}\beta_{\rho}^i(s)=0 \text{ and }|\partial_s\beta_{\rho}^i(s)|<C.
\nonumber \end{equation}
\begin{lem} \label{lem:delta}
For each $\gamma \in \mathscr{P}(H^{\tau})$ which corresponds to a $k$-fold Reeb orbit, we have $\Delta_{v_i}=(\epsilon(\gamma))^{i-1}\Delta_{v_1}$ where $\epsilon(\gamma)$ is equal $1$ if $\gamma$ is good, and $-1$ if $\gamma$ is bad.
\end{lem}

\begin{proof}
For each $(\theta_i, v_i) \in \mathscr{M}_{\Delta}(\widehat{\gamma}, \widecheck{\gamma})$, the stabilization of $D_{v_i}$
\begin{eqnarray}
&& \overline{D}_{v_i}\colon T_{\theta_i}S^1 \oplus W^{1,p}(\R \times S^1, S^1 \times \R^{2n-1}) \rightarrow L^p(\R \times S^1, S^1 \times \R^{2n-1}), \nonumber \\
&& \overline{D}_{v_i}(v, X)=v+D_{v_i}(X)\nonumber
\end{eqnarray}
is a surjective operator. For $\rho \gg 0 $, the glued operator $D_i:=T^{i-1}\#_{\rho} \overline{D}_{v_1}$ is also surjective with a uniformly bounded inverse $Q_i$. It is shown in \cite[Corollary 6]{BM} or \cite[Proposition 9]{FH} that restriction of the projection
\begin{equation}
(\mathbb{I}-Q_i\circ D_i)|_{\ker(\overline{ D}_{v_1})}\colon \ker(\overline{ D}_{v_1}) \rightarrow \ker(T^{i-1}\#_{\rho} \overline{ D}_{v_1})
\nonumber \end{equation}
is an isomorphism. Lemma \ref{lem:T} now implies that this isomorphism is explicitly given by
\begin{equation}
\epsilon(v_1)(f_1(s),0) \mapsto \epsilon(v_1)(\epsilon(\gamma))^{i-1}(f_i^{\#}(s),0),
\nonumber \end{equation}
where $f_i^{\#}(s)$ is the extension by zero of $f_1(s)$ in $\ker(\overline{ D})_{v_1}$ as the operator $T^{i-1}$ has trivial kernel. Since the norm $||f_i^{\#}-f_1||_{W^{1,p}} \rightarrow 0$ as $\rho $ approaches $\infty$, we have $f_i^{\#}(s) < 0$ as well.
Then it suffices to show that for each $i$ there is a continuous path of surjective operators connecting $\overline{D}_{v_i}$ and $T^{i-1} \#_{\rho} \overline{D}_{v_1}$ in $\mathscr{O}(\widecheck{S}(t+\frac{i}{k}),\widehat{S}(t))$. We consider the path of stabilized operators
\begin{eqnarray}
&& \overline{D}^r_i\colon \R \oplus W^{1,p}(\R \times S^1,  S^1 \times \R^{2n-1}) \rightarrow L^p(\R \times S^1, S^1 \times \R^{2n-1}),  \nonumber \\
&& \overline{D}^r_i(v,X)=v+\partial_s(X) +J_0\partial_t(X)+\widecheck{S}(t+h^r_i(s))(X),\nonumber
\end{eqnarray}
where $h^r_i(s)\colon \R \rightarrow S^1$ is a linear homotopy between $a_i(s)$ and $\frac{i}{k}\beta_{\rho}^i(s) \#_{\rho} a_1(s)$ satisfying 
\begin{equation}
\displaystyle\lim_{s \rightarrow -\infty}h^r_i(s)=\frac{i}{k}, \displaystyle\lim_{s \rightarrow +\infty}h^r_i(s)=\frac{c_0}{k} \ \text{ for } i=1,\hdots k \text{ and }r \in [0,1].
\nonumber \end{equation} 
Let $\ker(\overline{D}^r_i)$ be spanned by $(f^i(r, s),0)$ with 
\begin{equation}
f^i(0,s)=\epsilon(v_i) \cdot f_i(s) \text{ and } f^i(1,s)=(\epsilon(\gamma))^{i-1}\epsilon(v_1) \cdot f_i^{\#}(s).
\nonumber \end{equation}
By linearity of $\overline{D}^r_i$, the sign of $f^i(r, s)$ is constant sign for all $r \in [0,1]$ and for each $i$, that is,
\begin{equation}
\sign(\epsilon(v_i)f_i(s))=\sign(\epsilon(v_1)(\epsilon(\gamma))^{i-1} f_i^{\#}(s)).
\nonumber \end{equation}
Together with the fact that $f_i(s)<0$ and $f_i^{\#}(s)<0$, we conclude that 
\begin{equation}
\epsilon(v_i)=(\epsilon(\gamma))^{i-1}\epsilon(v_1),   \text{ or, }\Delta_{v_i}=(\epsilon(\gamma))^{i-1}\Delta_{v_1} \text{ for }i=1,\cdots, k
\nonumber \end{equation}
 as claimed.
\end{proof}

Having computed the equivariant differential $\delta^{S^1}$ between $\widehat{\gamma}$ and $ \widecheck{\gamma}$, we present the proof of Theorem \ref{thm:local} as follows.

\begin{proof}[Proof of Theorem \ref{thm:local}]
Given $0=l_0 < l_1 < l_2 < \cdots$ with $l_i \in \mathscr{S}$ for $i \geq 1$, we set $\tau_i=\frac{l_i+l_{i+1}}{2}$ and define $\hpl(M) \cong \varinjlim \hpl(M, H^{\tau_i}_t)$ as in \eqref{eqn:psh}. By Lemma \ref{lem:preact}, for fixed $\tau_i$ there is a filtration induced by \eqref{defn:af} on $CF^*(M, H^{\tau_i}_t)((u))$ given by
\begin{equation}
F^p CF^*(M, H^{\tau_i}_t)((u))=\Big(\bigoplus_{\substack {\mathscr{A}_{H^{\tau_i}_t}(\gamma) \geq a_{\tau_p} \\ \gamma \in \mathscr{P}(H^{\tau_i}_t)}} \Z\langle o_{\gamma} \rangle \Big)((u)),
\nonumber \end{equation}
and the associated graded complexes are defined by
\begin{eqnarray}
G^p CF^*(M, H^{\tau_i}_t)&:=&F^p CP^*(M, H^{\tau_i}_t)/F^{p-1} CP^*(M, H^{\tau_i}_t)\nonumber\\
&=&\Big(\bigoplus_{\substack { a_{\tau_{p-1}}\leq {\mathscr{A}_{H^{\tau_i}_t}(\gamma) \leq a_{\tau_p}}  \\ \gamma \in \mathscr{P}(H^{\tau_i}_t)}} \Z\langle o_{\gamma}\rangle \Big)((u)).\nonumber
\end{eqnarray}
The $E_0$-page of the spectral sequence associated to this filtration is given by
\begin{equation} 
E_0^{p,q}=G^p CP^{p+q}(M, H^{\tau_i}_t).
\nonumber \end{equation}
Since we have assumed that all $1$-periodic orbits $\gamma$ have distinct action $\mathscr{A}_{H^{\tau_i}_t}(\gamma)$ by condition \eqref{eqn:cond}, the associated graded complexes are precisely
\begin{equation}
G^p CF^*(M, H^{\tau_i}_t)((u)) \cong\Z \langle o_{ \widehat{\gamma}}, o_{\widecheck{\gamma}} \rangle \otimes_{\Z} \Z((u)).
\nonumber \end{equation}
It suffices to compute $\delta_i$ on $ o_{\widehat{\gamma}}$ and $o_{\widecheck{\gamma}}$ to obtain the differential
\begin{equation}
d_0\colon E_0^{p,q}=G^0 CP^{p+q}(M, H^{\tau_i}_t) \rightarrow  E_0^{p, q+1}=G^0 CP^{p+q+1}(M, H^{\tau_i}_t).
\nonumber \end{equation} 
Proposition \ref{prop:1} implies that for all $p \neq 0$ and $|\widehat{\gamma}|=|\widecheck{\gamma}|+1=q$, we have
\begin{equation} \label{eqn:good}
\cdots \xrightarrow{\pm k} E_0^{p,q-1}\cong\Z \langle \widecheck{\gamma} \rangle \xrightarrow{0} E_0^{p,q}\cong \Z \langle \widehat{\gamma} \rangle \xrightarrow{\pm k} E_0^{p,q+1}\cong \Z \langle u\widecheck{\gamma} \rangle \xrightarrow{0} \cdots,
\end{equation}
if $\gamma$ corresponds to a good Reeb orbit on $\partial M$, or 
\begin{equation} \label{eqn:bad}
\cdots \xrightarrow{0} E_0^{p,q-1}\cong \Z \langle \widecheck{\gamma} \rangle  \xrightarrow{\pm 2} E_0^{p,q}\cong \Z \langle \widehat{\gamma} \rangle \xrightarrow{0} E_0^{p,q+1} \cong \Z \langle u\widecheck{\gamma} \rangle \xrightarrow{\pm 2} \cdots,
\end{equation}
if $\gamma$ corresponds to a bad Reeb orbit. When $p =0$, as all $1$-periodic orbits $\gamma$ with $\mathscr{A}_{H^{\tau_i}_t}(\gamma)=0$ are critical points of the Morse function $H^{\tau_i}_t|_{\Int(M)}$, there is no regular solutions to \eqref{eqn:delta} for $i \geq 1$. So $\delta_i=0$ on $E_0^{0,q}$ for $i \geq 1$ and all $q$. This implies that $d_0\colon E_0^{0,q} \rightarrow  E_0^{0, q+1}$ agrees with the Floer differential $\delta_0=d$. Also, it is shown in \cite[Theorem 7.3]{SZ} that for $C^2$-small Morse function $H^{\tau_i}_t|_{\Int(M)}$ the differential $\delta_0$ agrees with the Morse differential $d_M$. Together with \eqref{eqn:good} and \eqref{eqn:bad}, one obtains that the $E_1$-page of the spectral sequence associated to the corresponding filtration on $CP^*(M, H^{\tau_i}_t)\otimes_{\Z}\Q$ is given by
\begin{eqnarray}
&& E^{p.q}_1=0, \ \ p \geq 1,\nonumber \\
&& E^{0.q}_1=H^q(M, \Q)((u)).\nonumber
\end{eqnarray}
We conclude that this spectral sequence degenerates at $E_1$-page and converges to $H^*(M, \Q)((u))$ for each $\tau_i$. Now for $\tau_i  < \tau_j$, there is a natural inclusion of filtered cochain complexes
\begin{equation}
\iota_{i,j}\colon CF^*(M, H^{\tau_i}_t)((u))\otimes_{\Z}\Q \hookrightarrow CF^*(M, H^{\tau_j}_t)((u))\otimes_{\Z}\Q,
\nonumber \end{equation}
which induces identity maps between $(E_0^{p,q})^{\tau_i}$ and $(E_0^{p,q})^{\tau_j}$ if $p \leq i$ and zero maps otherwise. As $\tau_i \rightarrow \infty$,
\begin{equation}
\{(E_r^{p,q})^{\tau_i}, (d_r)^{\tau_i} \}_{i\geq 0}
\nonumber \end{equation}
forms a direct limit of spectral sequences with respect to $\iota_{i,j}$. We have shown that $((E_0^{p,q})^{\tau_i}, (d_0)^{\tau_i})$ converges to $H^*(M, \Q)((u))$ for all $\tau_j$, it follows that the direct limit is still $H^*(M, \Q)((u))$. This finishes proof.
\end{proof}

\section{Periodic Symplectic Cohomology $\hp(M)$} \label{sec:cpsh}
We define periodic symplectic cohomology of a Liouville domain $\hp(M)$ in this section. Given $\tau_i \rightarrow \infty$ with $\tau_i \notin \mathscr{S}$, we choose a cofinal system of admissible Hamiltonians $\mathbf{H}=\{H^{\tau_i}_t \}_{i \geq 0}$ of slope $\tau_i$. One considers a formal variable $q$ of degree $-1$ satisfying $q^2=0$ and set
\begin{equation}
CF^*(M, H^{\tau_i}_t)[q]:= CF^*(M, H^{\tau_i}_t)\oplus qCF^*(M, H^{\tau_i}_t) \text{ for all } i.
\nonumber \end{equation}
Elements of $CF^*(M, H^{\tau_i}_t)[q]$ are of the form $a+bq$ for $a\in CF^*(M, H^{\tau_i}_t)$ and $b \in CF^*(M, H^{\tau_i}_t)$. The telescope construction of $\{CF^*(M, H^{\tau_i}_t)\}_{i\geq 1}$ introduced in \cite{AS} is given by the following diagram with appropriate Koszul signs omitted.
\[
\xymatrix@!=1.8pc{
CF^*(M, H^{\tau_1}_t) \ar@(ul,ur)^{d} && CF^*(M, H^{\tau_2}_t)  \ar@(ul,ur)^{d} && CF^*(M, H^{\tau_3}_t)  \ar@(ul,ur)^{d} && \cdots\\
qCF^*(M, H^{\tau_1}_t) \ar[u]^{id} \ar[urr]^{\kappa_0} \ar@(dl,dr)_{d} && qCF^*(M, H^{\tau_2}_t)  \ar[u]^{id} \ar[urr]^{\kappa_0} \ar@(dl,dr)_{d} && qCF^*(M, H^{\tau_3}_t) \ar[u]^{id} \ar[urr]^{\kappa_0} \ar@(dl,dr)_{d} && \cdots
}
\]
where $\kappa_0$ is the continuation map associated to the monotone homotopy between $H^{\tau_i}_t$ and $H^{\tau_{i+1}}_t$.  
Explicitly, this is given by the infinite direct sum
\begin{equation}
\widehat{CF^*}(M,\mathbf{H}):=\bigoplus_{i=1}^{\infty} CF^*(M, H^{\tau_i}_t)[q]
\nonumber \end{equation}
with the total differential
\begin{equation}
\widehat{\delta}_0(a+ qb) = (-1)^{deg(a)}d(a) + (-1)^{deg(b)}(qd(b) + \kappa_0(b) - b)
\nonumber \end{equation}
for $a+qb \in \widehat{CF^*}(M,\mathbf{H})$. Similarly given $\delta_j$ with $j\geq 1$ in the definition of $\delta^{S^1}$, we can extend it to an operation $\widehat{\delta}_i$ on $\widehat{CF^*}(M,\mathbf{H})$ by setting
\begin{equation}
\widehat{\delta}_j(a+ qb) = (-1)^{deg(a)}\delta_j(a) + (-1)^{deg(b)}(q\delta_j(b) + \kappa_j(b)),
\nonumber \end{equation}
where $\kappa_j\colon CF^*(M, H^{\tau_i}_t) \rightarrow CF^{*-2j}(M, H^{\tau_{i+1}}_t)$ are defined previously in \eqref{defn:k_i}.
Using the facts
\begin{equation}
\sum_{i+j=k}\delta_i\delta_j =0\ \ \text{ and } \sum_{i+j=k} \kappa_i  \delta_j - \delta_j \kappa_i=0 \ \text{ for all } k \geq 0,
\nonumber \end{equation}
it can be shown that $\sum_{i+j=k} \widehat{\delta}_i \widehat{\delta}_j=0$.
This implies that the cochain complex $(\widehat{CF^*}(M,\mathbf{H}),\{ \widehat{\delta}_j \}_{j \geq 0} )$ is an $S^1$-complex. The periodic symplectic cohomology of the Liouville domain $M$ is then defined to be the homology of the cochain complex
\begin{equation}
\big(\widehat{CF^*}(M,\mathbf{H})((u)), \widehat{\delta}^{S^1}=\widehat{\delta}_0+u\widehat{\delta}_1+u^2\widehat{\delta}_2+\cdots \big).
\nonumber \end{equation} 
We denote the resulting cohomology groups by $\hp(M, \mathbf{H})$. In fact, periodic symplectic cohomology is independent of the choice of Hamiltonians $\mathbf{H}=\{ H^{\tau_i}_t\}_{i \geq 0}$, which is shown as follows.

\begin{prop}\label{prop:indepham}
Periodic symplectic cohomology $\hp(M, \mathbf{H})$ is independent of the choice of the sequence of admissible Hamiltonians $\mathbf{H}=\{H^{\tau_i}_t\}_{i \geq 0}$ used in the definition.
\end{prop}

\begin{proof}
Given another choice of the cofinal system of admissible Hamiltonians $\mathbf{K}=\{K^{\sigma_j}_t\}_{j \geq 0}$, we need to construct a map $f=(f_0,f_1,\hdots)$ of $S^1$-complexes between $\widehat{CF^*} (M, \mathbf{H})$ and $\widehat{CF^*}(M, \mathbf{K})$. Let $\mathbf{H} \sqcup \mathbf{K}$ be the union of the cofinal systems of Hamiltonians. There is a pre-order on $\mathbf{H} \sqcup \mathbf{K}$ by slopes of the Hamiltonians. We call $K^{\sigma_{j+m}}_t$ is the \textit{$m$-th successor} of $H^{\tau_{i}} _t$ in $\mathbf{K}$ if $H^{\tau_{i}}_t$ and $K^{\sigma_{j+m}}_t$ in $\mathbf{H} \sqcup \mathbf{K}$ satisfy
\begin{eqnarray}
&& H^{\tau_i}_t \preceq H^{\tau_{i+1}}_t  \preceq \hdots \preceq H^{\tau_{i+{n_i}}}_t \preceq K^{\sigma_{j+1}}_t  \preceq \hdots\label{eqn:order}\\
&& \preceq K^{\sigma_{j+m}}_t \preceq \hdots \preceq K^{\sigma_{j+{m_j}}}_t \preceq H^{\tau_{i+n_i+1}}_t\nonumber
\end{eqnarray}
for some non-negative integers $n_i$ and $m_j$.
We will define the map $f$ of $S^1$-complexes on each summand of $\widehat{CF^*}(M, \mathbf{H})$ as follow.
For $a_i \in CF^*(M, H^{\tau_i}_t)$, the map $f_p$ is given by 
\begin{equation}
f_p(a_i)=\eta_p^1(a_i)+\eta_p^2(a_i)+\hdots+\eta_p^{m_j}(a_i),
\nonumber \end{equation}
where $\eta_p^m \colon CF^*(M, H^{\tau_i}_t) \rightarrow CF^{*-2p}(M, K^{\sigma_{j+m}}_t)$ is the $p$-th order continuation map associated to the monotone homotopy between $H^{\tau_i}_t$ and $K^{\sigma_{j+m}}_t$ for $m=1, \hdots, m_j$. To determine the value of $f_p$ on the element $qb_i$ of $qCF^*(M, H^{\tau_{i}}_t)$ with $n_i \geq 1$ and $m<m_j$ in \eqref{eqn:order}, we consider the following diagram for each successor $K_t^{\sigma_{j+m}}$ of $H_t^{\tau_{i}}$ in $\mathbf{K}$
\[
\xymatrix@!=0.75pc{
\cdots && CF^*(M, H^{\tau_{i}}_t) \ar@(ul,ur)^{\delta^{S^1}} \ar[dd]_{\eta^m} && \ \ && CF^*(M, H^{\tau_{i+1}}_t) \ar[dd]_{\eta^{m+1}} \ar@(ul,ur)^{\delta^{S^1}} && \cdots\\
\cdots  && \ \ && qCF^*(M, H^{\tau_{i}}_t) \ar[ull]^{id} \ar[urr]^{\kappa} \ar@(ul,ur)^{\delta^{S^1}} \ar[drr]^{h^m} \ar[dd]_{\eta^{m}} &&  \ \ &&  \cdots \\
\cdots && CF^*(M, K^{\sigma_{j+m}}_t) \ar@(dl,dr)_{\delta^{S^1}} && \ \ && CF^*(M, K^{\sigma_{j+m+1}}_t)  \ar@(dl,dr)_{\delta^{S^1}} && \cdots\\
\cdots  && \ \ && qCF^*(M, K^{\sigma_{j+m}}_t) \ar[ull]^{id} \ar[urr]^{\kappa'} \ar@(dl,dr)_{\delta^{S^1}} &&  \ \ &&  \cdots
}
\]
where $h^m:=(h_0^m,h_1^m, \hdots)$ is the $S^1$-equivariant homotopy operator, provided by Lemma \ref{lem:homo} below, between the maps $\eta^{m+1}\circ\kappa$ and $\kappa' \circ \eta^m$ of $S^1$-complexes $qCF^*(M, H^{\tau_{i}}_t)$ and $CF^*(M, K^{\sigma_{j+m+1}}_t)$. Depending on the ordering of $H^{\tau_i}_t$ and $K^{\sigma_{j+m}}_t$ in \eqref{eqn:order}, we define that
\begin{equation}
f_p(qb_i)=
\begin{cases}
\sum_{s=1}^{m_j}  q\eta_p^s(b_i) + \sum_{s=1}^{m_j-1} h_p^s(b_i) & \text{if } n_i\geq 1;\\
\sum_{s=1}^{m_j} q\eta_p^s(b_i)  + h_p^{m_j}(b_i) & \text{if } n_i=0 .
\end{cases}
\nonumber \end{equation}
The fact that $h^m$ is a $S^1$-equivariant homotopy operator between the maps $\eta^{m+1}\circ\kappa$ and $\kappa' \circ \eta^m$ of $S^1$-complexes $qCF^*(M, H^{\tau_{i}}_t)$ and $CF^*(M, K^{\sigma_{j+m+1}}_t)$ for all $m \leq m_j$ implies that 
\begin{equation}
\sum_{k+l=j} (f_k\widehat{\delta}_l+ \widehat{\delta}_l f_k)(qb_i)=0 \text{ for all } i, j.
\nonumber \end{equation}
The same relations hold for all $a_i \in CF^*(M, H^{\tau_i}_t)$. So $f=(f_0,f_1, \hdots)$ defines a map of $S^1$-complexes $\widehat{CF^*}(M,\mathbf{H})$ and $\widehat{CF^*}(M,\mathbf{K})$. \\
\indent There is a natural filtration on $\widehat{CF^*}(M,\mathbf{H})$ defined by
\begin{equation}
F^k \widehat{CF^*}(M, \mathbf{H})=\bigoplus_{i=1}^{k-1}CF^*(M,H^{\tau_i}_t)[q]\oplus CF^*(M,H^{\tau_k}_t).
\nonumber \end{equation}
The natural inclusion 
\begin{equation}
\iota\colon CF^*(M,H^{\tau_k}_t) \rightarrow F^k\widehat{CF^*}(M,\mathbf{H})
\nonumber \end{equation}
induces a quasi-isomorphism. For $k \leq l$, there is a commutative diagram
\[
\xymatrixcolsep{3pc}\xymatrix{
CF^*(M,H^{\tau_k}_t) \ar@{^{(}->}[r]^-{\iota}
 \ar[d]^-{\kappa} & F^k \widehat{CF^*}(M,\mathbf{H}) \ar@{^{(}->}[d]^-{\iota} \\
CF^*(M,K^{\sigma_l}_t) \ar@{^{(}->}[r]^-{\iota} & F^l \widehat{CF^*}(M,\mathbf{K}) 
}
\]
up to a chain homotopy. This implies that
\begin{equation}
H_*(\widehat{CF^*}(M,\mathbf{H}), \widehat{\delta}_0) \cong \varinjlim  HF^*(M,H^{\tau_i}_t) \cong SH^*(M).
\nonumber \end{equation}
\begin{equation}
H_*(\widehat{CF^*}(M,\mathbf{K}), \widehat{\delta}_0) \cong \varinjlim HF^*(M,K^{\sigma_j}_t) \cong SH^*(M).
\nonumber \end{equation} 
As the restriction $(f_0|_{F^i})_*$ of the induced map $(f_0)_*$ in homology to $F^i \widehat{CF^*}(M,\mathbf{H})$ satisfies that
\begin{equation}
(f_0|_{F^i})_*\colon HF^*(M, H^{\tau_i}_t) \rightarrow HF^*(M, K^{\sigma_{j+m_j}}_t)
\nonumber \end{equation}
and $(f_0|_{F^i})_*$ is compatible with the direct system for all $i\geq 0$, the $S^1$-equivariant cochain map $f$ yields a quasi-isomorphism
\begin{equation}
(f_0)_*: \widehat{CF^*}(M, \mathbf{H})\rightarrow \widehat{CF^*}(M,\mathbf{K}).
\nonumber \end{equation} 
We can then conclude that there is an isomorphism 
$$\hp(M, \mathbf{H}) \cong \hp(M, \mathbf{K})$$ 
by appealing to Proposition \ref{prop:cycinv}.
\end{proof}

\begin{lem}\label{lem:homo}
Given admissible Hamiltonians $H_t$ and $K_t$ with $H_t \preceq K_t$, and continuation maps $\kappa=(\kappa_0, \kappa_1, \hdots)$ and $\kappa'=(\kappa_0', \kappa_1', \hdots)$ associated to two different monotone homotopies $(H_{s,t},J_{s,t})$ and $(H_{s,t}',J_{s,t}')$, there exists an $S^1$-equivariant homotopy operator $h=(h_0, h_1, \cdots)$ with 
$$h_i\colon CF^*(M,H_t) \rightarrow CF^{*-2i-1}(M, K_t)$$ 
such that the relation
\begin{equation} \label{eqn:hrel}
\kappa_n- \kappa'_n=\sum_{i+j=n} h_i \delta_j + \delta_j h_i
\end{equation}
is satisfied for all $n \geq 0$.
\end{lem}

\begin{proof}

Given two monotone homotopies $(H_{s,t}, J_{s,t})$ and $(H_{s,t}',J_{s,t}')$ between $(K_t,J_0)$ and $(H_t, J_1)$, we consider a family of admissible pairs $(H^r_{s,t}, J^r_{s,t})$ in $\mathscr{H}(M) \times \mathscr{J}(M)$ parametrized by $[0,1] \times \R$ such that 
\begin{equation}
H^r_{s,t} =
\begin{cases}
H_{s,t} \text{ if } r \in [0, \delta]\\
H_{s,t}'\text{ if } r \in [1-\delta,1],
\end{cases}
J^r_{s,t} =
\begin{cases}
J_{s,t} \text{ if } r \in [0, \delta]\\
J_{s,t}'\text{ if } r \in [1-\delta,1],
\end{cases}
\nonumber 
\end{equation}
for some $\delta >0$ sufficiently small. Using the extension procedure described in \eqref{eqn:ext}, one obtains the Floer data $(H_{N,s,t}^r, J_{N,s,t}^r)$ defined on $\widehat{M} \times S^{2N+1}$ for each $N$. Let $$\widetilde{\mathscr{M}}^h_i(\gamma_0, \gamma_1):=\widetilde{\mathscr{M}}^h_i(\gamma_0, \gamma_1,H_{N,s,t}^r, J_{N,s,t}^r, \widetilde{f}_{N}, \widetilde{g}_{N})$$ be the moduli space of triples $(r,u,z)$ where $r \in [0,1]$ and $u\colon \R \times S^1 \rightarrow \widehat{M}$, $z\colon \R \rightarrow S^{2N+1}$ are solutions to the system of equations
\begin{equation} \label{eqn:h_i}
\begin{cases}
\partial _s u+J^{r,z(s)}_{N,s,t}(u)(\partial _t u-X_{H_{N,s,t}^{r, z(s)}}(u))=0, \\
\dot{z}+\nabla \widetilde{f}_N(z)=0,
\end{cases}
\end{equation}
with asymptotic behaviors
\begin{equation}
\displaystyle\lim_{s \rightarrow - \infty}(u(s, \cdot), z(s)) \in  S^1 \cdot (\gamma_0, Z_i) ,\ \ \displaystyle\lim_{s \rightarrow \infty}(u(s, \cdot), z(s)) \in  S^1 \cdot (\gamma_1,Z_0).
\end{equation}
There is a free $S^1$-action on the moduli space $\widetilde{\mathscr{M}}^h_i(\gamma_0,\gamma_1)$, we denote by $\mathscr{M}^h_i(\gamma_0,\gamma_1)$ the quotient $\widetilde{\mathscr{M}}^h_i(\gamma_0,\gamma_1)/S^1$.
The linearization of  the first equation in \eqref{eqn:h_i} yields a linear map
\begin{equation}
D\colon \R \oplus T_z(S^1\cdot W^u(Z_i)) \oplus W^{1,p}(\R \times S^1, u^*(T\widehat{M})) \rightarrow L^p(\R \times S^1, u^*(T\widehat{M})),\nonumber
\end{equation}
where $p>2$ and $W^u(Z_i)$ is the unstable or descending manifold of the critical point $Z_i$ of index $2i$ on $\C P^{2N}$. The Floer data $(H^r_{N,s,t}, J^r_{N,s,t})$ is regular if the linear map $D$ is surjective. This is equivalent to surjectivity of  the linear map
\begin{equation}
T_r[0,1]  \oplus T_z(S^1\cdot W^u(Z_i)) \rightarrow \coker(D_u),
\nonumber \end{equation} 
where $D_u$ is the linearized operator associated to an element $(r,u,z)$ in $\widetilde{\mathscr{M}}^h_i(\gamma_0,\gamma_1)$. For regular Floer data $(H^r_{N,s,t},J^r_{N,s,t})$, the dimension of the moduli space $\mathscr{M}^h_i(\gamma_0, \gamma_1)$ is $|\gamma_0|-|\gamma_1|+2i+1$. Similar to the case of BV operator, there is an short exact sequence
\begin{equation}
 \resizebox{0.95\hsize}{!}{$ 0 \rightarrow T_u\widetilde{\mathscr{M}}^h_i(x_0, x_1) \rightarrow T_r[0,1]  \oplus T_z(S^1\cdot W^u(Z_i)) \oplus \ker(D_u) \rightarrow \coker(D_u) \rightarrow 0,
$} 
\end{equation}
which induces an isomorphism of determinant lines
\begin{equation} \label{eqn:orh1}
 \resizebox{0.98\hsize}{!}{$\det(T_u\mathscr{M}^h_i(\gamma_0, \gamma_1)) \otimes \det(T_{\theta}S^1) \cong \det(D_u) \otimes \det(T_r[0,1]) \otimes \det(T_{\theta}S^1) \otimes \det(T_z\cdot W^u(Z_i)).
$} 
\end{equation}
There is another isomorphism given by gluing theory in section \ref{subsec:orient}
\begin{equation} \label{eqn:orh2}
|\det(T_u\mathscr{M}^h_i(\gamma_0, \gamma_1))| \otimes o_{\gamma_1} \cong o_{\gamma_0}.
\end{equation}
We fix the orientation of $T_r[0,1]$ to be $\R \langle \frac{\partial}{\partial_r} \rangle$ and choose a coherent orientation for each unstable manifold $W^u(Z_i)$ of a critical point $Z_i$ on $\C P^N$. For $|\gamma_0|=|\gamma_1|-2i-1$, the moduli space $\mathscr{M}^h_i(\gamma_0, \gamma_1)$ is zero-dimensional and $T_u\mathscr{M}^h_i(\gamma_0, \gamma_1)$ is canonically trivial. By comparing \eqref{eqn:orh1} and \eqref{eqn:orh2}, one obtains an isomorphism
\begin{equation}
h_{i,u}\colon o_{\gamma_1} \rightarrow o_{\gamma_0}
\nonumber \end{equation}
for each $u$ in $\mathscr{M}^h_i(\gamma_0, \gamma_1)$ with $|\gamma_0|=|\gamma_1|-2i-1$.
We can then define operations $h_i\colon CF^*(M,H_t) \rightarrow CF^{*-2i-1}(M,K_t)$ by
\begin{equation}
h_i|_{o_{\gamma_1}}=\bigoplus_{|\gamma_0|=|\gamma_1|-2i-1}\sum_{u \in \mathscr{M}^h_i(\gamma_0, \gamma_1)} h_{i,u}.
\nonumber \end{equation}
After taking $N \rightarrow \infty$, this defines the operation $h_i$ for all $i\geq 0$.
By applying the maximum principle to the solutions to the parametrized Floer equation \eqref{eqn:h_i} with admissible choice of $(H^r_{s,t}, J^r_{s,t}),$ proved in \cite[Lemma 19.1]{Ritter}, we obtain a compactification $\overline{\mathscr{M}^h_n}(\gamma_0, \gamma_1)$ of $\mathscr{M}^h_n(\gamma_0, \gamma_1)$ given by the usual Gromov compactness Theorem for all $n \geq 0$. To see the relation \eqref{eqn:hrel}, one needs to exam the boundary of the 1-dimensional manifold $\overline{\mathscr{M}^h_n}(\gamma_0, \gamma_1)$ with $|\gamma_0|=|\gamma_1|-2n$ for each $n$. Let $N$ be an integer such that $N \geq n$. Appealing to the usual gluing argument, we have that the boundary is explicitly given by
\begin{eqnarray}\label{eqn:bdy}
&& \partial \overline{\mathscr{M}^h_n}(\gamma_0, \gamma_1) = \mathscr{M}^{\kappa}_n(\gamma_0,\gamma_1, H_{N,s,t}, J_{N,s,t}) \cup \mathscr{M}^{\kappa}_n(\gamma_0,\gamma_1, H_{N,s,t}', J_{N,s,t}')  \\
&&  \ \ \  \cup  \bigcup_{i+j=n} \Big( \bigcup_{\gamma}  \mathscr{M}_i(\gamma_0,\gamma) \times \mathscr{M}^h_j(\gamma,\gamma_1) \cup \bigcup_{\gamma'} \mathscr{M}^h_i(\gamma_0,\gamma') \times \mathscr{M}_j(\gamma',\gamma_1) \Big) \nonumber,
\end{eqnarray}
where the equality is to be understood with appropriate orientations. Therefore, we can conclude that $h=(h_0,h_1,\hdots)$ is an $S^1$-equivariant homotopy operator that satisfies
\begin{equation} 
\kappa_n- \kappa'_n=\sum_{i+j=n} h_i \delta_j + \delta_j h_i.
\nonumber \end{equation}
This completes the proof.
\end{proof}

Given Proposition \ref{prop:indepham}, we can promote our notation of periodic symplectic cohomology to $\hp(M)$. This implies that $\hp(M)$ is an invariant of the completion $M$ up to Liouville isomorphism.

\begin{rmk}
Proposition \ref{prop:indepham} shows that the periodic symplectic cohomologies $\hp(M)$ is independent of the choice of admissible Hamiltonians, so in practice it is convenient to define $\hp(M)$ using autonomous Hamiltonians $H^{\tau_i}$ with specific perturbations $\mathbf{H}=\{H^{\tau_i}_t\}$ described in section \ref{sec:concx}. Namely, we choose a cofinal system of Hamiltonians $H^{\tau_i}_t$ and define 
\begin{equation}
\widehat{CF^*}(M):=\bigoplus_{i=1}^{\infty} CF^*(M, H^{\tau_i}_t)[q].
\nonumber \end{equation}
We equip $\widehat{CF^*}(M)((u))$ with the differential $\widehat{\delta}^{S^1}=\widehat{\delta}_0+u\widehat{\delta}_1+u^2\widehat{\delta}_2+\cdots$. The map $\widehat{\delta}_j$ is defined by
\begin{equation}
\widehat{\delta}_j(a+ qb) = (-1)^{deg(a)}\delta_j(a) + (-1)^{deg(b)}(q\delta_j(b) + \kappa_j(b)),
\nonumber \end{equation}
where $\kappa_j\colon CF^*(M, H^{\tau_i}_t) \rightarrow CF^{*-2j}(M, H^{\tau_{i+1}}_t)$ is now taken to be the action-preserving map $\kappa_j$. There is a filtered $S^1$-structure on $\widehat{CF^*}(M)$, and consequently a filtration on $\hp(M)$ by action. It will be convenient to use this formulation to compute $\hp(M)$ in the next section.
\end{rmk}

\section{Computations for the disc and the annulus}\label{sec:comp}
\subsection{$\hpl(D^2)$ for the disk}
\indent We consider $\C$ with a primitive of $\widehat{\omega}=dx \wedge dy$ given by $\widehat{\theta}=\frac{1}{2}(xdy-ydx)$. If we denote $(r, \theta)$ as the polar coordinates on $\C$, the Liouville vector field is given by $Z= \frac{r}{2}\partial_r$. Since the Liouville vector field is of the form $R\partial_R$ using the cylindrical coordinate on $[1, \infty) \times S^1$, this implies that $r^2=R$ so that $R\partial_R=\frac{r}{2} \partial_r$.
Similarly, the Reeb vector field $\mathcal{R}_{\alpha}$ on the contact manifold $S^1=\{z \in \C \mathbin{|} |z|=1\}$ is given by $2\partial_{\theta}$. The periods of Reeb orbits on $S^1$ are $k\pi$ for $k \in \Z_+$.  We choose a cofinal system of Hamiltonians $H^{k\pi+1}\colon \C \rightarrow \R, k\geq 0$ defined in section \ref{sec:concx} such that $H^{k\pi+1}|_{\Int(M)}$ is a negative $C^2$ small Morse function and
\begin{equation} \label{eqn:hkpi}
H^{k\pi+1}|_{[1, \infty) \times \partial M} =
\begin{cases}
\frac{(R-1)^2}{2}, & \text{if } R-1 \in [0, k\pi+1];\\
(k\pi +1)(R-1)+C, & \text{if } R-1 \in [k\pi+1, \infty).
\end{cases}
\end{equation}
The contact manifold $S^1$ consists only good Reeb orbits which are maps $\gamma^i \colon S^1 \rightarrow S^1 \subset \C$ of degree $i$ for all $i \in \Z_+$. The Reeb orbit $\gamma^i$ of degree $i$ corresponds to the $1$-periodic orbit $\gamma_i$ when $R-1=i\pi$.
One can perturb $H^{k\pi+1}$ locally in a neighborhood of each non-constant $1$-periodic orbit $\gamma_i$. Using the explicit perturbation 
$$H^{k\pi+1}_t:=H^{k\pi+1}+ \sum_{i \geq 1}\epsilon_{k\pi+1} f_{\gamma_i},$$ 
each $\gamma_i$ gives rise to two nondegenerate $1$-periodic orbits $\widehat{\gamma}_i, \widecheck{\gamma}_i$ of $H^{k\pi+1}_t$. Their indices can be computed as follows
\begin{equation}
|\widecheck{\gamma}_i|=-2i, \ |\widehat{\gamma}_i|=-2i+1, i \geq 1.
\nonumber \end{equation}
Also, the origin $x_0:=(0,0)$ in $\C$ is a critical point of $H^{k\pi+1}_t$ for all $k \geq 0$. The index of $x_0$ is the Morse index, that is $|x_0|=0$.
For simplicity, we denote the generator of $CF^*(D^2, H^{k\pi+1}_t)$ of index $i$ by $x_i$. Then we have
\begin{equation}
CF^*(D^2, H^{k\pi+1}_t)=\Z\langle x_0, x_{-1}, \hdots x_{-2k+1}, x_{-2k} \rangle.
\nonumber \end{equation}
Since $\delta_i$ preserve the action filtration on $CF^*(D^2, H^{k\pi+1}_t)$ defined in \eqref{defn:af}, we have that $\delta_i=0$ for all $i\geq 2$. Then Proposition \ref{prop:1} implies that for $j=1, \hdots, k$
\begin{equation}
d(x_{-2j})=0, \ \ \ \Delta(x_{-2j+1})=jx_{-2j}.
\nonumber \end{equation}
Also, we see that $\Delta(x_{-2j})=0$ as $\Delta$ preserves the action filtration on $CF^*(M, H^{k\pi+1}_t)$ in \eqref{defn:af}. There is precisely one solution to the Floer equation with asymptotic conditions $x_{-2j+1}$ and $x_{-2j+2}$ at $\pm \infty$. This implies that 
$d(x_{-2j+1})=x_{-2j+2}$. Thus we have
\begin{eqnarray}
&& CF*(D^2, H^{k\pi+1}_t)((u)) \cong \Z((u))\langle x_0, x_{-1}, \hdots x_{-2k+1}, x_{-2k} \rangle,\nonumber \\
&& \delta^{S^1}(x_{-2j})=0,\ \ \delta^{S^1}(x_{-2j+1})=x_{-2j+2}+ujx_{-2j} \text{ for } j=1, \hdots ,k. \nonumber
\end{eqnarray}
The homology of this cochain complex is 
\begin{equation}
HP^*_{S^1}(D^2, H^{k\pi+1}_t) \cong \Z((u)) \langle [x_{-2k}] \rangle.
\nonumber \end{equation}
To determine the homomorphism 
\begin{equation}
(\kappa_{k, k+1})_* \colon HP^*_{S^1}(D^2, H^{k\pi+1}_t) \rightarrow HP^*_{S^1}(D^2, H^{(k+1)\pi+1}_t),
\nonumber \end{equation} 
we consider the continuation map $\kappa_{k,k+1}$ induced by the monotone homotopy between $H^{k\pi+1}_t$ and $H^{(k+1)\pi+1}_t$. As the continuation map $\kappa_{k,k+1}$ respects action filtration, the cochain complex $CF^*(D^2, H^{k\pi+1}_t)$ have the same generators as those of $CF^*(D^2, H^{k\pi+1}_t)$ which we also denote by $x_i$ for $i=0, \hdots, -2k$. For 1-periodic orbit $\widehat{\gamma}$ of $H^{k\pi+1}_t$ and $H^{k\pi+1}_t$, there is precisely one regular solution to the continuation equation given by 
\begin{equation}
u\colon \R \times S^1 \rightarrow N(\gamma),\ \ u(s,t) \mapsto \widehat{\gamma}(t).
\nonumber \end{equation}
So the continuation map $\eta_0$ is the identity. As the values of the Hamiltonians $H^{k\pi+1}_t$ and $H^{(k+1)\pi+1}_t$ agree on $M \cup [1, k\pi+2] \times \partial M$, the continuation map 
$$\kappa_0'\colon CF^*(D^2, H^{k\pi+1}_t) \rightarrow CF^*(D^2, H^{(k+1)\pi+1}_t)$$
is given by the natural inclusion. This implies that $\kappa_{k,k+1}=\kappa_0$ is also the natural inclusion
\begin{eqnarray}
\kappa_{k,k+1}\colon  CF^*(D^2, H^{k\pi+1}_t)((u)) &\rightarrow & CF^*(D^2, H^{(k+1)\pi+1}_t)((u)),\nonumber\\
 u^j x_i &\mapsto & u^j x_i. \nonumber
\end{eqnarray}
After passing to homology, we have that
\begin{equation}
(\kappa_{k,k+1})_*([x_{-2k}])=-(k+1)u[x_{-2k-2}].
\nonumber \end{equation}
The direct limit with respect to 
\begin{equation}
(\kappa_{k,k+1})_*\colon  HP^*_{S^1}(D^2, H^{k\pi+1}_t) \rightarrow HP^*_{S^1}(D^2, H^{(k+1)\pi+1}_t)
\nonumber \end{equation} 
is explicitly given by

\begin{equation}
\xymatrixcolsep{4.5pc}\xymatrix{
\Z((u))\langle [x_0] \rangle \ar[r]^-{[x_0] \mapsto -u[x_{-2}]} & \Z((u))\langle [x_{-2}] \rangle \ar[r]^-{[x_{-2}] \mapsto -2u[x_{-4}]} &\Z((u))\langle [x_{-4}] \rangle \ar[r]^-{[x_{-4}] \mapsto -3u[x_{-6}]} &\cdots.
} \nonumber
\nonumber 
\end{equation}

Under the homomorphism, 
\begin{equation}
\psi_k \colon \Z((u))\langle [x_{-2k}] \rangle \rightarrow \Q((u))\langle x_0 \rangle, \ \ a \mapsto a/(-1)^kk!u^k,
\nonumber \end{equation}
we conclude that
\begin{equation}
\hpl(D^2)\colon =\varinjlim HP^*_{S^1}(D^2, H_k)\cong\Q((u))\langle [x_0] \rangle \cong \Q((u)),
\nonumber \end{equation}
which verifies Theorem \ref{thm:local} in the case of $D^2 \subset \C$.

\subsection{$\hp(D^2)$ for the disk}
Let $\{H^{k\pi +1}_t\}_{k\geq 0}$ be the cofinal system of Hamiltonian defined in \eqref{eqn:hkpi}. We apply the telescope construction to $\{CF^*(D^2, H^{k\pi+1}_t)\}$ and obtain that
\begin{equation}
\widehat{CF^*}(D^2)=\bigoplus_{k=1}^{\infty}CF^*(D^2, H^{k\pi_1}_t)[q].
\nonumber \end{equation}
As shown in section \ref{sec:cpsh}, the sequence of operations $(\widehat{\delta}_0, \widehat{\delta}_1, \hdots)$ gives rise to a $S^1$-structure on $\widehat{CF^*}(D^2)$.
One notices that $H_*(\widehat{CF^*}(D^2), \widehat{\delta}_0)=SH^*(D^2)$. By Corollary \ref{cor:van}, we conclude that $\hp(D^2)=0$ as $SH^*(D^2)=0$.

\subsection{$\hpl(A)$ for the annulus}
Let $A\cong[-1, 1] \times S^1$ be the annulus with coordinate $(r,\theta)$ in $\R \times S^1 \cong T^*S^1$. There is a biholomorphism 
\begin{equation}
\iota\colon T^*S^1 \rightarrow \C^*, (r, \theta) \mapsto e^{r+i\theta},
\nonumber \end{equation}
which embeds $T^*S^1$ into $\C$. We equip $A$ with the Liouville structure induced by $\iota^* \widehat{\theta}$, where $\widehat{\theta}$ is the standard Liouville form $\widehat{\theta}=\frac{1}{2}(xdy-ydx)$ on $\C$. We consider a cofinal system of autonomous Hamiltonians $H^{i\pi+1}\colon T^*S^1 \rightarrow \R$ with $i \geq 0$ on each cylindrical end $[1,\infty) \times S^1$ defined by
\begin{equation} \label{eqn:annham}
H^{i\pi+1}=
\begin{cases}
\frac{(R-1)^2}{2}, & \text{ if } R-1 \in [0, i\pi+1] ;\\
(i\pi +1)(R-1) +C, & \text{ if } R-1 \in [i\pi+1, \infty).
\end{cases}
\end{equation}
and $H^{i\pi+1}|_{\Int(A)}$ is a negative $C^2$-small Morse function, which has a minimum at $(0,0)$ and a maximum at $(0, \frac{1}{2})$ in $ A=[-1,1] \times S^1$. We denote by $\widehat{\gamma}_0, \widecheck{\gamma}_0$ the minimum and the maximum of $H^{i\pi+1}|_{\Int(\iota(A))}$, respectively. For each $H^{i\pi+1}$, there are $2i$ non-constant $1$-periodic orbits denoted by $\gamma_{k}$ and $\gamma_{-k}$ with $k=1, \hdots, i$, where $\gamma_{k}$ and $\gamma_{-k}$ correspond to Reeb orbits of multiplicity $k$ on the two components of $\partial A =S^1 \sqcup S^1$ with opposite orientations. After perturbing $H^{i\pi+1}$ to $H^{i\pi+1}_t$ as before, we obtain that two nondegenerate $1$-periodic orbits $\widehat{\gamma}_k, \widecheck{\gamma}_k$ for each $\gamma_k$ with indices 
\begin{equation}
|\widecheck{\gamma}_k|=0 \text{ and } |\widehat{\gamma}_k|=1\text{ for } k=-i,\hdots, i.
\nonumber \end{equation}
with respect to the trivialization of the tangent space of $T^*S^1$ given by $\iota^*\alpha$, where $\alpha=\frac{dz}{z}$ is the standard holomorphic form  defined on $\C^*$. The Floer cochain complex of $H^{i\pi+1}_t$ is given by
\begin{equation}
CF^*(A, H^{i\pi+1}_t)=\Z \langle \widehat{\gamma}_{-i}, \widecheck{\gamma}_{-i}, \widehat{\gamma}_{-i+1}, \widecheck{\gamma}_{-i+1}\hdots, \widehat{\gamma}_{i-1}, \widecheck{\gamma}_{i-1}, \widehat{\gamma}_i, \widecheck{\gamma}_i \rangle.
\nonumber \end{equation}
For $j,k \neq 0$, the generators $\widecheck{\gamma}_j$ and $\widehat{\gamma}_k$ are in different free homotopy class if $j \neq k$. So we have that $d(\widecheck{\gamma}_k)=0$ if $j \neq k$. Together with that fact 
$$d(\widecheck{\gamma}_k)=0\cdot \widehat{\gamma}_k=0$$
by Proposition \ref{prop:1}, one obtains $\delta^{S^1}(\widecheck{\gamma}_k)=0$ for all $k \neq 0$. Moreover, one has
\begin{equation}
\delta^{S^1}(\widehat{\gamma}_k)=0+ku\widecheck{\gamma}_k=ku\widecheck{\gamma}_k \text{ for all }k \neq 0,
\nonumber \end{equation}
 as $\gamma_k$ corresponds to a good orbit of multiplicity $k$ on $\partial A =S^1 \sqcup S^1$. For $k=0$, we have $\delta^{S^1}(\widehat{\gamma}_0)=0$ by degree reasons, and $\delta^{S^1}(\widecheck{\gamma}_0)=0$ since the Floer differential $d$ agrees with the Morse differential between $\widehat{\gamma}_0$ and $\widecheck{\gamma}_0$ and there is no regular solution to the BV operator in this case. This implies that the homology of $(CP^*(A,H^{i\pi+1}_t), \delta^{S^1})$ is given by
\begin{eqnarray}
&& \ \ HP_{S^1}^{\mathrm{even}}(M, H^{i\pi+1}_t) \cong \bigoplus_{k \in [-i,i]\cap \Z } \Z/k\Z((u)),\nonumber \\
&& \ \  HP_{S^1}^{\mathrm{odd}}(M,H^{i\pi+1}_t) \cong \Z((u)),\nonumber \\
&& \ \ HP^*_{S^1}(M, H^{i\pi+1}_t)=\Z ((u)) \oplus \bigoplus_{k \in [-i,i]\cap \Z} \Z/k\Z((u)).\nonumber
\end{eqnarray}
For $H^{i\pi+1}_t \preceq H^{(i+1)\pi+1}_t$, higher order continuation maps $\kappa_i$ vanish for all $i \geq 1$. Similar to the case of $D^2$ in $\C$, the continuation map 
\begin{equation}
\kappa_{i,i+1}\colon CF^*(A, H^{i\pi+1}_t)((u)) \rightarrow CF^*(A, H^{(i+1)\pi+1}_t)((u))
\nonumber \end{equation}
is induced by $\kappa_0 \colon CF^*(M, H^{i\pi+1}_t) \rightarrow CF^*(M, H^{(i+1)\pi+1}_t)$ given by $\kappa_0$. With respect to 
\begin{equation}
(\kappa_{i,i+1})_*\colon HP^*_{S^1}(A, H^{i\pi+1}_t) \rightarrow HP^*_{S^1}(A, H^{(i+1)\pi+1}_t),
\nonumber \end{equation}
we conclude that 
\begin{equation}
\hpl(A)=\varinjlim HP^*_{S^1}(A, H^{i\pi+1}_t)=
\Z((u)) \oplus \bigoplus_{k \in \Z} \Z/k\Z((u)).
\nonumber \end{equation} 
This confirms that $\hpl(A, \Q) \cong H^*(A,\Q)((u))\cong H^*(S^1,\Q)((u))$.

\subsection{$\hp(A)$ for the annulus}
Let $H^{i\pi+1}_t$ be the sequence of admissible Hamiltonian defined in \eqref{eqn:annham}. The telescope construction of $\{CF^*(A, H^{i\pi+1}_t)\}_{i \geq 0}$ is given by 
\[
\xymatrixcolsep{.1pc}\xymatrix{
 \Z \langle \widehat{\gamma}_0, \widecheck{\gamma}_0 \rangle  \ar@(ul,ur)^{d} && \Z \langle  \widehat{\gamma}_{-1}, \widecheck{\gamma}_{-1},\widehat{\gamma}_0, \widecheck{\gamma}_0, \widehat{\gamma}_1, \widecheck{\gamma}_1 \rangle   \ar@(ul,ur)^{d} &&   \Z \langle  \widehat{\gamma}_{-2}, \widecheck{\gamma}_{-2},  \widehat{\gamma}_{-1}, \widecheck{\gamma}_{-1},\widehat{\gamma}_0, \widecheck{\gamma}_0, \widehat{\gamma}_1, \widecheck{\gamma}_1, \widehat{\gamma}_2, \widecheck{\gamma}_2  \rangle   \ar@(ul,ur)^{d}  \\
q \Z \langle \widehat{\gamma}_0, \widecheck{\gamma}_0 \rangle  \ar[u]^{id} \ar[urr]^{\kappa_0} \ar@(dl,dr)_{d} &&  q\Z \langle  \widehat{\gamma}_{-1}, \widecheck{\gamma}_{-1},\widehat{\gamma}_0, \widecheck{\gamma}_0, \widehat{\gamma}_1, \widecheck{\gamma}_1 \rangle \ar[u]^{id} \ar[urr]^{\kappa_0}  \ar@(dl,dr)_{d} &&  \cdots \cdots
}
\]
By degree reasons, $\delta_i=0$ for $i \geq 2$ and $\kappa_j=0$ for $j \geq 1$ in the definition of $\widehat{\delta}^{S^1}$.
This implies that 
\begin{eqnarray}
&& \widehat{\delta}^{S^1}(\widehat{\gamma}_0)=\widehat{\delta}^{S^1}(\widecheck{\gamma}_0)=0,  \nonumber \\
&& \widehat{\delta}^{S^1}(q\widehat{\gamma}_0)=-(\kappa_0(\widehat{\gamma}_0)- \widehat{\gamma}_0),\ \   \nonumber\\
&& \widehat{\delta}^{S^1}(q\widecheck{\gamma}_0)=\kappa_0(\widecheck{\gamma}_0)- \widecheck{\gamma}_0, \nonumber \\
&& \widehat{\delta}^{S^1}(\widecheck{\gamma}_k)=0, \ \   \widehat{\delta}^{S^1}(\widehat{\gamma}_k)=ku\widecheck{\gamma}_k,\nonumber \\
&&\widehat{\delta}^{S^1}(q\widecheck{\gamma}_k)=\kappa_0(\widecheck{\gamma}_k)-\widecheck{\gamma}_k\text{ for all } k \in \Z, \nonumber\\
&& \widehat{\delta}^{S^1}(q\widehat{\gamma}_k)=-(qku\widecheck{\gamma}_k+\kappa_0(\widehat{\gamma}_k)-\widehat{\gamma}_k) \text{ for all } k \in \Z.\nonumber
\end{eqnarray}  
Therefore, we conclude that
\begin{equation}
\hp(A)= \big( \Z \oplus \bigoplus_{k \in \Z} \Z/k\Z \big)((u)).
\nonumber \end{equation}

\end{document}